\documentclass{amsart}
\usepackage{amssymb}
\usepackage{braket}
\usepackage{mathrsfs}
\usepackage{ifthen}
\usepackage{here}
\usepackage{todonotes}
\usepackage{tikz}
\usetikzlibrary{patterns}
\usepackage{comment} 
\usepackage{mleftright}
\usepackage{hyperref}
\usepackage{cleveref}
 \usepackage[all]{xy}
 \usepackage{amscd}
 \usepackage{amsrefs}
 \usepackage{color}
 \usepackage{enumitem}
\newlist{steps}{enumerate}{1}
\setlist[steps, 1]{label = Step \arabic*:}
\usepackage[abs]{overpic}
\usepackage[latin1]{inputenc}

\makeatletter
\DeclareRobustCommand\widecheck[1]{{\mathpalette\@widecheck{#1}}}
\def\@widecheck#1#2{%
   \setbox\z@\hbox{\m@th$#1#2$}%
   \setbox\tw@\hbox{\m@th$#1%
      {%
         \vrule\@width\z@\@height\ht\z@
         \vrule\@height\z@\@width\wd\z@}$}%
   \dp\tw@-\ht\z@
   \@tempdima\ht\z@ \advance\@tempdima2\ht\tw@ \divide\@tempdima\thr@@
   \setbox\tw@\hbox{%
      \raise\@tempdima\hbox{\scalebox{1}[-1]{\lower\@tempdima\box\tw@}}}%
   {\ooalign{\box\tw@ \cr \box\z@}}}
\makeatother

\theoremstyle{plain}
\newtheorem{thm}{Theorem}[section]
\crefname{thm}{Theorem}{Theorems}
\Crefname{thm}{Theorem}{Theorems}
\newtheorem{prop}[thm]{Proposition}
\crefname{prop}{Proposition}{Propositions}
\Crefname{prop}{Proposition}{Propositions}
\newtheorem{lem}[thm]{Lemma}
\crefname{lem}{Lemma}{Lemmas}
\Crefname{lem}{Lemma}{Lemmas}
\newtheorem{cor}[thm]{Corollary}
\crefname{cor}{Corollary}{Corollaries}
\Crefname{cor}{Corollary}{Corollaries}

\crefname{claim}{Claim}{Claims}
\Crefname{claim}{Claim}{Claims}

\crefname{property}{Property}{Properties}
\Crefname{property}{Property}{Properties}

\crefname{problem}{Problem}{Problems}
\Crefname{problem}{Problem}{Problems}
\newtheorem{conjecture}[thm]{Conjecture}
\crefname{conjecture}{Conjecture}{Conjecture}
\Crefname{conjecture}{Conjecture}{Conjecture}

\theoremstyle{definition}
\newtheorem{defn}[thm]{Definition}
\crefname{defn}{Definition}{Definitions}
\Crefname{defn}{Definition}{Definitions}

\crefname{notation}{Notation}{Notations}
\Crefname{notation}{Notation}{Notations}

\crefname{convention}{Convention}{Conventions}
\Crefname{convention}{Convention}{Conventions}

\crefname{cond}{Condition}{Conditions}
\Crefname{cond}{Condition}{Conditions}

\crefname{assum}{Assumption}{Assumptions}
\Crefname{assum}{Assumption}{Assumptions}

\crefname{conj}{Conjecture}{Conjectures}
\Crefname{conj}{Conjecture}{Conjectures}
\newtheorem{claim1}[thm]{Claim}
\crefname{claim1}{Claim}{Claims}
\Crefname{claim1}{Claim}{Claims}
\Crefname{ques}{Question}{Question}
\newtheorem{ques}[thm]{Question}
\crefname{que}{Question}{Question}
\Crefname{que}{Question}{Question}

\theoremstyle{remark}
\newtheorem{rem}[thm]{Remark}
\crefname{rem}{Remark}{Remarks}
\Crefname{rem}{Remark}{Remarks}

\crefname{ex}{Example}{Examples}
\Crefname{ex}{Example}{Examples}

\crefname{section}{Section}{Sections}
\Crefname{section}{Section}{Sections}
\crefname{subsection}{Subsection}{Subsections}
\Crefname{subsection}{Subsection}{Subsections}
\crefname{figure}{Figure}{Figures}
\Crefname{figure}{Figure}{Figures}

\newtheorem*{acknowledgement}{Acknowledgement}

\newcommand{\spinc}{\text{Spin}^c}
\newcommand{\Z}{\mathbb{Z}}

\newcommand{\Q}{\mathbb{Q}}
\newcommand{\CP}{\mathbb{CP}}

\newcommand{\spc}{\mathrm{Spin}^c}

\newcommand{\frakm}{\mathfrak{m}}

\newcommand{\fraks}{\mathfrak{s}}
\newcommand{\frakt}{\mathfrak{t}}

\newcommand{\Met}{\mathrm{Met}}

\newcommand{\Diff}{\mathrm{Diff}}
\newcommand{\Homeo}{\mathrm{Homeo}}
\newcommand{\Aut}{\mathrm{Aut}}

\newcommand{\diag}{\mathrm{diag}}

\newcommand{\del}{\partial}

\newcommand{\Ker}{\mathop{\mathrm{Ker}}\nolimits}

\newcommand{\Hom}{\mathop{\mathrm{Hom}}\nolimits}

\newcommand{\id}{\mathrm{id}}

\newcommand{\B}{\mathcal B}
\newcommand{\G}{\mathcal G}

\newcommand{\C}{\mathbb{C}}
\newcommand{\s}{\mathfrak{s}}

\def\om{\omega}
\def\Om{\Omega}

\def\Spinc{\text{Spin}^c}

\newcommand{\R}{\mathbb R}

\def\ker{\operatorname{Ker}}
\def\cok{\operatorname{Cok}}

\def\det{\operatorname{det}}

\def\dim{\operatorname{dim}}

\def\Hom{\operatorname{Hom}}

\def\Con{\mathcal{C}}

\def\id{\operatorname{Id}}

\newcommand{\mbar}[1]{{\ooalign{\hfil#1\hfil\crcr\raise.167ex\hbox{--}}}}

\def\cV{\mathcal{V}}

\def\Aut{\operatorname{Aut}}

\def\wt{\widetilde}

\newcommand{\HMf}{\widehat{\mathit{HM}}}
     \RequirePackage{rotating}                   
    \def\HMt{%
       \setbox0=\hbox{$\widehat{\mathit{HM}}$}
       \setbox1=\hbox{$\mathit{HM}$}
       \dimen0=1.1\ht0
       \advance\dimen0 by 1.17\ht1
       \smash{\mskip2mu\raise\dimen0\rlap{%
          \begin{turn}{180}
              {$\widehat{\phantom{\mathit{HM}}}$}
           \end{turn}} \mskip-2mu    
                \mathit{HM}
                    }{\vphantom{\widehat{\mathit{HM}}}}{}}

\title[Exotic diffeomorphisms and embeddings]{Diffeomorphisms of 4-manifolds with boundary and exotic embeddings}

\author{Nobuo Iida}
\address{Department of Mathematics, Tokyo Institute of Technology, 2-12-1 Ookayama, Meguro-ku, Tokyo 152-8551, Japan}
\email{iida.n.ad@m.titech.ac.jp}

\author{Hokuto Konno}
\address{Graduate School of Mathematical Sciences, the University of Tokyo, 3-8-1 Komaba, Meguro, Tokyo 153-8914, Japan \\and\\
RIKEN iTHEMS, Wako, Saitama 351-0198, Japan}
\email{konno@ms.u-tokyo.ac.jp}

\author{Anubhav Mukherjee}
\address{Department of Mathematics \\ Princeton University \\  Princeton  \\ New Jersey \\ USA}
\email{anubhavmaths@princeton.edu}

\author{Masaki Taniguchi}
\address{2-1 Hirosawa, Wako, Saitama 351-0198, Japan}
\email{masaki.taniguchi@riken.jp}

\date{}
\begin{document}

\begin{abstract}
We define family versions of the invariant of 4-manifolds with contact boundary due to Kronheimer and Mrowka and use these to detect exotic diffeomorphisms of 4-manifolds with boundary. Further, we show the existence of the first example of exotic 3-spheres in a smooth closed 4-manifold with diffeomorphic complements. 

%exotic codimension-2 submanifolds and exotic codimension-1 submanifolds in 4-manifolds. 
\end{abstract}

\maketitle

%\tableofcontents

\section{Introduction}

Let $W$ be a smooth compact 4-manifold with boundary.
Denote by $\operatorname{Homeo}(W)$ and $\operatorname{Diff}(W)$ the groups of homeomorphisms and diffeomorphisms of $W$, respectively; and by $\operatorname{Homeo} (W, \partial )$ and $\operatorname{Diff} (W, \partial )$ the groups of homeomorphisms and diffeomorphisms fixing boundary pointwise, respectively.
In this paper, we study a comparison between the mapping class groups arising from these groups, through the maps
\begin{align}
\label{eq: main target}
\pi_0(\operatorname{Diff} (W))
\to \pi_0(\operatorname{Homeo} (W)), \quad 
\pi_0(\operatorname{Diff} (W, \partial ))
\to \pi_0(\operatorname{Homeo} (W, \partial ))
\end{align}
induced from the natural inclusions.
A non-zero element of the kernels of maps \eqref{eq: main target} may be called an {\it exotic diffeomorphism} of $W$.
Such a diffeomorphism is topologically isotopic to the identity, but not smoothly so.
%We shall consider the kernels of both maps in \eqref{eq: main target}.

The first example of exotic diffeomorphisms in dimension 4 was given by Ruberman~\cite{Rub98}.
This example was detected by an invariant for diffeomorphisms based on Yang--Mills gauge theory for families.
Later, based on Seiberg--Witten theory for families, Baraglia and the second author \cite{BK20}, Kronheimer and Mrowka \cite{KM20}, and J.~Lin \cite{JLin20} gave other examples.
However, all of these examples are given for closed 4-manifolds.
In this paper, we shall present exotic diffeomorphisms of 4-manifolds with non-trivial (i.e. not $S^3$) boundary.
Some of our main results are formulated as an attempt at solving the following conjecture that we propose:

\begin{conjecture}
\label{main conj}
 Let $Y$ be a closed, connected, oriented 3-manifold. Then there exists a compact, oriented smooth 4-manifold $W$ with $\partial W=Y$ such that the natural map 
 \[
\pi_0 (\operatorname{Diff} (W, \partial )  )\to \pi_0 ( \operatorname{Homeo} (W, \partial )  )
\]
is not injective, i.e. there exists an exotic diffeomorphism of $W$ relative to the boundary.
More strongly,
\[
\pi_0 (\operatorname{Diff} (W)  )\to \pi_0 ( \operatorname{Homeo} (W)  )
\]
is not injective.
\end{conjecture}
The problem of finding exotic smooth structures on $W$ was considered in \cite[Question 1.4]{Ya12} and \cite[Question 1]{EMM20} and finding a non-trivial element in the kernel of the above maps is related to finding a non-trivial loop in the space of smooth structures on $W$, so this conjecture can be seen as a family version of the work in \cite[Question 1.4]{Ya12} and \cite[Question 1]{EMM20}.
%There are some 
Three major difficulties in solving this conjecture are: %The two big issues are:
\begin{itemize}
    \item What kind of gauge-theoretic invariants of diffeomorphisms in the case of a 4-manifold with boundary should be considered?
    One possibility might be to use families Floer theoretic relative invariants such as those in \cite{SS21}, but they are hard to compute in general.
    \item 
    %Also it is another problem that we have to 
    For what concrete diffeomorphisms can the invariants be calculated? %the invariant.
    \item In the closed simply-connected case a homeomorphism acting trivially on homology is topologically isotopic to the identity by work of Quinn~\cite{Q86} and Perron~\cite{P86}. Is there an analog of this theorem for 4-manifolds with boundaries?
    %
   % One big advantage in working with simply-connected closed 4-manifold is that existence of the theorem by Quinn~\cite{Q86} and Perron~\cite{P86} which says that if a homeomorphism acting trivially on the homology, then it is in fact topologically isotopic to the identity.
    %There is no known analogue of this theorem for 4-manifolds with boundary.
\end{itemize}

The last point was resolved by Orson and Powell \cite{OP} very recently, and one of our main contributions in this paper is about the first two points where we define a numerical and computable invariant for diffeomorphisms of 4-manifolds with boundary.

\begin{rem}
One may also conjecture that the maps in \cref{main conj} are non-surjective.
%One may also propose a conjecture similar to \cref{main conj} by  considering the non-surjectivity of the maps  \eqref{eq: main target}, instead of the non-injectivity.
To our knowledge, there is no known counter example to this conjecture.  %as well as \cref{main conj}.
Positive evidence to the non-surjectivity of the maps is given in \cite{KT20} where non-surjectivity is shown for certain $W$, which is a partial generalization of results for closed 4-manifolds in \cite{FM88,FM882,Do90,MS97,Ba21,KN20} to 4-manifolds with boundary.
\end{rem}

\begin{rem}
For higher homotopy groups, it seems also natural to state conjectures similar to \cref{main conj}. Compared to the $\pi_0$ case, there are few results on the non-injectivity or the non-surjectivity on $\pi_i$ for $i>0$, even for closed 4-manifolds.
For example, a non-injectivity result for $S^4$ is proven by Watanabe in \cite{Wa18} and a non-surjectivity result on $K3$ is proven by Baraglia and the second author in \cite{BK19}.
Our framework also yields invariants, say,  $\pi_i(\Diff(W,\del)) \to \Z$ for $i>0$, and it may be possible to attack this problem using these invariants.
\end{rem}

%We can also propose a weaker conjecture as following.
%\begin{conjecture}
% Let $Y$ be a closed, connected, oriented 3-manifold. Then there exists a compact, oriented 4-manifold $W$ with $\partial W=Y$ such that the natural map 
% \[
%\pi_0 (\operatorname{Diff} (W, \partial )  )\to \operatorname{Aut} (H_2(W;\Z)  )
%\]
%is not injective.
%\end{conjecture}

\subsection{Families Kronheimer--Mrowka invariants} 
\label{subsection: Families Kronheimer--Mrowka's invariants}

Our basic strategy is to make use of contact structures on 3-manifolds $Y$ to study diffeomorphisms of 4-manifolds bounded by $Y$.
Let us first explain the main tool of this paper, which is a family version of the invariant defined by Kronheimer and Mrowka in \cite{KM97} for 4-manifolds with contact boundary. This is a rather simple numerical invariant compared to invariants using monopole Floer theory \cite{KM07}.

Let $(W, \s)$ be a compact Spin$^c$ 4-manifold with contact boundary $(Y,\xi)$, here the $\Spinc$ structure on $Y$ induced by $\xi$ agrees with the one induced by $\s$.
Define $\Diff(W, [\fraks], \del)$ to be the group of diffeomorphisms that preserve the isomorphism class $[\fraks]$ and $\del W$ pointwise.
Define a subgroup $\Diff_H(W, [\fraks], \del)$ of $\Diff(W, [\fraks], \del)$ by collecting all diffeomorphisms in  $\Diff(W, [\fraks], \del)$ that act trivially on the homology of $W$.
From these groups, we shall define homomorphisms
\begin{align}
\label{intro FKM1}
FKM(W, \fraks, \xi, \bullet) : \pi_0(\Diff(W, [\fraks],\del)) \to \Z_2
\end{align}
and 
\begin{align}
\label{intro FKM2}
FKM(W, \fraks, \xi, \bullet) : \pi_0(\Diff_H(W, [\fraks],\del)) \to \Z,
\end{align}
which we call the {\it families Kronheimer--Mrowka invariants} of diffeomorphisms.
These invariants are defined by considering a family of Seiberg--Witten equations over a family of 4-manifolds with a cone-like almost K\"ahler end.
This end is associated with the contact structure $\xi$ on $Y$.

The families Kronheimer--Mrowka invariant is defined similarly to %along a similar spirit of 
the definition of the families Seiberg--Witten invariant for families of closed 4-manifolds \cite{Rub98,LL01}, but to derive an interesting consequence involving isotopy of absolute diffeomorphisms (property (i) of \cref{main1})
we need to consider a refined version of the above invariant using  %defined with 
information about the topology of the space of contact structures on $Y$.
Let $\Xi^{\mathrm{cont}}(Y)$ denote the space of contact structures on $Y$.
We shall define a map, which is not a homomorphism in general,
\begin{align*}
 FKM(W, \fraks, \xi, \bullet, \bullet) : \Diff(W, [\fraks],\del) \times \pi_1(\Xi^{\mathrm{cont}}(Y), \xi) \to \Z_2,
\end{align*}
or equivalently
\begin{align}
\label{intro refined FKM}
 \overline{FKM}(W, \fraks, \xi, \bullet) : \Diff(W, [\fraks],\del)  \to (\Z_2)^{\pi_1(\Xi^{\mathrm{cont}}(Y), \xi)},
\end{align}
which is defined by considering families of 4-manifolds with contact boundary, rather than fixing the contact structure on the boundary.
%and 
%\[
%FKM(W, \fraks, \xi, \bullet) : \Diff_H(W, [\fraks],\del) \to \Z^{\pi_1(\Xi^{\mathrm{cont}}(Y), \xi)},
%\]
The map \eqref{intro refined FKM} satisfies the property that, if $f \in \Diff(W, [\fraks],\del)$ is isotopic to the identity in 
%through the {\it absolute} diffeomorphism group 
$\Diff(W)$, then some component of the value $\overline{FKM}(W, \fraks, \xi, f) \in (\Z_2)^{\pi_1(\Xi^{\mathrm{cont}}(Y), \xi)}$ is trivial.
For some class of $f$ and element of $\pi_1(\Xi^{\mathrm{cont}}(Y), \xi)$, we may also extract a $\Z$-valued invariant rather than $\Z_2$.

Now we shall see some applications of these invariants.

%{\color{red} Shall we also say about the difficulty in definiing the signed version?}

\subsection{Exotic diffeomorphisms}

%Although at this moment we could not able to answer \cref{main conj}, in our first main theorem we show evidence on why we should believe that such a conjecture is true. 
We give evidence for \cref{main conj} by showing that many $3$-manifolds bound $4$-manifolds that admit exotic diffeomorphisms with several other interesting properties.% which includes a property involving exotic 2-spheres.

%{\color{red}We need to define ribbon homology cobordism here.}
Let $Y_1,Y_2$ be two closed, oriented 3-manifolds. 
We say there is a {\it ribbon homology cobordism} from $Y_1$ to $Y_2$ if there is an 
%We call $Y_2$ is in the {\it ribbon homology cobordism class} of $Y_1$ if there exists an 
integer homology cobordism from $Y_1$ to $Y_2$ which has a handle decomposition with only 1- and 2-handles. Note that ribbon homology cobordism is not a symmetric relation. %and the ribbon homology cobordism class above does not mean any equivalence class.

\begin{thm}
\label{main1}
Given any closed, oriented 3-manifold $Y$ with rational homology of $S^3$ or of $S^1 \times S^2$, there exists a ribbon homology cobordism to a 3-manifold $Y'$ %in the ribbon homology cobordism class of  $Y$ 
such that $Y'$ bounds a compact oriented smooth 4-manifold $W$ for which there exists a diffeomorphism $f \in \Diff(W,\del)$ such that
%The mapping class of $f$ in $\Diff(W)$ is a non-trivial element of the kernel of
%\[
%\pi_0 (\operatorname{Diff} (W) )\to\pi_0(\operatorname{Homeo}(W)).
%\]
%Moreover, $f$ is topologically isotopic to the identity through $\Homeo(W,\del)$.
the (absolute) mapping class $[f] \in \pi_0(\Diff(W))$ generates an infinite subgroup %isomorphic to $\Z$ 
of the kernel of
\[
\pi_0 (\operatorname{Diff} (W) )\to\pi_0(\operatorname{Homeo}(W)).
\]
\end{thm}

Namely, $\{f^n\}_{n \in \Z \setminus \{0\}}$ are exotic diffeomorphisms that are smoothly non-isotopic to each other in $\Diff(W)$.
Moreover, in fact, $\{f^n\}_{n \in \Z}$ are topologically isotopic to the identity in $\Homeo(W,\del)$.
(This follows from a result by Orson and Powell, \cref{top isotopy}.)

\begin{rem}
For a general oriented closed 3-manifold $Y$, we have a statement similar to \cref{main1}, under replacing $\pi_0 (\operatorname{Diff} (W) )\to\pi_0(\operatorname{Homeo}(W))$ with $\pi_0 (\operatorname{Diff} (W) )\to \Aut (H_2(W;\Z))$.
\end{rem}

Moreover, the diffeomorphism $f$ in \cref{main1} yields exotic spheres in 4-manifolds:

\begin{thm}
\label{thm: emb sphere}
In the setup of \cref{main1}, there exists a homologically non-trivial embedded 2-sphere $S$ in $W$ such that the spheres $\{f^n(S)\}_{n \in \Z}$ are mutually exotic in the following sense: if $n \neq n'$, then $f^n(S)$ and $f^{n'}(S)$ are topologically isotopic in $\Homeo(W, \del)$, but not smoothly isotopic in $\Diff(W, \del)$. 
\end{thm}

\begin{rem}
There is a great deal of work on exotic surfaces in 4-manifolds that makes use of the diffeomorphism types of their complements, see \cite{FKV87, FS97,A14, AKHMR15}. Recently, Baraglia~\cite{Ba20} gave exotic surfaces in closed 4-manifolds whose complements are diffeomorphic, based on a method closely related to our technique. Also, in \cite{LM21}, J.~Lin and the third author gave exotic surfaces in the punctured $K3$ whose complements are diffeomorphic using the 4-dimensional Dehn-twist and the families Bauer--Furuta invariant.  
%{\color{red} Shall we mention my work with Jianfeng as well?} 
\end{rem}

%\todo{Mention studies of exotic 2-knots here: At least, Auckly-Ruberman-Melvin-$\cdots$, Schwartz, Yasui, Baraglia, and studies having non-diffeomorphic complements (just one word?)}

\begin{rem}
One motivation to consider 4-manifolds with boundary is to find simple manifolds with exotic diffeomorphisms. At the moment the simplest example of a closed 4-manifold that admits an exotic diffeomorphism has $b_2=25$: concretely, $4\C P^2\#21(-\C P^2) \cong K3 \# \C P^2\#2(-\C P^2)$. On the other hand, when we consider 4-manifolds with boundary, one may get an exotic diffeomorphism of a compact 4-manifold with $b_2=4$, see \cref{small b_2} for details.% ({\color{red} maybe we should give an explicit description of $b_2=7$? Write explicit examples}).
\end{rem}

%\todo{Mention Betti numbers of $W$ is much smaller than closed 4-manifolds case?}

An interesting point in \cref{main1} is that, despite that this theorem shall be shown by the invariants involving contact structures explained in \cref{subsection: Families Kronheimer--Mrowka's invariants}, the results can be described without contact structures.
Another interesting point is that the $f^n$ are not mutually smoothly isotopic in the {\it absolute} diffeomorphism group $\Diff(W)$.
Of course, this is stronger than the corresponding statement for  $\Diff(W,\del)$, and the proof requires us to use the refined invariant \eqref{intro refined FKM}, which is less straightforward than the invariants \eqref{intro FKM1} and \eqref{intro FKM2}.
Not only the gauge-theoretic aspect, the proof of \cref{main1} uses some non-trivial techniques of Kirby calculus.

\begin{rem}
Under the setup of \cref{main1},
the mapping class of $f$ in $\Diff(W,\del)$ generates a direct summand isomorphic to $\Z$ in the abelianization of the kernel of 
\[
\pi_0 (\operatorname{Diff} (W,  \partial )  )\to \pi_0(\operatorname{Homeo}(W, \partial)),
\]
see \cref{rem: direct sum main}. This comes from the fact that the map \eqref{intro FKM2} is a homomorphism.
In an upcoming paper \cite{IKMT}, under suitable conditions on $Y$, we also prove the existence of a $\Z^\infty$-summand in the abelianization of the kernel of $\pi_0 (\operatorname{Diff} (W,  \partial )  )\to \pi_0(\operatorname{Homeo}(W, \partial))$. 
\end{rem}

\subsection{Exotic codimension-1 submanifolds in 4-manifolds}
Now we will focus on exotic 3-manifolds in 4-manifolds. 
\begin{defn}\label{exotic def}
We call two smoothly embedded 3-manifolds $Y_1$ and $Y_2$ in a smooth, oriented 4-manifold $X$ are {\it exotic} if \begin{itemize}
    \item[(i)] there is a topological ambient isotopy $H_t: X\times [0,1] \to X$ such that $H_0= Id$ and $H_1(Y_1)=Y_2$,
    \item[(ii)] there is no such smooth isotopy,
    \item[(iii)] complements of $Y_1$ and $Y_2$ are diffeomorphic. 
\end{itemize} 
\end{defn}

\begin{rem}

 If we ignore (i) in \cref{exotic def}, recently, Budney--Gabai \cite{BG19} and independently Watanabe \cite{Wa20} found examples of smoothly embedded 3-balls in $S^1\times D^3$ that are homotopic but not smoothly isotopic.
%We shall consider exotic 3-dimensional submanifolds in 4-manifolds, that is two copies of 3-manifolds in a given 4-manifold that are topologically isotopic but not smoothly isotopic. 
\end{rem}
\begin{rem}
%We shall focus on the case that the complements of the submanifolds are diffeomorphic to each other, since if two 3-manifolds are topologically isotopic but their complements are not diffeomorphic then there cannot exist any smooth isotopy between them. 
If we do not consider (iii), one can easily construct an example that satisfies (i) and (ii) in \cref{exotic def} as follows. Let $W$ and $W'$ be a pair of Mazur corks \cite{Akb16,G&S}  that are homeomorphic but not diffeomorphic \cite{AY19,HMP}. Then their doubles $D(W) $ and $D(W')$ are diffeomorphic to the standard $S^4$ \cite{mathoverflow}. Thus we have two embedded 3-manifolds $Y=\partial W$ and $Y' =\partial W'$ that are not smoothly isotopic in $S^4$, since if they were smoothly isotopic. By the isotopy extension theorem, their complements would be diffeomorphic as well. Now notice that the two embedded 3-manifolds $Y$ and $Y'$ are topologically isotopic since there exists an orientation preserving homeomorphism $f: S^4\to S^4$ such that $f(Y)= Y'$. Now Quinn's result \cite{10.4310/jdg/1214437139} says that $f$ is topologically isotopic to identity and thus the two embedded 3-manifolds $Y$ and $ Y'$ are topologically isotopic.
\end{rem}

In this work, to the author's best knowledge, we construct the first example of exotic 3-spheres in a 4-manifold with diffeomorphic complements i.e, satisfying all conditions (i)-(iii):

\begin{thm}
\label{thm: simplest exotic emb}
Let $X=4 \mathbb{C}P^2 \# 21 (-\mathbb{C}P^2)$ and $S$ be the 3-sphere embedded in $X$ which gives a connected sum decomposition
$X= X_1\# X_2$, with $b^+_2(X_1)=b^+_2(X_2)=2$.
%(2 \mathbb{C}P^2\# 21 (-\mathbb{C}P^2)) \#_S (2 \mathbb{C}P^2). 
Then there exists infinitely many copies of $S$ with diffeomorphic complements that are topologically isotopic but not smoothly.
\end{thm}

This \lcnamecref{thm: simplest exotic emb} follows from a more general statement, \cref{exotic 3-manifolds non-vanishing}.

%In \cref{exotic 3-manifolds non-vanishing} we shall show that a similar result holds for 4-manifolds with boundary supposing that $FKM(X,\fraks,\xi,f)\neq0$, in place of $FSW(X,\fraks,f)\neq0$.
%\begin{rem}
%The above result also true for 4-manifolds with boundary, and for that purpose we need to use the invariant $FKM$. {\color{red} we may try to write the above theorem in a such a way that consider both FSW and FKM in the statement and we do not need to write this extra remark.}
%\end{rem}

\begin{rem}
It is not explicitly stated but known to the experts that techniques involving the complexity of h-cobordism defined Morgan--Szab\'o \cite{MS98} can be used to show the existence of exotic 3-spheres in 4-manifolds. However, to the authors' knowledge, the techniques can only produce a pair of such $S^3$ and moreover, the complements of such spheres may not be diffeomorphic. 
%Also, in \cite{Sc19}, Schwartz detected exotic 2-spheres whose complements are diffeomorphic based on  Morgan--Szab\'o's  technique.
\end{rem}

\begin{rem}
\label{rem: intro Baraglia}
In \cite{Ba20}, Baraglia gave examples of exotic embedded surfaces $\Sigma$ in some 4-manifolds with diffeomorphic complements.
It follows from his work that one may find infinitely many exotic embedded $S^1 \times \Sigma$ whose complements are diffeomorphic to each other.
Baraglia's argument is based on a family version of the adjunction inequality.
Our exotic embedded 3-manifolds  have vanishing $b_1$,
and use some vanishing results for family versions of connected sum formula for Seiberg--Witten invariants.
\end{rem}

%In \cref{exotic 3-sphere}, for simplicity, we stated our vanishing result only when the 3-manifold $Y=S^3$. Moreover, we shall also prove an analogous vanishing result for some class of 3-manifolds; connected sums of elliptic 3-manifolds, some of hyperbolic three-manifolds in the Hodgson--Weeks census, which is explained in details at Section~\ref{section: Exotic embeddings of 3-manifolds in 4-manifolds}. Using the vanishing result, we can obtain the following result for those 3-manifolds:
In \cref{section: Exotic embeddings of 3-manifolds in 4-manifolds}, we generalize \cref{thm: simplest exotic emb} to 3-manifolds other than $S^3$, in particular, we establish the following result. 
 
\begin{thm}
\label{strongly exotic} 
Let $Y$ be one of the following 3-manifolds: 
\begin{itemize}
\item[(i)] a connected sum of elliptic 3-manifolds, or 
    \item[(ii)] a hyperbolic three-manifold labeled by 
\[
0, 2,3,8 , 12, \ldots, 16, 22, 25, 28 ,\ldots, 33, 39, 40, 42, 44, 46, 49
\]
in the Hodgson--Weeks census, \cite{HW}. 
\end{itemize}
Then there exists a smooth closed 4-manifold $X$ and infinitely many embeddings $\{i_n  : Y \to X\}_{n \in \Z}$ that are exotic in the following sense: topologically isotopic (as embeddings) but not smoothly isotopic (as submanifolds).
Moreover, the complements $X \setminus i_n(Y)$ are diffeomorphic to each other. 
\end{thm}

\begin{rem}
In addition to the use of the families gauge theory, the proof of \cref{strongly exotic} uses two deep results in 3-manifold theory; the generalized Smale conjecture for hyperbolic 3-manifolds proven by Gabai \cite{Ga01} and the connectivity of the space of positive scalar curvature metrics for 3-manifolds proven by Bamler and Kleiner \cite{BaKl19}. 
These results are used to obtain the vanishing of certain parameterized moduli space for the Seiberg--Witten equations. 
The list in (ii) comes from \cite[Figure 1]{LL18} due to F.~Lin and Lipnowski,
which lists small hyperbolic 3-manifolds having no irreducible Seiberg--Witten solution.   
\end{rem}

\begin{rem}
Recently in \cite{AR1,AR2}, Auckly and Ruberman also detected higher-dimensional families of exotic embeddings and diffeomorphisms by using families Yang--Mills gauge theory. 
Also, they detected exotic embeddings of 3-manifolds into $S^4$.

Drouin also detected exotic lens spaces in some 4-manifolds by modifying the argument by Auckly and Ruberman.
\end{rem}

%\todo{Mention the Smale conjecture here}

\begin{comment}
\begin{thm}
For a fixed element in $\Lambda(W, \s,  \xi)$, we have 
\begin{itemize}
    \item[(i)]  $FKM(W, \s, \xi, f)$ is invariant under isotopy of $f$ rel boundary,  
    \item[(ii)] $FKM(W, \s, \xi, f)$ is invariant under isotopy of $\xi$,
    \item[(iii)] $FKM(W, \s, \xi, g\circ f) = FKM(W, \s, \xi, g)+ FKM(W, \s, \xi, f)$, 
    \item[(iv)]  $FKM(W, \s, \xi, \id)=0$
    \item[(v)] if there is a diffeomorphism $g : (W, \s, \xi) \to (W', \s', \xi')$, then 
    \[
    FKM(W, \s, \xi, f)= FKM(W', \s', \xi', g\circ  f \circ g^{-1} )
    \]
    for a certain choice of an element $\Lambda(W', \s',  \xi')$.
\end{itemize}
\end{thm}
\end{comment}

%We finish off this introduction with an outline of the contents of this paper.
\textit{Organization:} In \cref{Signed Kronheimer--Mrowka invariants}, we review the Kronheimer--Mrowka invariant for 4-manifolds with contact boundary, which has a sign ambiguity and show that certain auxiliary data fixes the sign. 
In \cref{section Families Kronheimer--Mrowka invariant}, we define the families Kronheimer--Mrowka invariants \eqref{intro FKM1} and \eqref{intro FKM1} of 4-manifolds with contact boundary, and the refined families Kronheimer--Mrowka invariant \eqref{intro refined FKM} as well.
We also establish basic properties of these invariants.
In \cref{sectionSeveral vanishing results}, we show several vanishing results for the families Seiberg--Witten and Kronheimer--Mrokwa invariant which can be regarded as family versions of the connected sum formula, Fr\o yshov's vanishing result, and adjunction inequality.
In \cref{section: Construction of diffeomorphisms and non-vanishing results}, we construct some diffeomorphisms of some 4-manifolds with boundary, for which we shall show the families Kronheimer--Mrowka invariants are non-trivial. 
In \cref{sectionProof of Theorem main1}, we give the proof of one of our main theorems, \cref{main1}.
In \cref{section: Exotic embeddings of 3-manifolds in 4-manifolds}, we prove the results on exotic embeddings of 3-manifolds into 4-manifolds.

\begin{acknowledgement} 
The authors would like to thank Patrick Orson, Mark Powell, and Arunima Ray for letting us know about \cite{OP}. The authors would like to express their deep gratitude to Juan Mu\~noz-Ech\'aniz for answering several questions on his paper \cite{J21}. The authors would also like to express their appreciation to Tadayuki Watanabe for answering several questions on his paper and discussing codimension-1 embeddings \cite{Wa20}. The authors also wish to thank Dave Auckly, David Baraglia, John Etnyre, Jianfeng Lin, Ciprian Manolescu, Hyunki Min and Danny Ruberman for answering various questions regarding diffeomorphisms of 4-manifolds and making helpful comments on the draft. 
%Dave Auckly, Jianfeng Lin, David Baraglia

Nobuo Iida was supported by JSPS KAKENHI Grant Number 19J23048 and the Program for Leading Graduate Schools, MEXT, Japan.
Hokuto Konno was supported by JSPS KAKENHI Grant Numbers 17H06461, 19K23412, and 21K13785.
Anubhav Mukherjee was partially supported by NSF grant DMS-1906414. Masaki Taniguchi was supported by JSPS KAKENHI Grant Number 20K22319, 22K13921 and RIKEN iTHEMS Program.
\end{acknowledgement}

\section{Signed Kronheimer--Mrowka invariants} 
\label{Signed Kronheimer--Mrowka invariants}
We first review the Kronheimer--Mrowka invariant introduced in \cite{KM97}. 
The Kronheimer--Mrowka invariant 
\[
\mathfrak{m}(W, \s, \xi) \in \Z/\{\pm 1\}
\]
is an invariant of a 4-manifold $W$ with a contact boundary $(Y,\xi)$ equipped with a 4-dimensional $\Spinc$ structure $\mathfrak{s}$ compatible with the $\xi$. 

Usually, the notion of Spin/Spin$^c$ structure on an oriented manifold $W$ of dimension $d$ is defined by fixing a metric on $W$ and as a lift of the structure of the $SO(d)$-frame bundle of $W$ to a principal $Spin(d)$- or $Spin^c(d)$-bundle.
Here we note that one can define those notions without using Riemannian metric, which shall be convenient when we consider families of manifolds.
Denote by $GL^{+}(d,\R)$ the group of real square matrices of size $d$ of $\det >0$.
A Spin structure on an oriented $d$-manifold can be defined as a lift of the structure group of the frame bundle from $GL^{+}(d,\R)$ to the double cover $\widetilde{GL}^{+}(d,\R)$.
Similarly, a Spin$^{c}$ structure is also defined using $(\widetilde{GL}^{+}(d,\R) \times S^{1})/\pm1$ instead of $Spin^{c}(d)$.

In the case of the usual Seiberg--Witten invariant for a closed 4-manifold $X$, it is enough to fix a homology orientation of $X$, i.e. an orientation of $H^1(X; \R)\oplus H^+(X; \R)$ to fix a sign of the invariant. However, in Kronheimer--Mrowka's setting, we cannot use such data to give an orientation of the moduli space. In order to improve this, we introduce a two element set
\[
\Lambda (W, \s, \xi) 
\]
depending on a tuple $(W, \s, \xi) $ whose element gives an orientation of the moduli space in Kronheimer--Mrowka's setting.

Let $W$ be a connected compact oriented 4-manifold with connected contact boundary $(Y, \xi)$.
Let $\s$ be a $\Spinc$ structure on $W$ which is compatible with $\xi$.
Pick a contact 1-form $\theta$ on $Y$ and a complex structure $J$ of $\xi$ compatible with the orientation. There is now an unique Riemannian metric $g_1$ on $Y$ such that $\theta$ satisfies $|\theta|=1$, $d\theta=2*\theta$, and $J$ is an isometry for $g|_\xi$, where $*$ is the Hodge star operator with respect to $g_1$.
This can be written explicitly by
\[
g_1=\theta\otimes \theta +\frac{1}{2}d\theta(\cdot, J\cdot)|_\xi.
\]
Define a symplectic form $\omega_0$ on $\R^{\geq 1}\times Y$ by the formula
$\omega_0=\frac{1}{2}d(s^2\theta)$, where $s$ is the coordinate of $\R^{\geq 1}$.
We define a conical metric on $\R^{\geq 1} \times Y$ by 
\begin{align}\label{conical metric} 
g_0 := ds^2 + s^2 g_1. 
 \end{align}

On $\R^{\geq 1} \times Y$, we have a canonical $\Spinc$ structure $\s_0$ , a canonical $\Spinc$ connection $A_0$, a canonical positive Spinor $\Phi_0$.
These are given as follows.
The pair $(g_0, \om_0)$ determines a compatible almost complex structure $J$ on $\R^{\geq 1} \times Y$. The $\Spinc$ structure on $\R^{\geq 1} \times Y$ is given by:
 \[
 \s_0:= ( S^+ = \Lambda^{0,0}_J \oplus  \Lambda_J^{0,2}, S^- = \Lambda_J^{0,1} , \rho :\Lambda^1 \to \Hom (S^+ , S^-) ),
 \]
  where the Clifford multiplication $\rho$ is given by
\[
 \rho  = \sqrt{2} \operatorname{Symbol} (\overline{\partial} + \overline{\partial}^* ) . 
 \]
 (See Lemma 2.1 in \cite{KM97}.)
 The notation $\Phi_0$ denotes 
 \[
 (1,0) \in \Om_{\R^{\geq1} \times Y}^{0,0} \oplus  \Om_{\R^{\geq1} \times Y}^{0,2}= \Gamma (S^+|_{\R^{\geq1} \times Y}).
 \]
 Then the {\it canonical $\Spinc$ connection} $A_0$ on $\s_0$ is uniquely defined by the equation 
 \begin{align}\label{A0}
 D^+_{A_0} \Phi_0= 0
 \end{align}
 on $\R^{\geq 1} \times Y$.
 We write the conical part $\R^{\geq 1} \times Y$ by $ C^+$. 

Let $W^+$ be a non-compact 4-manifold with a conical end 
\[
W^+ := W \cup_Y  (\R^{\geq 1} \times Y).
\]
Fix a smooth extension of $(A_0, \Phi_0)$ on $W^+$.
On $W^+$ define Sobolev spaces
\[
\mathcal{C}_{W^+}=(A_0, \Phi_0)+L^2_{k, A_0 }(i \Lambda^1_{W^+}\oplus S^+_{W^+}) \text{ and } 
\]
\[
{\cV}_{W^+}=L^2_{k-1, A_0 }( i\Lambda^+_{W^+}\oplus S^-_{W^+})
\]
for $k \geq 4$, where $S^+_{W^+}$ and $S^-_{W^+}$ are positive and negative spinor bundles and the Sobolev spaces are given as completions with respect to the following inner products: 
\begin{align}\label{inner}
\langle s_1, s_2\rangle_{L^2_{k, A_0} } := \sum_{i=0}^k \int_{W^+} \langle \nabla^i_{A_0} s_1, \nabla^i_{A_0} s_2  \rangle  \operatorname{dvol}_{W^+}, 
\end{align}
where the connection $\nabla^i_{A_0}$ is  induced  from ${A_0}$ and the Levi-Civita connection.  
The gauge group is defined by 
\[
\G_{W^+} := \left\{ u: W^+\to U(1) | 1- u \in L^2_{k+1} \right\}. 
\]
The action of $\G_{W^+}$ on ${\mathcal{C }}_{W^+}$ is given by 
\[
 u \cdot  (A, \Phi) :=   (A- u^{-1} du , u\Phi).  
\]
Set 
\[
\mathcal{B}_{W^+} :=   \Con_{W^+} /  \G_{W^+}
\] and call it {\it the configuration space}. Note that since $(A_0, \Phi_0)$ is irreducible in the end, one can see every element in $\mathcal{B}_{W^+}$ is irreducible. 

We have {\it the perturbed Seiberg--Witten map} 
\begin{align}\label{FX+}
\begin{split}
&\mathfrak{F}:  {\mathcal{C}}_{W^+}\to {\cV}_{W^+}\\
&(A, \Phi)\mapsto (\frac{1}{2}F^+_{A^t}-\rho^{-1}(\Phi\Phi^*)_0-\frac{1}{2}F^+_{A^t_0}+\rho^{-1}(\Phi_0\Phi^*_0)_0+\eta, D^+_A\Phi).
\end{split}
\end{align}
Here $\eta$ is a generic perturbation decaying $C^r$ exponentially.

We have the infinitesimal action of gauge group at every point $(A, \Phi)\in \mathcal{C}_{W^+}$
\[
\delta_{(A, \Phi)}: L^2_{k+1, A_0 }(i\Lambda^0_{W^+})   \to L^2_{k, A_0 }(i\Lambda^1_{W^+}\oplus S^+)
\]
and the
linearization of the Seiberg--Witten map at $(A, \Phi)\in \mathcal{C}_{W^+}$ 
\[
\mathcal{D}_{(A, \Phi)}\mathfrak{F}: L^2_{k, A_0 }(i \Lambda^1_{W^+}\oplus S^+_{W^+})  \to L^2_{k-1, A_0 }( i\Lambda^+_{W^+}\oplus S^-_{W^+}).
\]
The sum 
\begin{align}\label{gauge slice}
\mathcal{D}_{(A, \Phi)}\mathfrak{F}+\delta^*_{(A, \Phi)}:
L^2_{k, A_0 }(i \Lambda^1_{W^+}\oplus S^+_{W^+})  \to L^2_{k-1, A_0 }(i\Lambda^0_{W^+}\oplus i\Lambda^+_{W^+}\oplus S^-_{W^+})
\end{align}
is a linear Fredholm operator.

The moduli space is defined to be 
\[
M(W, \s, \xi) := \{ (A, \Phi) \in {\mathcal{B}}_{W^+} | {\mathfrak{F}} (A,\Phi) =0 \}.    
\]
It is proven in \cite{KM97}, the moduli space $M(W, \s, \xi)$ is compact. For a suitable class of perturbations, it is proven in \cite{KM97} that $M(W, \s, \xi)$ is a smooth manifold of dimension 
\[
d(W, \s, \xi) = \langle e(S^+ , \Phi_0 |_{\partial W} ) , [W, \partial W] \rangle, 
\]
where $e(S^+ , \Phi_0 |_{\partial W} ) \in H^4(W, \partial W)$ is the relative Euler class of $S^+$ with respect to the section $\Phi_0 |_{\partial W}$ on the boundary.

In order to give orientations of moduli spaces, we need the following lemma: 
\begin{lem}\cite[Theorem 2.4]{KM97}\label{triviality of orientation bundle}The line bundle
\[
\operatorname{det} (\mathcal{D}_{(A, \Phi)}\mathfrak{F}+\delta^*_{(A, \Phi)})  \to \B_{W^+}
\]
is trivial. 
\qed
\end{lem}

%\begin{proof}
%For any loop $l$ in $\B_{W^+}$, it is sufficient to prove $\operatorname{det} (\mathcal{D}_{(A, \Phi)}\mathfrak{F}+\delta^*_{(A, \Phi)})|_{l}$ is trivial. We see triviality of  $\operatorname{det} (\mathcal{D}_{(A, \Phi)}\mathfrak{F}+\delta^*_{(A, \Phi)})|_{l}$ via family excision argument. Take an almost complex 4-manifold $(Z_1,J_1)$, $(Z_2,J_2)$ bounded by $(Y,\xi)$, $(-Y,\xi)$ respectively.

%\end{proof}

Here we give data to fix this sign. We first give a definition of an orientation set.
\begin{defn} 
Define the two element set by 
\[
\Lambda (W, \s, \xi ): = \{ \text{ orientations of the determinant line bundle of \eqref{gauge slice} over $\mathcal{B}_{W^+} $ }  \} . 
\]
\end{defn}
Note that $\Lambda (W, \s, \xi )$ does not depend on the choices of elements in $\mathcal{B}_{W^+} $ since $\mathcal{B}_{W^+} $ is connected. 
Once we fix an element in $\Lambda (W, \s, \xi )$, we have an induced orientation on the moduli space $M(W, \s, \xi)$.  
We also give another description of $\Lambda (W, \s, \xi )$ using almost complex 4-manifolds bounded by $(Y, \xi)$. 
We use the following existence result of almost complex 4-manifolds bounded by a given 3-manifold.
The proof is written in the proof of Proposition 28.1.2 of \cite{KM07}, for example.
\begin{lem}
Let $Y$ be a closed oriented 3-manifold and $\xi$ be an oriented 2-plane field on $Y$.
Then there is an almost complex 4-manifold $(W, J)$ bounded by $(Y, \xi)$, which means $\partial W=Y$ and $JTY \cap TY = \xi$ up to homotopy of 2-plane fields. 
\qed
\end{lem}
Using this lemma, we also define another two element set. 
\begin{defn} For a fixed almost complex 4-manifold $(Z, J)$ bounded by  $(-Y, \xi)$, we define 
\[
\Lambda (W, \s, \xi, Z, J )
\]
to be the two-element set of trivializations of the orientation line bundle for the linearized equation with a slice on the closed Spin$^c$ 4-manifold $(W \cup Z, \s \cup \s_J)$, where $\s_J$ is the Spin$^c$ structure determined by $J$. 
\end{defn} 
By its definition, $\Lambda (W, \s, \xi, Z, J )$ can be regarded as the set of homology orientations of the closed 4-manifold $W \cup Z$. 

We see behavior of $\Lambda (W, \s, \xi, Z, J )$ under the changes of $(Z, J)$.
The excision argument enables us to show the following: 
\begin{lem} 
For any two choices of almost complex bounds $(Z, J)$ and $(Z', J')$, one has a canonical identification 
\[
\Lambda(W, \s, \xi, Z, J ) \cong \Lambda(W, \s, \xi, Z', J' ). 
\]
\end{lem}
\begin{proof}
This follows from an excision argument. Take an almost complex 4-manifold $Z_1$ bounded by $(Y, \xi)$.
We apply \cref{excision} by putting $A_1 = W$, $B_1 = Z$, $A_2=Z_1$, $B_2=Z' $, 
\[
D_1=\text{The linearization of the Seiberg--Witten map with slice on }X_1 = W \cup Z, \text{ and } 
\]
\[
D_2=\text{The linearization of the Seiberg--Witten map with slice on }X_2 = Z_1 \cup Z'.
\]
By \cref{excision}, we have an isomorphism 
\[
\det D_1 \otimes \det D_2\to 
\det \wt{D}_1 \otimes \det \wt{D}_2
\]
up to homotopy. Since the Spin$^c$ structures on $X_2 $ and $\wt{X}_2$ are induced by almost complex structures, $\det {D}_2$ and $\det \wt{D}_2$ has a canonical trivialization. 
So, we obtain a canonical isomorphism 
\[
\det D_1 \to \det \wt{D}_1. 
\]
This gives a correspondence between $\Lambda(W, \s, \xi, Z, J )$ and $\Lambda(W, \s, \xi, Z', J' )$. 
\end{proof}

%Thus, we omit the notion $Z$ and $J$ in $\Lambda(W, \s, \xi, Z, J )$, so we write $\Lambda(W, \s, \xi )$. 

For a Spin$^c$ 4-manifold with countact boundary $(W,  \xi)$, we introduced two orientations sets 
\[
\Lambda(W, \s, \xi, Z, J ) \text{ and }\Lambda (W,\s,   \xi ). 
\]
We can define a natural correspondence between these orientation sets. 
Take an almost complex bound $Z_1$ of $(Y, \xi)$.
We again apply \cref{excision} by putting $A_1=W$, $B_1= C^+$, $A_2 = Z_1$ and $B_2 =Z$ 
\[
D_1=\text{The linearization of the Seiberg--Witten map with slice on }X_1 = W \cup C^+, \text{ and } 
\]
\[
D_2=\text{The linearization of the Seiberg--Witten map with slice on }X_2 = Z_1 \cup Z.
\]
By \cref{excision}, we have an isomorphism 
\[
\det D_1 \otimes \det D_2\to 
\det \wt{D}_1 \otimes \det \wt{D}_2. 
\]
Since the Spin$^c$ structures on $X_2 $ and $\wt{X}_2$ are induced by almost complex structures, $\det {D}_2$ and $\det \wt{D}_2$ has a canonical trivialization. 
So, we obtain a canonical isomorphism 
\[
\det D_1 \to \det \wt{D}_1. 
\]
This gives a bijection
\[
\psi: \Lambda(W, \s, \xi, Z, J ) \to \Lambda (W, \s , \xi ). 
\]
A similar proof enables us to show $\psi$ does not depend on the choices of $Z_1$.
\begin{lem}
There is a canonical one-to-one correspondence 
\begin{align}\label{det line identification}
\psi: \Lambda(W, \s, \xi, Z, J ) \to \Lambda (W, \s, \xi ). 
\end{align}
\end{lem}
We use this alternative description of the orientation set when we define signed families Kronheimer--Mrowka invariants.

\begin{rem}
For a symplectic filling $(W, \om )$, one can choose a canonical element in $\Lambda(W, \s_\om , \xi)$ by choosing an orientation coming from a compatible almost complex structure, where $\s_\om$ is the Spin$^c$ structure coming from $\om$.  
\end{rem}

%\begin{lem}For a given element in $\Lambda(W, \s, \xi)$, an orientation of the determinant line of the operator $L$ is fixed. 
%\end{lem}

 \begin{defn}
 For a fixed element in $\lambda \in \Lambda(W, \s, \xi)$, 
 we define the {\it signed Kronheimer--Mrowka invariant} by 
 \[
 \mathfrak{m}(W,\s, \xi, \lambda) 
 := \begin{cases} \# {M}(W,\s, \xi) \in \Z \text{ if } \langle e(S^+ , \Phi_0 |_{\partial W} ) , [W, \partial W] \rangle =0  \\ 
  0 \in \Z \text{ if } \langle e(S^+ , \Phi_0 |_{\partial W} ) , [W, \partial W] \rangle \neq 0. 
 \end{cases} 
 \]
 %Note that the configuration space $\B(W)$ 
% where the counting $\# \mathcal{M}_E$ is considered with respect to an orientation of $B$ and an orientation of $\operatorname{det} \mathbb{L}_{W^+}|_b$ on a fiber of $b\in B$.
\end{defn}
We often abbreviate $ \mathfrak{m}(W,\s, \xi, \lambda)$ by $ \mathfrak{m}(W,\s, \xi)$.

The above definition enables us to define a map 
\[
 \mathfrak{m} : \operatorname{Spin}^c(W, \xi ) \to \Z  
\]
for a fixed element in $\Lambda(W, \s, \xi)$, where $\operatorname{Spin}^c(W, \xi )$ is the set of isomorphism classes of all Spin$^c$ structures which are compatible with $\xi$ on the boundary.

%The group ring $\Z[H^2 (W, \partial W; \Z)/ \operatorname{Tor}]$ is the ring generated by formal variables $t_K$ which satisfy
%\[
%t_K t_{K'} = t_{K+K'} \text{ and }  t_0 =1. 
%\]

%For example, one has 
%\begin{itemize}
    %\item 
 %$t_{-K} = t^{-1}_K$ and 
    %\item $t_{nK} = t_{K}^n$. 
    %\end{itemize}

\section{Families Kronheimer--Mrowka invariant}
\label{section Families Kronheimer--Mrowka invariant}

We introduce families Kronheimer--Mrowka invariants in this section. 
We follow the construction of families Seiberg--Witten invariants \cite{Rub98,LL01}.

Let $W$ be a connected compact oriented 4-manifold with contact boundary $Y$.
It is possible to consider a version of our invariant for disconnected $Y$: In that case, we need to replace $\pi_1(\Xi^{\mathrm{cont}}(Y), \xi)$ that appeared in the introduction with the direct sum of such fundamental groups for all components of $Y$.
For simplicity, we shall suppose that $Y$ is connected in this paper.

As explained in \cref{Signed Kronheimer--Mrowka invariants}, we define the notion of  Spin structure/Spin$^{c}$ structure without using Riemannian metric, by considering
$\widetilde{GL}^{+}(d,\R)$.
If a Spin structure or a Spin$^{c}$ structure $\fraks$ is given on $W$, define 
$\Aut(W,\fraks)$ to be the group of automorphisms of the Spin$^c$ manifold $(W,\fraks)$.

Let $\Xi^{\mathrm{cont}}(Y)$ be the space of contact structures on $Y$ equipped with the $C^\infty$-topology, which is an open subset of the space of oriented 2-plane distributions.
Let $\tilde{\Gamma} : B \to \Xi^{\mathrm{cont}}(Y)$ be a smooth map. We denote the smooth homotopy class of $\tilde{\Gamma}$ by $\Gamma$.
Let $W \to E \to B$ be a fiber bundle over a closed smooth manifold $B$ of dimension $n$.
Let $\Diff^+(W,[\fraks])$ denote the group of orientation preserving diffeomorphisms fixing the isomorphism class of $\fraks$, and let
$\Diff(W, [\fraks],\del)$ denote the group of diffeomorphisms fixing boundary pointwise and the isomorphism class of $\fraks$.
Let $\Aut_{\del}(W,\fraks)$ denote the inverse image of $\Diff(W, [\fraks],\del)$ under the natural surjection
\[
\Aut(W,\fraks) \to \Diff^+(W,[\fraks]).
\]
Suppose that the structure group of $E$ reduces to $\Aut_{\del}(W,\fraks)$.
Namely, $E$ is a fiber bundle whose restriction to the boundary is a trivial bundle of 3-manifolds, and is equipped with a fiberwise Spin$^c$ structure $\fraks_E$.
Suppose also that
\[
\s_{\tilde{\Gamma}(b)} =  \s_E|_{E_b}
\]
on each fiber.

%whose structure group is $\operatorname{Diff}(X, \partial )$, where $\operatorname{Diff}(X, \partial )$ is the group of diffeomorphisms on $X$ fixing its boundary pointwise. 
%Let $\xi$ be a contact structure on $Y=\partial W$ compatible with the boundary orientation.
For these data, we define the families Kronheimer--Mrowka  invariant 
\[
FKM(E, \Gamma) = FKM(E, W, \s_E, \Gamma) \in \Z_2.  
\]
This invariant is trivial by definition unless
when $\langle e(S^+, \Phi_0), [W, \partial W]\rangle+n=0$
where $e(S^+, \Phi_0)$ is the relative Euler class with respect to a special section $\Phi_0$, which we introduced in the previous section.

When $n=1$, we can define a signed family Kronheimer--Mrowka invariant 
\[
FKM(E) \in \Z   
\]
under a certain assumption on determinant line bundles.

\subsection{Notation}
Let $(Y, \xi)$ be a closed contact 3-manifold. 
We use the following geometric objects used in \cref{Signed Kronheimer--Mrowka invariants}: 
\begin{itemize}
 \item a contact 1-form $\theta$ on $Y$, 
 \item a complex structure $J$ of $\xi$ compatible with the orientation, 
 \item the Riemannian metric $g_1$ on $Y$ such that $\theta$ satisfies $|\theta|=1$, $d\theta=2*\theta$, 
 \item the symplectic form $\omega_0$ on $\R^{\geq 1}\times Y$, 
 \item the conical metric $g_0$ on $\R^{\geq 1} \times Y$, 
 \item the canonical $\Spinc$ structure $\s_0$ on $\R^{\geq 1} \times Y$, 
 \item the canonical positive Spinor $\Phi_0$ on $\R^{\geq 1} \times Y$, and
 \item the canonical $\Spinc$ connection $A_0$ on $\s_0$. 
\end{itemize}
Let $(W, \s)$ be a connected compact oriented $\Spinc$ 4-manifold with connected contact boundary $(Y, \xi)$. 

Assume that a trivialization of $E|_{\del}$, the fiberwise boundary of $E$, is given.
From this assumption, we may further suppose that there is a trivialization of a family of collar neighborhoods of the family of the boundaries $E|_{\del}$.
This is because the group of diffeomorphisms of $W$ fixing boundary pointwise is weakly homotopy equivalent to the group of diffeomorphisms of $W$ that are the identity near $\del W$ (see, e.g. \cite[Theorem 5.3.1]{Kup19}).
Let $W^+$ be a non-compact 4-manifold with a conical end defined in \cref{Signed Kronheimer--Mrowka invariants}. 
%Pick a Riemann metric $
%g$ on $W^+$ such that $g|_{\R^{\geq 1} \times Y}=g_0$. 
We define a fiber bundle 
\[
W^+ \to E^+ \to B
\]
whose fiber is $\Spinc$ 4-manifold with conical end  obtained from $W \to E \to B$ by considering $W^+ = W \cup_Y  (\R^{\geq 1} \times Y)$ on each fiber.

From now on, we will explain auxiliary data that are needed to define the family Kronheimer--Mrowka invariant. These data consist of choices of a contact form, a complex structure on a contact plane, a compatible Riemann metric, an extension of the canonical connection and the canonical spinor, which are denoted by $\mathcal{Q}$.
In addition, we also need to fix choices of a weight function and a perturbation, which are again denoted by $\mathcal{R}$. The main point here is that the set of these auxiliary data is non-empty and contractible, and thus the cobordism class of the Seiberg--Witten moduli space does not depend on the choices of such additional data. Although it is not so hard to verify it, for the readers let us carefully write the spaces of such additional data. 
%Experts can skip it. 

%In order to consider families Seiberg--Witten equation, we need to fix auxiliary data $(\theta, J, g, A_0, \Phi_0), (\sigma, \eta)$ which we will explain below.
%We have to consider the fiber bundle structure of these data in order to do the Sard-Smale transversality argument rigorously. 

\begin{comment}
Let $\mathcal{P}_1(Y, \xi)$ be the set of pairs $(\theta, J)$, where 
\begin{itemize}
\item $\theta$ is a contact form for the contact structure $\xi$,
\item $J$ is an complex structure on the contact structure $\xi$ compatible with orientation.
\end{itemize}
We equip $\mathcal{P}_1(Y, W, \fraks, \xi)$ with the $C^\infty$-topology.

For each $(\theta, J) \in \mathcal{P}_1(Y, W, \fraks, \xi)$
Let $\Met(W,\xi, \theta, J)$ be the set of metrics that are smooth extensions of the canonical metric for $(\xi, \theta, J)$ on the conical end to the whole manifold $W$.
Varying over $\mathcal{P}_1(Y, W, \fraks, \xi)$, we obtain a fiber bundle $\mathcal{P}_2(Y, W, \fraks, \xi) \to \mathcal{P}_1(Y,  \xi)$ with fiber $\Met(W,\xi, \theta, J)$.

For each $g \in \mathcal{P}_2(Y, W, \fraks, \xi)$, 
\end{comment}

Let $\mathcal{Q}(Y, W, \fraks,\xi)$ be the set of tuples
\[
(\theta, J, g,A_0^W, \Phi_0^W),
\]
\begin{itemize}
\item $\theta$ is a contact form for the contact structure $\xi$,
\item $J$ is an complex structure on the contact structure $\xi$ compatible with orientation.
\item
$g$ is a smooth extension  of the canonical metric for $(\xi,\theta,J)$ on the conical end to the whole manifold $W$,
\item 
$(A^W_0, \Phi_0^W)$ is a smooth extension of the canonical configuration $(A_0, \Phi_0)$ on the conical end to the whole manifold $W$.
\end{itemize}
Varying over $\Xi^{\mathrm{cont}}(Y)$, we obtain a fiber bundle
$\mathcal{Q}(Y, W, \fraks) \to \Xi^{\mathrm{cont}}(Y)$
with fiber $\mathcal{Q}(Y, W, \fraks,\xi)$.

Let $\mathcal{P}(Y, W, \xi, g)$ be the set of pairs
$(\sigma, \eta)$,
where 
\begin{itemize}
\item
$\sigma$ is a smooth proper extension of the $\R^{\geq 1}$ coordinate of the conical end to the whole manifold $W$, and 
\item
$\eta$ is an imaginary valued $g$-self-dual 2-form that belongs to $e^{-\epsilon_0 \sigma}C^r(i\Lambda^{+_g})$ for some $\epsilon_0>0$ and $r\geq k$, where $e^{-\epsilon_0 \sigma}C^r(i\Lambda^{+_g})$ denotes the completion of the vector space of compactly supported smooth sections of $i\Lambda^{+_g}$ with respect to the norm: 
\[
\| s \| := \| e^{-\epsilon_0 \sigma} s \|_{C^r (W^+; i\Lambda^{+_g} ) }.  
\]
\end{itemize}
Varying over the set of $g$,
we obtain a fiber bundle $\mathcal{R}(Y,W,\fraks,\xi) \to \mathcal{Q}(Y,W,\fraks,\xi)$ with fiber $\mathcal{P}(Y, W,\xi, g)$, which is independent from $(A_0^W, \Phi_0^W)$-component.
Varying over $\Xi^{\mathrm{cont}}(Y)$, we obtain a fiber bundle $\mathcal{R}(Y,W,\fraks) \to \Xi^{\mathrm{cont}}(Y)$ with fiber $\mathcal{R}(Y,W,\fraks,\xi)$, which covers $\mathcal{Q}(Y,W,\fraks) \to \Xi^{\mathrm{cont}}(Y)$.

Except for $\eta$, we consider the $C^\infty$-topology for the above data. For $\eta$, we equip weighted $C^r$-topology.
Since the total space of a fiber bundle with contractible fiber and base is also contractible, we have that $\mathcal{R}(Y, W, \fraks, \xi)$ is contractible.

The group $\Aut_\del(W,\fraks)$ acts on the total space $\mathcal{R}(Y, W, \fraks)$ via pull-back.
Thus $E$ induces an associated fiber bundle $E_\mathcal{R} \to B$ with fiber $\mathcal{R}(Y, W, \fraks)$.
Since the image of $\Aut_\del(W,\fraks)$ under the natural map
$\Aut(W,\fraks) \to \Diff(W,[\fraks])$ is contained in $\Diff(W,\del)$, whose restriction to the boundary acts trivially on $\Xi^{\mathrm{cont}}(Y)$, the map $\mathcal{R}(Y, W, \fraks) \to \Xi^{\mathrm{cont}}(Y)$ induces a map $E_\mathcal{R} \to \Xi^{\mathrm{cont}}(Y)$.
Define $\pi : E_\mathcal{R} \to B \times \Xi^{\mathrm{cont}}(Y)$ be the product of these natural maps $E_\mathcal{R} \to B$ and $E_\mathcal{R} \to \Xi^{\mathrm{cont}}(Y)$.
Then each fiber of $\pi$ is homeomorphic to $\mathcal{R}(Y, W, \fraks, \xi)$.
The fiber bundle
\[E_{\mathcal{R}}^{\tilde{\Gamma}} :=  (\id,\tilde{\Gamma})^\ast E_{\mathcal{R}} \to B
\]
the pull back of $\pi : E_{\mathcal{R}} \to B \times \Xi^{\mathrm{cont}}(Y)$ under $(\id,\tilde{\Gamma}) : B \to B \times \Xi^{\mathrm{cont}}(Y)$, has fiber homeomorphic to $\mathcal{R}(Y, W, \fraks, \xi)$.
Since $\mathcal{R}(Y, W, \fraks, \xi)$ is contractible, the space of sections of $E_{\mathcal{R}}^{\tilde{\Gamma}} \to B$ is non-empty and contractible.

In a similar manner, we can define a fiber bundle $E_{\mathcal{Q}}^{\tilde{\Gamma}} \to B$ with fiber $\mathcal{Q}(Y, W, \fraks,\xi)$, associated with $E$ and $\Gamma$.
%Since the fiber $\mathcal{P}(Y, W, \fraks, \xi)$ is contracible, the space of sections of this fiber bundle is non-empty and contractible.
%the identity component $\Diff(W,[\fraks],\del)$ of the whole diffeomorphism group $\Diff(W)$ acts on $\Xi^{\mathrm{cont}}(Y)$ via pull-back, and along the map  $\mathcal{P}(Y, W, \fraks, \xi) \to \Xi^{\mathrm{cont}}(Y)$, this action lifts to an action on $\mathcal{P}(Y, W, \fraks, \xi)$ again via pull-back.
We first fix a section 
\[
s^\mathcal{Q} = (\theta_b,J_b, g_b, A_{0, b}, \Phi_{0,b})_{b \in B}   : B \to E_{\mathcal{Q}}^{\tilde{\Gamma}} 
\]
which determines the following data: 
\begin{itemize}
\item a fiberwise contact form $\theta_b$, 
\item a fiberwise complex struxture $J_b$, 
    \item a fiberwise Riemann metric $
\{g_b\}_{b \in B}$ on $E^+$ such that $g_b|_{\R^{\geq 1} \times Y}=g_{0,b}$, and
\item a smooth family of smooth extensions $(A_{0,b}, \Phi_{0,b})$ of $(A_0, \Phi_0)$ on each fiber.  
\end{itemize}
Here $g_{0,b}$ is the metric on $Y$ depending on $J_b$ and $\theta_b$ introduced in the previous section. 
%Fix a smooth family of $\Spinc$ structures 
%\[
%\{\s_b\}_{b \in B}=\{(S^{\pm}_b, \rho_b)\}_{b \in B} 
%\]
%on $E^+$ equipped with an isomorphism $\s_b\to \s_0$ on $E^+_b\setminus E_b$.
%We will omit this isomorphism in our notation.
%Let us write the second coordinate of the cone $\R^{\geq 1} \times Y$ by $\sigma=\{\sigma\}_{b \in B}$.
%We also fix a smooth proper extension $\sigma$ to all of $E^+$ which is an extension of $s \in \R^{\geq 1}$ coordinate  on each fibe and is $0$ on $W \setminus \nu (\partial W)$, where $\nu (\partial W)$ is a small color neighborhood of $\partial W$ in $W$.
Consider the following functional spaces on each fiber $E^+_b$:
\[
\Con_{W^+,b}=(A_{0, b}, \Phi_{0, b})+L^2_{k, A_{0,b}, g_b }(i \Lambda^1_{W^+}\oplus S^+_{W^+}) \text{ and } 
\]
\[
{\cV}_{W^+,b}=L^2_{k-1, A_{0,b}, g_b }( i\Lambda^+_{W^+}\oplus S^-_{W^+}).
\]
These give Hilbert bundles: 
\[
{\mathbb{U}}_{E^+}( s^\mathcal{Q}):= \bigcup_{b \in B} \Con_{W^+,b}
\text{ and }
{\mathbb{V}}_{E^+}( s^\mathcal{Q}) := \bigcup_{b \in B} {\cV}_{W^+,b}.
\]

For the precise definitions of these Sobolev spaces see \cref{Signed Kronheimer--Mrowka invariants}.
The gauge group 
\[
\G_{W^+} := \Set{ u: W^+\to U(1) | 1- u \in L^2_{k+1} } 
\]
is defined and it
acts on $\mathbb{U}_{E^+}$ preserving fibers in \cref{Signed Kronheimer--Mrowka invariants}. 
We define a family version of the configuration space by 
\begin{align}\label{conf family}
\B_{E^+}( s^\mathcal{Q})  : = \mathbb{U}_{E^+} / \G_{W^+}. 
\end{align}

Now we also choose a section 
\[
s^\mathcal{R}= (\theta_b,J_b, g_b, A_{0, b}, \Phi_{0,b}, \sigma_b, \eta_b)_{b \in B} : B \to E^\Gamma_\mathcal{R}
\]
which is compatible with the fixed section $s^\mathcal{Q}$, i.e. the first five components of $s^\mathcal{R}$ coincide with these of $s^\mathcal{S} = (\theta_b,J_b, g_b, A_{0, b}, \Phi_{0,b})_{b \in B}$. 
%Let us describe a class of perturbations used to define a family Kronheimer--Mrowka's invariant. 
%\begin{comment}
%Set
%\[
%Met(E_b)=\{g : \text{Riemann metric on } E^+_b \text{ such that } g|_{\R_{\geq 1}\times Y }=g_0\}
%\]
%For each $b \in B$,
%\[
%\Pi_b \to Met(E_b)
%\]
%is defined as the disjoint union 
%\[
%\Pi_b = \sqcup_{g \in Met(E_b)}\Omega^{+_g}(E_b).
%\]
%\end{comment}
%Set
%\[
%\Pi = \sqcup_{b \in B} \Om^{+_{g_b}}(E^+_b).
%\]
%We have a natural projection $\Pi \to B$.
%Fix $r\geq k$ and $\epsilon>0$.
%\begin{defn}
%A section of this bundle $\Pi \to B$ satisfying that $\|e^{\epsilon \sigma} \eta_b\|_{C^r}$ is bounded for all $b \in B$
%is called a {\it families perturbation}.
%\end{defn}
%For each $b \in B$,
%the space of self-dual 2-forms $\eta$ on $E^+_b$ with bounded $\|e^{\epsilon \sigma} \eta\|_{C^r}$ is contractible.
%It follows from this that the space of families perturbations is contractible.
For each fiber $E^+_b$, we have the perturbed Seiberg--Witten map
\begin{align}\label{FX+}
\begin{split}
&{\mathfrak{F}}_b  :  \Con_{E^+,b}\to {\cV}_{E^+,b}\\
&(A-A_{0, b}, \Phi-\Phi_{0,b})\mapsto (\frac{1}{2}F^+_{A^t}-\rho^{-1}(\Phi\Phi^*)_0-\frac{1}{2}F^+_{A^t_{0,b}}+\rho^{-1}(\Phi_{0,b}\Phi^*_{0,b})_0+\eta_b, D^+_A\Phi).
\end{split}
\end{align}
This gives a bundle map 
\begin{align}\label{perturbed SW family}
{\mathfrak{F}}(s^\mathcal{R}): {\mathbb{U}}_{E^+}(s^\mathcal{Q})\to 
{\mathbb{V}}_{E^+}(s^\mathcal{Q}) .
\end{align}

\begin{defn}
We say that $\{\eta_b\}_{b \in B}$ is a {\it family regular perturbation} if \eqref{perturbed SW family} is transverse to zero section of ${\mathbb{V}}_{E^+}(s^\mathcal{Q}) $. 
\end{defn}

For each fiber, 
we have the infinitesimal action of gauge group at every point $(A_b, \Phi_b)\in \mathcal{C}_{W^+,b}$
\[
\delta_{(A_b, \Phi_b)}: L^2_{k+1, A_{0,b} }(i\Lambda^0_{W^+})   \to L^2_{k, A_{0,b} }(i\Lambda^1_{W^+}\oplus S^+)
\]
and the
linearization of the Seiberg--Witten map at $(A_b, \Phi_b)\in \mathcal{C}_{W^+,b}$ 
\[
\mathcal{D}_{(A_b, \Phi_b)}\mathfrak{F}: L^2_{k, A_{b}, g_b }(i \Lambda^1_{W^+}\oplus S^+_{W^+})  \to L^2_{k-1, A_{b}, g_b }( i\Lambda^+_{W^+}\oplus S^-_{W^+})
\]
The sum 
\[
\mathcal{D}_{(A_b, \Phi_b)}\mathfrak{F}+\delta^*_{(A_b, \Phi_b)}:
L^2_{k, A_{b}, g_b }(i \Lambda^1_{W^+}\oplus S^+_{W^+})  \to L^2_{k-1, A_{b}, g_b }( i\Lambda^+_{W^+}\oplus S^-_{W^+})
\]
is a linear Fredholm operator.
This gives a fiberwise Fredholm operator
\[
\mathbb{L}_{E^+}(s^\mathcal{Q}): T_{\operatorname{fiber}}{\mathbb{U}}_{E^+}(s^\mathcal{Q}) \to {\mathbb{V}}_{E^+}(s^\mathcal{Q}), 
\]
where $T_{\operatorname{fiber}}$ means the fiberwise tangent bundle of ${\mathbb{U}}_{E^+}(s^\mathcal{Q})$. 
By taking the determinant line bundle for each fiber, we obtain a line bundle
\begin{align}\label{determinant line for family}
\operatorname{det} (\mathbb{L}_{E^+}(s^\mathcal{Q}) )  \to \B_{E^+}(s^\mathcal{Q}) . 
\end{align}

\begin{rem}
As a similar study, in \cite{J21}, Juan defined the families monopole contact invariant for families of contact structures on a fixed 3-manifold. 
At the moment, we do not know the triviality of the line bundle \eqref{determinant line for family}. 
\end{rem}

\subsection{Constructions of the invariant}

For a fixed section $
s^\mathcal{R} : B \to  E^\Gamma_\mathcal{R}
$ such that $\{\eta_b\}_{b\in B}$ is regular, 
the parametrized moduli space is defined to be 
\[
\mathcal{M}(E,\tilde{\Gamma} ,  s^\mathcal{R}) := \{ (A, \Phi) = (A_b, \Phi_b)_{b \in B} \in {\mathbb{U}}_{W^+} | {\mathfrak{F}} (A,\Phi) =0 \} / \G_{W^+}.    
\]
Recall the formal dimension 
\[
d(W,\fraks, \xi)=
\langle e(S^+ , \Phi_0 |_{\partial W} ) , [X, \partial X] \rangle
\]
of the (unparametrized) moduli space over the cone-like end 4-manifold $W^+$.

\begin{prop}\label{orientability of determinant line}
For a regular perturbation, $\mathcal{M}(E, \tilde{\Gamma}, s^\mathcal{R})$ is a smooth compact manifold of dimension $d(W,\fraks, \xi)+n$. If the determinant line bundle 
\[
\operatorname{det} (\mathbb{L}_{E^+}(s^\mathcal{Q}) )  \to \B_{E^+}(s^\mathcal{Q}) ,
\] is trivialized, an orientation of $\mathcal{M}(E,\tilde{\Gamma},   s^\mathcal{R})$ is naturally induced by an orientation of $B$ and an orientation of $\operatorname{det} (\mathbb{L}_{E^+}(s^\mathcal{Q}) )|_b$ on a fiber of $b\in B$.  
\end{prop}
\begin{proof}
The proof is the standard perturbation argument with the compact parameter space $B$. We omit it.
\end{proof}
 
\begin{defn}
 We define the {\it families Kronheimer--Mrowka invariant} of $E$ by 
 \[
 FKM(E, \tilde{\Gamma}, s^\mathcal{R}) 
 := \begin{cases} \# \mathcal{M}(E, \tilde{\Gamma},  s^\mathcal{R}) \in \Z_2 &\text{ if } d(W,\fraks, \xi) + n=0,  \\ 
  0 \in \Z_2 &\text{ if } d(W,\fraks, \xi) + n\neq 0 
 \end{cases} 
 \]
 for a fixed section $s^\mathcal{R}$. 
% where the counting $\# \mathcal{M}_E$ is considered with respect to an orientation of $B$ and an orientation of $\operatorname{det} \mathbb{L}_{W^+}|_b$ on a fiber of $b\in B$.
\end{defn}
Since we will see the number $FKM(E,\tilde{\Gamma} ,  s^\mathcal{R})$ does not depend on the choices of sections $s^\mathcal{R}$ and $\tilde{\Gamma}$ up to smooth homotopy, we always drop $s^\mathcal{R}$ in the notion and write $FKM(E, \Gamma)$.

 \begin{prop}\label{independence of most general inv}
 The number $FKM(E, \tilde{\Gamma},   s^\mathcal{R})$ is independent of the choices of the following data: 
 \begin{itemize}
     \item a section $s^\mathcal{R}$ and
     \item a choice of $\tilde{\Gamma}$ which belongs to the homotopy class $\Gamma$. 
 \end{itemize} 
 Also $FKM(E, \Gamma )$ depends only on the isomorphism class of $E$ as $\operatorname{Aut}((W, \s), \partial)$-bundles and $\Gamma$.
 \end{prop}
 
\begin{proof}
We take a smooth homotopy $\tilde{\Gamma}_t : I \times B \to \Xi^{\mathrm{cont}}(Y) $ between $\tilde{\Gamma}_0 $ and $\tilde{\Gamma}_1$ parametrized $t \in [0,1]$. Take two sections 
\[
s_0^\mathcal{R} : B \to E^\mathcal{R}_{\tilde{\Gamma}_0} 
\text{ and } 
s_1^\mathcal{R} : B \to E^\mathcal{R}_{\tilde{\Gamma}_1} 
\]
so that \eqref{perturbed SW family} is transverse for $i=0$ and $i=1$.
Note that a fiber of the bundle  \[
\bigcup_{t\in I} E^{\tilde{\Gamma}_t}_\mathcal{R}\to I \times B
\]
is contractible, we can take a section $s_t^\mathcal{R} : I \times B \to \bigcup_{t\in I} E^{\tilde{\Gamma}_t}_\mathcal{R}$ connecting $s_0^\mathcal{R}$ and $s_1^\mathcal{R}$ such that  \eqref{perturbed SW family} for $s_t^\mathcal{R}$ is transverse. So, the moduli space for $s_t^\mathcal{R}$ gives a cobordism between $\mathcal{M}(E, \tilde{\Gamma}_0, s_0^\mathcal{R} )$ and $\mathcal{M}(E, \tilde{\Gamma}_1, s_1^\mathcal{R} )$. This completes the proof. 
\end{proof}

\begin{rem}
Note that there is no reducible solution to the monopole equations over the conical end $4$-manifold $W^+$ under our boundary condition, and we do not have to impose any condition on $b^+_2(W)$ to ensure the well-definedness of the invariant, as well as the unparametrized Kronheimer--Mrowka invariant.
\end{rem}

 \subsection{Invariant of diffeomorphisms}
Now suppose that the base space $B$ is $S^1$ and $\tilde{\Gamma}$ is a constant map to $\xi$.
%When $\tilde{\Gamma}$ is a constant $\xi$, we write the family Kronheimer--Mrowka's invariant by $FKM(E, \xi)$.
In this case, the family $E \to S^1$ is determined by an element of $\Aut_{\partial}(W,\fraks)$.
An element of $\Aut_{\partial}(W,\fraks)$ is given as a pair $(f,\tilde{f})$: $f$ is a diffeomorphism $f : W \to W$ which preserves the isomorphism class of $\fraks$ and fix $\del W$ pointwise, and $\tilde{f}$ is a lift of $f$ to an automorphism on the honest Spin$^c$ structure $\fraks$ acting trivially over $\del W$.
All $E \to S^1$ can be viewed as the mapping torus of $W$ by $(f, \tilde{f})$. 

\begin{lem}
\label{lem: indep of lift}
Let $E$ be the mapping torus of $W$ by $(f, \tilde{f})$.
Then the invariant $FKM(E, \Gamma)$ depends only on the diffeomorphism $f$ and $\xi$.
\end{lem}

\begin{proof}
The kernel of the natural surjection
\[
\Aut_{\partial}(W,\fraks) \to \Diff(W, [\fraks],\del)
\]
is given by the gauge group $\G_{W}=\Set{u : W \to U(1) | u|_{\del W} \equiv 1}$.
Now suppose that we have two lifts $\tilde{f}_1$ and $\tilde{f}_2$ of $f$ to $\Aut_{\partial}(W,\fraks)$.
Let $E_i$ be the mapping torus of $(f, \tilde{f}_i)$.
Then the composition $\tilde{f}_1 \circ \tilde{f}_2^{-1}$ is given by a smooth map $u : W \to U(1)$ with $u \equiv 1$ on $\del W$.
Taking an extension of $u$ to a neighborhood of $W$ in $W^+$, and also a partition of unity around $\del W$, we can extend $u$ to a smooth map $u^+ : W^+ \to U(1)$ with $1-u \in L^2_{k+1}$.
Hence the moduli spaces used in the definition of $FKM(E_1)$ and that of $FKM(E_2)$ are identical to each other.
\end{proof}

\begin{defn}
For a fixed contact structure $\xi$ on $Y$, we define the {\it Kronheimer--Mrowka invariant for diffeomorphims} $FKM(W, \fraks, \xi, f)$ to be the invariant $FKM(E, \Gamma \equiv \xi)$ of the mapping torus $E$ with fiber $(W,\fraks)$ defined taking a lift of $f$ to $\Aut_{\partial}(W,\fraks)$.
Note that, by \cref{lem: indep of lift}, $FKM(W, \fraks, \xi, f)$ is independent of the choice of lift.
If $(\fraks, \xi)$ is specified, we sometimes abbreviate $FKM(W, \fraks, \xi, f)$ to $FKM(W, f)$.
\end{defn}

Now we have defined a map
\[
FKM(W, \fraks, \xi, \bullet) : \Diff(W, [\fraks],\del) \to \Z_2.
\]
We will show that this map is a homomorphism and descents to a map
\[
FKM(W, \fraks, \xi, \bullet) : \pi_0(\Diff(W, [\fraks],\del)) \to \Z_2.
\]

\subsection{A signed refinement of $FKM$ for diffeomorphisms}
\label{subsection A signed refinement of m for diffeomorphisms}
Again, in this subsection, we assume that $\Gamma$ is a constant function to $\xi$. Define a subgroup $\Diff_H(W, [\fraks], \del)$ of the relative diffeomorphism group $\Diff(W, [\fraks], \del)$ as the group of diffeomorphisms that act trivially on homology and preserve the isomorphism class $[\fraks]$ and $\del W$ pointwise.
Note that, if $W$ is simply-connected, $\Diff_H(W, [\fraks], \del)$ coincides with the group $\Diff_H(W, \del)$, the group of diffeomorphisms that act trivially on homology and preserve $\del W$ pointwise.

For each element of $\Lambda(W,\fraks,\xi)$, 
we shall define a map
\[
FKM(W, \fraks, \xi, \bullet) : \Diff_H(W, [\fraks],\del) \to \Z.
\]
The construction of this map is done essentially by a similar fashion to define $FKM : \Diff(W, [\fraks],\del) \to \Z_2$, but we need to count the parametrized moduli space taking into account its orientation.

For $f \in \Diff_H(W, [\fraks],\del)$, let $\tilde{f}$ be a lift of $f$ to an automorphism of the Spin$^c$ structure.
Let $E_{f, \tilde{f}}$ denote the mapping torus of $(W,\frakt)$ as a fiber bundle of Spin$^c$ 4-manifolds. Take a section 
\[
s^\mathcal{Q}:  S^1 \to E^\Gamma_\mathcal{Q}. 
\]
Let $\B_{E_{f, \tilde{f}}} (s^\mathcal{Q}) $ denote the families (quotient) configuration space associated to $E_{f, \tilde{f}}$ introduced in \eqref{conf family}. 

\begin{lem} 
\label{lem: diffeo moduli ori}
Suppose $f$ is homologically trivial. 
Each element in $\Lambda(W, \s, \xi)$ induces  a section of the orientation bundle
\begin{align}\label{ori-bundle}
 \Lambda (E_{f, \tilde{f}})  \to \B (E_{f, \tilde{f}})   
\end{align}
over the configuration space $\B (E_{f, \tilde{f}})$ for all  $f \in \Diff_H(W, [\fraks],\del)$ and lifts $\tilde{f}$ to the Spin$^c$ structure. 
\end{lem}

\begin{proof}
We first prove the line bundle \eqref{ori-bundle} is trivial for any pair $(f, \tilde{f})$ such that $f$ is homologically trivial.

We first regard $\B_{E_{f, \tilde{f}}} (s^\mathcal{Q}) $ as a mapping torus of the trivial bundle 
\[
\underline{\B}_{W, \s, \xi} := I \times \B_{W, \s, \xi} \to I
\]
via the map $(f, \tilde{f})$. From \cref{triviality of orientation bundle}, we see that the determinant line bundle $\Lambda(W, \s, \xi)$  over $\underline{\B}_{W, \s, \xi}$ is trivial. So, it is sufficient to prove that the induced map 
\[
(f, \tilde{f})_*:  \Lambda(W, \s, \xi) \to \Lambda(W, \s, \xi) 
\]
preserves a given orientation of $\Lambda(W, \s, \xi)$. 
In order to see this, we use the following canonical identification 
\eqref{det line identification}. 
First, we fix an almost complex 4-manifold $(Z_1, J_1)$ bounded by  $(-Y, \xi)$. We recall that 
\[
\Lambda (W, \s, \xi, Z, J )
\]
is defined by the two-element set of trivializations of the orientation line bundle for the linearized equation with a slice on the closed Spin$^c$ 4-manifold $(W \cup Z, \s \cup \s_J)$. Then, \eqref{det line identification} gives an identification \[
\psi : \Lambda (W, \s, \xi, Z, J ) \to \Lambda(W, \s, \xi) . 
\]
Note that $(f, \tilde{f})$ also naturally acts on $\Lambda (W, \s, \xi, Z, J )$. 

\begin{claim1}\label{family orientation excision}
The following diagram commutes: 
 \[
    \begin{CD}
    \Lambda (W, \s, \xi, Z, J ) @>{\psi}>> \Lambda(W, \s, \xi) \\
  @V{(f, \tilde{f})_*}VV    @V{(f, \tilde{f})_*}VV     \\
    \Lambda (W, \s, \xi, Z, J ) @>{\psi}>> \Lambda(W, \s, \xi) . 
  \end{CD}
 \]
\end{claim1}
\begin{proof}[Proof of \cref{family orientation excision}]
This result follows the construction of $\psi$ based on \cref{excision}. 
\end{proof}
Note that the orientation of $\Lambda (W, \s, \xi, Z, J )$ is determined just by the homology orientation of $W\cup Z$. 
Since we assumed that $\psi$ is homologically trivial, the induced action 
\[
(f, \tilde{f})_* : \Lambda (W, \s, \xi, Z, J ) \to \Lambda (W, \s, \xi, Z, J )
\]
is also trivial. Hence, we can see that the bundle $\Lambda (E_{f, \tilde{f}})  \to \B (E_{f, \tilde{f}}) $ is trivial. 
Now, we give an orientation of $\Lambda (E_{f, \tilde{f}}) $ from a fixed element in $\Lambda(W, \s, \xi)$. For a fixed element in $\Lambda(W, \s, \xi)$, an element $\Lambda (E_{f, \tilde{f}}) $ is induced by choosing a point in $\B_{ E_{f, \tilde{f}}}$ and restricting the bundle $\Lambda (E_{f, \tilde{f}}) $ to the point. Note that such a correspondence does not depend on the choices of lifts $\tilde{f}$. This completes the proof. 
\end{proof}

If $E=E_{(f, \tilde{f})}$ is the mapping torus, we can count the parametrized moduli space associated with $E$ over $\Z$ by \cref{lem: diffeo moduli ori}.

\begin{defn}
For $f \in \Diff_H(W, [\fraks],\del)$ and a lift $\tilde{f}$,
let $E=E_{(f, \tilde{f})}$ be the mapping torus of $(W,\fraks)$ by  $(f, \tilde{f})$.
we define the {\it signed families Kronheimer--Mrowka invariant} of $E$ by 
 \[
 FKM(E, \xi) 
 := \begin{cases} \# \mathcal{M}(E, \Gamma \equiv \xi, s^\mathcal{R}  ) \in \Z &\text{ if } d(W,\fraks, \xi) + 1=0,  \\ 
  0 \in \Z &\text{ if } d(W,\fraks, \xi) + 1\neq 0
 \end{cases} 
 \]
 for a fixed element in $\Lambda(W, \s, \xi)$.
 \end{defn}

Repeating the argument in \cref{lem: indep of lift}, we obtain:

\begin{lem}
\label{lem: indep of lift Z lift}
Let $E$ be the mapping torus of $(W, \fraks)$ by $(f, \tilde{f})$.
Then the invariant $FKM(E) \in \Z$ depends only on $f \in \Diff_H(W, [\fraks],\del)$.
\end{lem}
\begin{proof}
The proof is essentially the same as that of \cref{lem: indep of lift}.
\end{proof}

\begin{defn}
We define the {\it signed Kronheimer--Mrowka invariant for diffeomorphims} $FKM(W, \fraks, \xi, f)$ to be the invariant $FKM(E, \xi)$ of the mapping torus $E$ with fiber $(W,\fraks)$ defined by taking a lift of $f$ to $\Aut_{\partial}(W,\fraks)$.
Note that, by \cref{lem: indep of lift Z lift}, $FKM(W, \fraks, \xi, f)$ is independent of the choice of lift.
If $(\fraks, \xi)$ is specified, we sometimes abbreviate $FKM(W, \fraks, \xi, f)$ to $FKM(W, f)$.
\end{defn}

\subsection{Properties of the families Kronheimer--Mrowka invariant}
\label{Properties of Kronheimer--Mrowka's invariant for compositions of diffeomorphisms}

In this \lcnamecref{Properties of Kronheimer--Mrowka's invariant for compositions of diffeomorphisms}, we prove some basic properties of the families Kronheimer--Mrowka invariant.
This is parallel to Ruberman's original argument \cite[Subsection~2.3]{Rub98}.

Let $(W, \s)$ be a connected compact oriented $\Spinc$ 4-manifold with connected contact boundary $(Y, \xi)$.
In this \lcnamecref{Properties of Kronheimer--Mrowka's invariant for compositions of diffeomorphisms}, we fix $(\fraks,\xi)$ and we sometimes drop this from our notation of $FKM(W, \fraks, \xi, f)$.
 
 First we note the following additivity formula: 
 \begin{prop}
 \label{prop: additivity}
 For diffeomorphisms $f, f'$ of $W$ preserving the isomorphism class of $\fraks$ and fixing $\del W$ pointwise,
 we have
 \[
 FKM(W, \fraks, \xi,f) + FKM(W, \fraks, \xi, f') =  FKM(W, \fraks, \xi, f' \circ  f)
 \operatorname{\ mod\ } 2 . 
 \]
Moreover, when an element of $\Lambda(W, \s, \xi)$ is fixed, 
 we have
 \[
 FKM(W, \fraks, \xi,f) + FKM(W, \fraks, \xi, f') =  FKM(W, \fraks, \xi, f' \circ  f)
 \]
 as $\Z$-valued invariants for homologically tirivial diffeomorphisms $f, f'$ .  
 \end{prop}
 
 \begin{proof}
 We regard $FKM(W, f' \circ  f)$ as the counting of $\mathcal{M}(E_{f' \circ  f})$. 
 Note that the moduli space $\mathcal{M}(E_{f' \circ  f})$ is equipped with the map $\mathcal{M}(E_{f' \circ  f}) \to S^1$. 
 We fix the following data: 
 \begin{itemize}
  \item a Riemann metric $g$ on $W$ which coincides with $ g_1$ on $\partial W=Y$ and 
     \item a regular perturbation $\eta$ on $W^+$ for the metric $g$.
 \end{itemize}
 The invariant $\mathcal{M}(E_{f})$ can be seen as the counting of parametrized moduli space over $[0,\frac{1}{2}]$ with regular 1-parameter family of perturbation  $\eta_t$ and a  1-parameter family of metrics $g_t$ satisfying
 \begin{itemize}
     \item $g_0 = g $, $g_\frac{1}{2} =  f ^* g$ and 
     \item $\eta_0 = \eta$,  $\eta_\frac{1}{2} =  f ^* \eta $. 
 \end{itemize}
 Also, the invariant $\mathcal{M}_{E_{  f'}}$ can be seen as the counting of parametrized moduli space over $[\frac{1}{2},1]$ with the regular 1-parameter family of perturbation  $\eta_t$ and a  1-parameter family of metrics $g_t$ satisfying
 \begin{itemize}
     \item $g_{\frac{1}{2}} =  f ^* g $, $g_1 = (f' \circ  f )^* g$ and 
     \item $\eta_{\frac{1}{2}}= f ^* \eta$,  $\eta_1 = (f' \circ  f )^* \eta $. 
 \end{itemize}
 Then we have a decomposition
 \[
 \bigcup_{t \in [0,\frac{1}{2}] }  \mathcal{M} (W, g_t, \eta_t) \cup  \bigcup_{t \in [\frac{1}{2}, 1] }  \mathcal{M} (W, g_t, \eta_t) = \bigcup_{t \in [0 , 1] }  \mathcal{M} (W, g_t, \eta_t). 
 \]
 The counting of $\bigcup_{t \in [0 , 1] }  \mathcal{M} (W, g_t, \eta_t)$ is equal to $FKM(W, f' \circ  f)$ by the definition. 
 This completes the proof. Once we fix an orientation of $\Lambda(W, \s, \xi)$, the same argument enables us to prove the equality 
 \[
 FKM(W, \fraks, \xi,f) + FKM(W, \fraks, \xi, f') =  FKM(W, \fraks, \xi, f' \circ  f)\in \Z.
 \]
\end{proof}

\cref{prop: additivity} immediately implies:
 
\begin{cor}
\label{lem: identification then trivial}
We have $FKM(W, \fraks, \xi, f) = 0$ for $f=\id$.
\end{cor}

\begin{comment}
\begin{proof}
Take a generic (unparametrized) perturbation $\eta$.
By formal-dimenional reason, the unparametrized moduli space for this perturbation is empty.
The mapping torus for $\id$ is the trivialized bundle $S^1 \times W$, and we may take a families perturbation as the pull-back of $\eta$ to the family over $[0,1]$.
The parametrized moduli space for this families perturbation is empty, since this is just the disjoint union of copies of the above unparametrized moduli space, which is empty.
\end{proof}
\end{comment}

\begin{lem}
\label{lem: rel inv under isotopy}
The number $FKM(W, \fraks, \xi, f)$ is invariant under smooth isotopy of diffeomorphisms in $\Diff(W,[\fraks],\del)$.
\end{lem}

\begin{proof}
By \cref{lem: identification then trivial} and \cref{prop: additivity}, it suffices to check that if $f$ is isotopic to the identity, then we have $FKM(W, f)=0$.
Take a generic unparametrized perturbation $\eta$.
Let $f_t$ be a smooth isotopy from $\id$ to $f$.
Let $\eta_t =f_t^{\ast}\eta$, and $g_t$ be the underlying family of metrics.
The $\mathcal{M} (W, g_t, \eta_t)$ is diffeomorphic to $\mathcal{M} (W, g_0, \eta_t)$, which is empty.
\end{proof}

\begin{cor}
If $FKM(W, \fraks, \xi, f) \neq 0$, then $f$ is not isotopic to the identity through $\Diff(W,[\fraks],\del)$.
\end{cor}

\begin{proof}
This follows from \cref{lem: identification then trivial,lem: rel inv under isotopy}.
\end{proof} 

We end up with this \lcnamecref{Properties of Kronheimer--Mrowka's invariant for compositions of diffeomorphisms} by summarizing the above properties:

\begin{cor}
\label{cor: summary of properties of FKM}
The families Kronheimer--Mrowka invariant defines homomorphisms
\[
FKM(W, \fraks, \xi, \bullet) : \pi_0(\Diff(W, [\fraks],\del)) \to \Z_2
\]
and
\[
FKM(W, \fraks, \xi, \bullet) : \pi_0(\Diff_H(W, [\fraks],\del)) \to \Z.
\]
\end{cor} 

\begin{proof}
This follows from \cref{prop: additivity,lem: identification then trivial,lem: rel inv under isotopy}.
\end{proof}

\subsection{Isotopy of absolute diffeomorphisms}
\label{subsection Isotopy of absolute diffeomorphisms}

We now consider a slight refinement of the families Kronheimer--Mrowka invariant for diffeomorphisms defined until \cref{subsection A signed refinement of m for diffeomorphisms} to take into account isotopies of diffeomorphisms that are not necessarily the identity on the boundary. We need to treat a family of contact structures on the boundary in Kronheimer--Mrowka's setting. Such a situation is also treated in \cite{J21}.

For a contact structure $\xi$ on an oriented closed 3-manifold $Y$, let $[\xi]$ denote the isotopy class of $\xi$.
Let $W$ be a compact oriented smooth 4-manifold bounded by $Y$.
%Denote by $\Diff(W, [\xi])$ the group of diffeomorphisms $f$ which satisfy that $f^\ast \xi$ is isotopic to $\xi$.
%Similarly, denote by $\Diff(W, [\fraks], [\xi])$ the group of orientation-preserving diffeomorphisms $f$ which preserve the isomorphism class $[\fraks]$ and the isotopy class $[\xi]$.
%Let $\mathcal{P}(Y, W, \fraks, \xi)$ be the set of tuples
%\[
%(\theta, J, g,A_0^W, \Phi_0^W, % \eta, \sigma),
%\]
%where 
%\begin{itemize}
%\item $\theta$ is a contact form for the contact structure $\xi$,
%\item $J$ is an complex structure on the contact structure $\xi$ compatible with orientation,
%\item
%$g$ is a smooth extension  of the canonical metric on the conical end to the whole manifold $W$,
%\item 
%$(A^W_0, \Phi_0^W)$ is a smooth extension of the canonical configuration $(A_0, \Phi_0)$ on the conical end to the whole manifold $W$,
%\item
%$\sigma$ is a smooth proper extension of the $\R^{\geq}$ coordinate of the conical end to the whole manifold $W$, and 
%\item
%$\eta$ is an imaginary valued $g$-self-dual 2-form that belongs to $e^{-\epsilon_0 \sigma}(i\Lambda^{+_g})$.
%\end{itemize}
%Now we are ready to define a refined families Kronheimer--Mrowka invariant.
Let $f \in \Diff(W, [\fraks], \del)$ and $\gamma$ be a homotopy class of a loop in $\pi_1(\Xi^{\mathrm{cont}}(Y), \xi)$.
Henceforth we fix $\xi$ and abbreviate $\pi_1(\Xi^{\mathrm{cont}}(Y), \xi)$ as $\pi_1(\Xi^{\mathrm{cont}}(Y))$.
Pick a representative ${\tilde{\gamma}} : S^1 \to \Xi^{\mathrm{cont}}(Y)$ of $\gamma$ and a section $s^\mathcal{R} : S^1 \to E^{\tilde\gamma}_\mathcal{R}$.
%Let $\A (W, \fraks, \xi, f, \tilde{\gamma}, s) $ 
%the parametrized moduli space 
%\[
%\mathcal{M}(W, \fraks, \xi, f, \tilde{\gamma}, s^\mathcal{R}) 
%= \bigcup_{t \in S^1}M(W, %\fraks, \xi, s(t))  
%\]
%as a subset of 
%$
%\B_{E_{f, \tilde{f}}} (s^\mathcal{Q})  $. The following 
%The family of (perturbed) Seiberg--Witten equations to define $\mathcal{M}(W, \fraks, \xi, f, \tilde{\gamma}, s)$ is along the additional data determined by $s$, under the boundary conditions determined by the family of contact structures $\tilde{\gamma}$.
Then we defined a $\Z_2$-valued invariant
\[
FKM(W, [\fraks], \xi, f, \gamma)
 \in \Z_2. 
\]

For $f$ and $\gamma$, we can define the monodromy action on $\Lambda(W, \s, \xi)$.
If this action is trivial, we may count the parametrized moduli space over $\Z$, and thus can define
\[
FKM(W, [\fraks], \xi, f, \gamma)
 \in \Z, 
\]
whose sign is fixed once we choose an element of $\Lambda(W, \s, \xi)$.
Henceforth we fix an element of $\Lambda(W, \s, \xi)$.
It is useful to note that, for a pair admitting a square root, say $(f^2,\gamma^2) \in \Diff(W, [\fraks], \del) \times \pi_1(\Xi^{\mathrm{cont}}(Y), \xi)$, the corresponding monodromy action is trivial.
Let us summarize the situation in the following diagram:
 \[
   \xymatrix{
     \Diff(W, [\fraks], \del) \times \pi_1(\Xi^{\mathrm{cont}}(Y), \xi) \ar[rrr]^-{FKM(W, [\fraks], \xi, \bullet, \bullet)} & & & \Z_2 \\
     \Set{f^2 | f \in \Diff(W, [\fraks], \del)} \times \Set{\gamma^2 | \gamma \in \pi_1(\Xi^{\mathrm{cont}}(Y), \xi)}\ar[rrr]_-{FKM(W, [\fraks], \xi, \bullet, \bullet)}\ar@{^{(}-_>}[u] &&& \Z. \ar[u]^{\text{mod 2}}
   }
\]

The cobordism argument as in \cref{independence of most general inv} enables us to prove the invariance of the signed and refined families Kronheimer--Mrowka's invariant. 

\begin{prop}
\label{lem: ab diff: vanishing for some gamma}
Let $f \in \Diff(W,[\fraks], \del)$.
If $f$ is isotopic to the identity through $\Diff(W)$, then there exists $\gamma \in \pi_1(\Xi^{\mathrm{cont}}(Y))$ such that
\begin{align}
\label{eq: vanishing abs diff iso 0}
FKM(W, [\fraks], \xi, f, \gamma)=0 \in \Z_2 
\end{align}
and
\begin{align}
\label{eq: vanishing abs diff iso 1}
FKM(W, [\fraks], \xi, f^2, \gamma^2)=0 \in \Z.
\end{align}
\end{prop}

\begin{proof}
We may suppose that $d(W,\fraks, \xi) + 1=0$.
Let $f_t$ be a path in $\Diff(W)$ between $f$ and the identity.
Define a path $\tilde{\gamma} : [0,1] \to \Xi^{\mathrm{cont}}(Y)$ by $\tilde{\gamma}(t) = f_t^\ast \xi$, and set $\gamma=[\tilde{\gamma}] \in \pi_1(\Xi^{\mathrm{cont}}(Y))$.
Pick a generic element $a$ of the fiber of $\mathcal{R}(Y, W, \fraks) \to \Xi^{\mathrm{cont}}(Y)$ over the $\xi$.
The pull-back $s(t) = f_t^\ast a$ gives rise to a section $s : [0,1] \to \tilde{\gamma}^\ast\mathcal{R}(Y, W, \fraks)$.
By formal-dimensional reason, the moduli space for $a$ is empty.
Moreover, the pull-back under $f$ induces a homeomorphism between the moduli space for $a$ and that for $f^\ast a$, and hence the parametrized moduli space for $s$ is empty.
This completes the proof of \eqref{eq: vanishing abs diff iso 0}.

Next we prove \eqref{eq: vanishing abs diff iso 1}.
Let $\tilde{\gamma}_\# : [0,1] \to \Xi^{\mathrm{cont}}(Y)$ denote the concatenated path of two copies of $\tilde{\gamma}$.
The path $\tilde{\gamma}_\#$ is a representative of $\gamma^2 \in \pi_1(\Xi^{\mathrm{cont}}(Y))$.
Define a section $s' : [0,1] \to \tilde{\gamma}^\ast\mathcal{R}(Y, W, \fraks)$ by $s'(t)  = f^\ast(f_t^\ast a)$.
Namely, $s'$ is the pull-back section of $s$ under $f$.
By concatenating $s$ with $s'$, we obtain a section $s \cup s' : [0,1] \to \tilde{\gamma}_\#^\ast\mathcal{R}(Y, W, \fraks)$.
The left-hand side of \eqref{eq: vanishing abs diff iso 1} is the signed counting of the parametrized moduli space for $s \cup s'$, but again the moduli space is empty.
Thus we have \eqref{eq: vanishing abs diff iso 1}.
\end{proof}

\cref{lem: ab diff: vanishing for some gamma} can be generalized more:

\begin{prop}
\label{lem: ab diff: vanishing for some gamma general}
Let $f, g \in \Diff(W,[\fraks], \del)$.
If $f$ and $g$ are isotopic to each other through $\Diff(W)$, then there exists $\gamma \in \pi_1(\Xi^{\mathrm{cont}}(Y))$ such that
\begin{align}
\label{eq: equality FKM gamma ab isptopy}
FKM(W, [\fraks], \xi, f, \gamma)
= FKM(W, [\fraks], \xi, g, \gamma) \in \Z_2
\end{align}
and 
\begin{align}
\label{eq: equality FKM gamma ab isptopy1}
FKM(W, [\fraks], \xi, f^2, \gamma^2)
= FKM(W, [\fraks], \xi, g^2, \gamma^2) \in \Z.
\end{align}
\end{prop}

\begin{proof}
We may suppose that $d(W,\fraks, \xi) + 1=0$.
Let $h_t$ be a path in $\Diff(W)$ between $f$ and $g$.
Define a path $\tilde{\gamma} : [0,1] \to \Xi^{\mathrm{cont}}(Y)$ by $\tilde{\gamma}(t) = h_t^\ast \xi$, and set $\gamma=[\tilde{\gamma}] \in \pi_1(\Xi^{\mathrm{cont}}(Y))$.
Take a section $s_f^{\mathcal{R}} : [0,1] \to \tilde{\gamma}^\ast \mathcal{R}(Y, W, \fraks)$ so that $s_f^{\mathcal{R}}(1)=f^\ast s_f^{\mathcal{R}}(0)$.
The quantity $FKM(W, [\fraks], \xi, f, \gamma)$
is the signed counting of the parametrized moduli space for $s_f^{\mathcal{R}}$.
Let $s_g^{\mathcal{R}}$ denote the pull-back of the section $s_f^{\mathcal{R}}$ under $\sqcup_{t \in [0,1]}(h_t^\ast f^{-1})$.
This section satisfies that $s_g^{\mathcal{R}}(1)=g^\ast s_g^{\mathcal{R}}(0)$, and $FKM(W, [\fraks], \xi, g, \gamma)$ can be calculated by the signed counting of the parametrized moduli space for this section $s_g^{\mathcal{R}}$.
However, the pull-back under $\sqcup_{t \in [0,1]}(h_t^\ast (f^{-1})^\ast)$ gives rise to a diffeomorphism between these moduli spaces corresponding to 
$s_f^{\mathcal{R}}$ and $s_g^{\mathcal{R}}$, and this implies \eqref{eq: equality FKM gamma ab isptopy}.

To prove \eqref{eq: equality FKM gamma ab isptopy1}, as in the proof of \cref{lem: ab diff: vanishing for some gamma},
let $\tilde{\gamma}_\# : [0,1] \to \Xi^{\mathrm{cont}}(Y)$ denote the concatenated path of two copies of $\tilde{\gamma}$, which represents $\gamma^2 \in \pi_1(\Xi^{\mathrm{cont}}(Y))$.
Define a section $(s')_f^{\mathcal{R}} : [0,1] \to \tilde{\gamma}^\ast \mathcal{R}(Y, W, \fraks)$ as the pull-back section of $s_f^{\mathcal{R}}$ under $f$.
Similarly, define $(s')_g^{\mathcal{R}}$ as the pull-back section of $s_g^{\mathcal{R}}$ under $g $.
The signed counting of the concatenated section $s_f^{\mathcal{R}} \cup (s')_f^{\mathcal{R}}$ is the left-hand side of \eqref{eq: equality FKM gamma ab isptopy1}, and similarly for $g$.
Again the pull-back under $\sqcup_{t \in [0,1]}(h_t^\ast (f^{-1})^\ast)$ gives rise to a diffeomorphism between these moduli spaces corresponding to 
$(s')_f^{\mathcal{R}}$ and $(s')_g^{\mathcal{R}}$, and together with the diffeomorphism for $s_f^{\mathcal{R}}$ and $s_g^{\mathcal{R}}$ considered above, we obtain a diffeomorphism between these moduli spaces corresponding to $s_f^{\mathcal{R}} \cup (s')_f^{\mathcal{R}}$ and $s_g^{\mathcal{R}} \cup (s')_g^{\mathcal{R}}$.
Thus we have \eqref{eq: equality FKM gamma ab isptopy1}. 
\end{proof}

\begin{prop}
\label{lem: ab diff: additive}
Let $f, f' \in \Diff(W, [\fraks], \del)$ and $\gamma \in \pi_1(\Xi^{\mathrm{cont}}(Y))$.
Then we have
\begin{align}
\label{eq: ab diff additive}
\begin{split}
FKM(W, [\s], \xi, f \circ f', \gamma)
&= FKM(W, [\s], \xi, f)+FKM(W, [\s], \xi, f', \gamma)\\
&= FKM(W, [\s], \xi, f, \gamma)+FKM(W, [\s], \xi, f')
\end{split}
\end{align}
in $\Z_2$.
Moreover, if all of the diffeomorphisms in \eqref{eq: ab diff additive} and $\gamma$ induce the trivial monodromy on $\Lambda(W,\mathfrak{s},\xi)$, then the equalities \eqref{eq: ab diff additive} hold over $\Z$.
\end{prop}

\begin{proof}
Denote by $\tilde{\gamma}_\mathrm{const}$ the constant path at $\xi$ in $\Xi^{\mathrm{cont}}(Y)$ and set $\gamma_\mathrm{const}=[\tilde{\gamma}_\mathrm{const}] \in \pi_1(\Xi^{\mathrm{cont}}(Y))$.
By definition, we have
\[
FKM(W, [\s], \xi, f, \gamma_\mathrm{const})
= FKM(W, [\s], \xi, f).
\]

Pick a generic element $a \in \mathcal{P}(Y, W, \fraks,\xi)$.
Take a path $a$ between $a$ and $f^\ast a$ in $\mathcal{P}(Y, W, \fraks, \xi)$.
Such a path can be thought of as a section 
\[
s : [0,1] \to \tilde{\gamma}_{\mathrm{const}}^\ast\mathcal{P}(Y, W, \fraks) \cong [0,1] \times \mathcal{P}(Y, W, \fraks,\xi).
\]
Pick a loop $\tilde{\gamma}$ that represents $\gamma$, and take a generic section
\[
s' : [0,1] \to \tilde{\gamma}^\ast\mathcal{P}(Y, W, \fraks)
\]
so that $s'(0)=f^\ast a$ and $s'(1)=(f \circ f')^\ast a = f'^\ast(f^\ast a)$.
Then $FKM(W, [\s], \xi, f, \gamma_\mathrm{const})$ is the algebraic count of the parametrized moduli space for $s$, and $FKM(W, [\s], \xi, f', \gamma)$ is that for $s'$.
Thus the algebraic count of the parameterized moduli space for the path that is obtained by concatenating $s$ and $s'$ is the right-hand side of \eqref{eq: ab diff additive}.

On the other hand, parametrizing intervals, we can regard the concatenation of the sections $s$ and $s'$ as a section
\[
s\cup s' : [0,1] \to \tilde{\gamma}^\ast\mathcal{P}(Y, W, \fraks)
\]
such that $(s\cup s')(0)=a$ and $(s\cup s')(0)=(f \circ f')^\ast a$.
Therefore the algebraic count of the parameterized moduli space for $s \cup s'$ is the left-hand side of \eqref{eq: ab diff additive}.
This completes the proof of the first equality. 
The second equality follows just by a similar argument.
\end{proof}

\section{Several vanishing results}
\label{sectionSeveral vanishing results}

In this section, we prove several vanishing results for both of the families Seiberg--Witten invariant and Kronheimer--Mrokwa invariant.

Before stating several vanishing results, let us introduce a notion of {\it strong L-space} for convenience. 

\begin{defn}
A rational homology 3-sphere $Y$ is a {\it strong L-space} if there exists a Riemann metric $g$ on $Y$ such that there is no irreducible solutions to the Seiberg--Witten equation on $(Y, g, \s)$ for a Spin$^c$ structure $\s$ on $Y$. 
\end{defn}

Note that all strong L-spaces are L-space. However, the authors do not know whether the converse is true or not. 

We also recall the families Seiberg--Witten invariant for diffeomorphisms following \cite{Rub98}.
Let $X$ be a closed oriented smooth 4-manifold with $b^+_2(X)>2$ and $\s$ be a Spin$^c$ structure on $X$.
Fix a homology orientation of $X$.
Let $f : X \to X$ be an orientation-preserving diffeomorphism of $X$ such that $f^\ast\s=\s$ (precisely, $f^\ast\s$ is isomorphic to $\s$).
Then we can define a numerical invariant $FSW(X,\s,f)$, which takes value in $\Z$ if $f$ preserves the homology orientation, and which takes value in $\Z_2$ if $f$ reverses the homology orientation.
It is valid essentially only when the formal dimension of $\s$ is $-1$: otherwise the invariant $FSW(X,\s,f)$ is just defined to be zero.

\subsection{A family version of the vanishing result for embedded submanifolds} 
\subsubsection{Embeddings of 3-manifolds}
We first prove a family version of the vanishing result for embedded 3-manifolds.
For the original version, see \cite{Fr10} for example. 
\begin{thm}\label{Froyshov family}
Let $(X, \s)$ be a closed Spin$^c$ 4-manifold with $b^+_2(X)>0$ and $Y$ be a closed oriented 3-manifold.  
\begin{itemize}
    \item[(i)] 
    Suppose there is a smooth embedding $i : Y \to X$. 
Let $(Y, g )$ be a strong L-space for some metric $g$ on $Y$ and $f$ be an orientation-preserving self-diffeomorphism of $X$ such that 
\begin{itemize}
\item[$\circ$] $f^* \s = \s$, 
    \item[$\circ$] $f ( i (Y))  = i (Y)$ as subsets of $X$, and 
    \item[$\circ$] $f : (i (Y), i_* g) \to (i (Y), i_* g)$ is an isometry.  
\end{itemize}
%$f(Y')$ is smoothly isotopic to $Y'$%
We also impose that the map 
\[
H^2 ( X; \Q ) \to H^2(Y; \Q) 
\]
is non-zero. Then we have 
\[
 FSW(X, \s, f)= \begin{cases} 0 \in \Z \text{ if $f$ does not flip the homology orientation of $X$,} \\
  0 \in \Z_2 \text{ if $f$ flips the homology orientation of $X$.} \end{cases}
 \]
 
\item[(ii)] Suppose $X$ contains an essentially embedded smooth surface $S$ with non-zero genus and zero self-intersection that violates the adjunction inequality, i.e. we have
\[
2g(S)-2 < |\left<c_1(\fraks), [S]\right>|.
\]
Also, the normal sphere bundle $\partial \nu(S)$ is supposed to be $f|_{\partial \nu(S)}=\id_{\partial \nu(S)}$. 
 %then for a diffeomorphism $f$ on $X$ such that $f|_{\partial (\nu (S))}$ is smoothly isotopic to identity map on $\partial (\nu (S))$, we have 
 Then, we have 
\[
 FSW(X, \s, f)=  \begin{cases} 0 \in \Z \text{ if $f$ does not flip the homology orientation of $X$,} \\
  0 \in \Z_2 \text{ if $f$ flips the homology orientation of $X$.} \end{cases}
 \]
 \end{itemize}

\end{thm}
Note that \cref{Froyshov family}(ii) also follows from \cite[Theorem 1.2]{Ba20}. 
In the proof, we mainly follow the original Fr\o yshov's argument which uses a neck stretching argument and non-exact perturbations and Kronheimer-Mrowka's proof of the Thom conjecture \cite{KM94}.

\begin{proof}
We first prove (i). Since the proof of (ii) is similar to that of (i), we only write a sketch of proof of (ii). Because the family Seiberg--Witten invariant is an isotopy invariant, we can assume that $f$ can be a product in a neighborhood $N$ of $Y'$ which is isometry with respect to $g$.
Now, we consider the family version of the moduli space 
\[
\mathcal{PM} (X, \s, f)   \to S^1
\]
as a mapping torus of the moduli space over $I=[0,1]$: 
\[
\bigcup_{t\in [0,1]} \mathcal{M}(X, \s, g_t )  \to I,  
\]
where $g_t$ is a smooth 1-parameter family of metrics such that 
\begin{itemize}
    \item $ g_1 = f^*g_0$ and 
    \item for any $t\in [0,1]$, we have $g_t|_{N} = g + dt^2$.     
\end{itemize}
Here we take a metric $g$ so that there is no irreducible solution to the Seiberg--Witten equation on $ Y$. 
By assumption, we have a trivialization of  the family $E_f  \to S^1 $ obtained as the mapping torus of $f$ near $N$: 
\[
 E_{f |_{i (Y)} }  \subset E_f,  
\]
where  $E_{f|_{i (Y)}}$  us the mapping torus of $f|_{i (Y)}$. 

Now, near $E_{f |_{i (Y)} }$, we consider a neck stretching argument. We consider a family of metrics $g_{t,s}$ parametrized by $s \in [0, \infty)$ satisfying the following conditions: 
\begin{itemize}
    \item $g_{t,0} =  g_t$, 
    \item as Riemann manifolds, $(E_{f |_{i (Y)} }, g_{t,s}|_{N}) = (  [0,s+1]\times Y, g + dt^2) $, 
    \item outside of $ E_{f (i (Y))} $ in $E_f$, the metric $g_{t,s}$ coincides with $g_t$.
\end{itemize}
By the assumption $\operatorname{Im } (H^2 ( X; \Q ) \to H^2(Y; \Q) ) \neq 0$, we take a closed 2-form $\eta$ on $X$ such that 
\[
0\neq [\eta|_{Y} ] \in H ^2 (Y; \R).  
\]
Then, we conisder the perturbation of the family Seiberg--Witten equation on $E_f$ using $\eta$: 
\[
\begin{cases}\label{SW3}
   F^+_{A^t_t} + \sigma (\Phi_t, \Phi_t) = \epsilon \eta^+   \\ 
   D_{A_t} \Phi_t =0
\end{cases}
\]
for a small $\epsilon>0$. We take the $\epsilon$ so that there is no solution to the $\epsilon\eta$-perturbed Seiberg--Witten equation on $Y$ with respect to $(g', \s|_{Y})$: 
\[
\begin{cases}
   F_{B^t} + \sigma (\phi , \phi ) = \epsilon \eta|_{Y}  \\ 
   D_{B} \phi =0. 
\end{cases}
\]
Suppose $FSW(X, \s, f) \neq 0$. Put $\eta_0 := \epsilon \eta$. 
Now we take an increasing sequence $s_i \to \infty$. Then, since $FSW(X, \s, f) \neq 0$, there is a sequence of solutions $(A_i, \Phi_i)$ to the Seiberg--Witten equation with respect to $g_{t_i, s_i}$ for some $t_i \in [0,1]$.  
\begin{claim1}\label{uniform bound of energy}
We claim that 
\[
\sup_{i\in \Z_{>0}} \mathcal{E}^{top}_{\eta_0 , g_{t_i, s_i} } ( (A_i, \Phi_i)|_{E_{f|_{i (Y)} }} )  < \infty, 
\]
where, for a Spin$^c$ 4-manifold $W$ with boundary, the topological energy perturbed by $\eta_0$ is defined to be 
\[
\mathcal{E}_{\eta_0}^{top}(A, \Phi ) :=  \frac{1}{4} \int_W (F_{A^t}-4 \eta_0)  \wedge  (F_{A^t}-4 \eta_0) 
- \int_{\partial W } \langle \Phi|_Y, D_B \Phi|_Y \rangle.  
\]
%is the perturbed topological energy introduced in \cite[Page 593, Subsection 29.1]{KM07}.

\end{claim1}
\begin{proof}[Proof of \cref{uniform bound of energy}]
For a Spin$^c$ 4-manifold $W$ with boundary, define the perturbed analytical energy by 
\[
\mathcal{E}_{\eta_0}^{an}(A, \Phi)  := \frac{1}{4} \int_W|F_{A^t}-4 \eta_0|^2 + \int_W | \nabla_A \Phi |^2 + \frac{1}{4} \int_W (| \Phi|^2 + (s/2))^2
\]
\[
- \int_W \frac{s^2}{16} + 2 \int_W \langle \Phi, \rho (\eta_0) \Phi\rangle- \int_{\partial W } (H/2) |\Phi|^2, 
\]
where $s$ is the scalar curvature and $H$ is the mean curvature. 
It is proven that if $(A, \Phi)$ is a solution to the Seiberg--Witten equation perturbed by $\eta_0$, the equality 
\[
\mathcal{E}_{\eta_0}^{an}(A, \Phi) = \mathcal{E}_{\eta_0}^{top}(A, \Phi ). 
\]
holds (\cite[(29.6), Page 593]{KM07}). Since $X$ is closed, we know that, on $X$,
\[
\sup_{i\in \Z_{>0}} \mathcal{E}_{\eta_0}^{an}(A_i, \Phi_i )=\sup_{i\in \Z_{>0}} \mathcal{E}_{\eta_0}^{top}(A_i, \Phi_i ) < \infty . 
\]
On the other hand, we have 
\[
 \mathcal{E}_{\eta_0}^{an}(A_i, \Phi_i ) = \mathcal{E}_{\eta_0}^{an}( (A_i, \Phi_i)|_{E_{f |_{i (Y)} }}  ) + \mathcal{E}_{\eta_0}^{an}( (A_i, \Phi_i)|_{ (E_{f |_{i (Y)} }) ^c }  ) . 
\]
Since $X$ is compact, $\eta_0$ is bounded and $g_{t,s} $ is a compact family of metrics, 
we have lower bounds 
\[
-\infty < \inf_{i\in \Z_{>0}} \mathcal{E}_{\eta_0}^{an}( (A_i, \Phi_i)|_{ (E_{f |_{i (Y)} }) ^c } (A_i, \Phi_i ). 
\]
So, we see 
\[
\sup_{i\in \Z_{>0}} \mathcal{E} ^{top}_{\eta_0, g_{t_i, s_i} } ( (A_i, \Phi_i)|_{E_{f |_{i (Y)} }} ) = \sup_{i\in \Z_{>0}} \mathcal{E} ^{an}_{\eta_0, g_{t_i, s_i} } ( (A_i, \Phi_i)|_{E_{f |_{i (Y)} }} ) < \infty. 
\]
\end{proof}
By taking a subsequence, we can suppose $t_i \to t_\infty \in [0,1]$. 
As in the proof of \cite{KM97}, we can also take a subsequence of $(A_i, \Phi_i)$ so that 
\[
\mathcal{E}^{top}_{g_{t_i, s_i} }  ( (A_i, \Phi_i)|_{[l_i, l_i+1] \times Y \subset  E_{f |_{i (Y)} }} ) \to 0. 
\]

So, as the limit of $(A_i, \Phi_i)|_{[l_i, l_i+1] \times Y \subset  E_{f |_{i (Y)} }}$, we obtain an (perturbed) energy zero solution $(A_\infty , \Phi_\infty)$ on $[0,1]\times Y $. 
By considering the temporal gauge, this gives a solution to \eqref{SW3}. This gives a contradiction. 

The proof of (ii) is similar to (i). 
Let $Y$ be the product $S^1 \times S$. We take a Riemman metric $g$ on $Y$ forms
\[
g = dt^2 + g_0,  
\]
where $t$ is a coordinate of $S^1$ and $g_0$ has a constant scalar curvature. Kronheimer--Mrowka's argument in \cite{KM07} implies that if there is a solution to the Seiberg--Witten equation on $Y$ with respect to $(g, \s|_Y)$, then the adunction inequality 
\[
|\langle c_1 (\s) , [S] \rangle| \leq 2 g(S) -2
\]
holds(\cite[Proposition 40.1.1]{KM07}) when $g(S)>0$. Since we now are assuming the opposite inequality $|\langle c_1 (\s) , [S] \rangle| < 2 g(S) -2$, it is sufficient to get a solution to the Seiberg--Witten equation on $Y$ with respect to $(g, \s|_Y)$. The remaining part of the proof is the same as that of (i), i.e. we consider the neck stretching argument near $\partial (\nu (S))$.  This completes the proof. 
\end{proof}

We also provide a version of \cref{Froyshov family} for the Kronheimer--Mrowka invariant. 
\begin{thm}\label{Froyshov family1}
Let $(W, \s)$ be a compact Spin$^c$ 4-manifold with contact boundary $(Y,\xi)$. 
\begin{itemize}
    \item[(i)]
 Suppose there is a smooth embedding $i : Y \to W$ where $(Y, g )$ is a strong L-space for some metric $g$ on $Y$. 
Let $f$ be a self-diffeomorphism of $W$ such that 
\begin{itemize}
\item $f|_{\partial W}$ is the identity, 
\item $f^* \s = \s$, 
\item $i(Y)$ separates $W$, 
    \item $f ( i (Y))  = i (Y)$ as subsets of $W$, and 
    \item $f : (i (Y), i_* g) \to (i (Y), i_* g)$ is an isometry.  
\end{itemize}
%$f(Y')$ is smoothly isotopic to $Y'$%
We also impose that the map 
\[
H^2 ( W; \Q ) \to H^2(Y; \Q) 
\]
is non-zero. Then we have 
\[
 FKM(X, \s,\xi,  f)= \begin{cases} 0 \in \Z \text{ if the action of $f$ on $\Lambda (Y, \s, \xi ) $ is trivial,}  \\
 0 \in \Z_2 \text{ if the action of $f$ on $\Lambda (Y,\s,  \xi ) $ is non-trivial. } \end{cases}  
 \]
 \item[(ii)] Suppose $W$ includes an essentially embedded smooth surface $S$ with non-zero genus that violating adjunction inequality and $[S]^2=0$.
 Then for a homologically trivial diffeomorphism $f$ on $W$ fixing the boundary pointwise such that $f( \partial (\nu (S)) )$ is smoothly isotopic (rel $\partial$) to $\partial (\nu (S))$, we have 
\[
 FKM(X, \s,\xi,  f)= \begin{cases} 0 \in \Z \text{ if the action of $f$ on $\Lambda (Y, \s, \xi ) $ is trivial,}  \\
 0 \in \Z_2 \text{ if the action of $f$ on $\Lambda (Y, \s, \xi ) $ is non-trivial.} \end{cases}
 \]
 \end{itemize}

\end{thm}

\begin{proof}
The proof is essentially parallel to that of \cref{Froyshov family}
An unparametrized version of \cref{Froyshov family1} is proven in \cite[Theorem 1.19 (i)]{IMT21}. We use the same perturbations used in the proof of \cite[Theorem 1.19 (i)]{IMT21}. 
The only difference between the proofs of \cref{Froyshov family} and \cref{Froyshov family1} appears in the proof of \cref{uniform bound of energy}. Note that when we consider the Seiberg--Witten equation on 4-manifolds with conical end, there is no global notion of energy: on the interior, we have usual topological and analytical energies, and on the cone, we have the symplectic energy. We combine these two energies to show \cref{uniform bound of energy}. 
For that part, see that the proof of \cite[Lemma 4.6]{IMT21}. It is easy to see the proof of \cite[Lemma 4.6]{IMT21} can be also applied to our family case. 
\end{proof}

\subsubsection{Embeddings of surfaces}

We also prove \cref{Froyshov family1} for embedded surfaces to find exotic embeddings of surfaces into 4-manifolds with boundary. 
Let $\Sigma_g$ denote a closed oriented surface of genus $g$.

\begin{thm}
\label{KM family adjunction}
Let $(W, \s)$ be a compact Spin$^c$ 4-manifold with contact boundary $(Y,\xi)$. Let $i: \Sigma_g \to W$ be a smooth embedding satisfying one of the following two conditions: 
\begin{itemize}
    \item[(i)]  
\begin{itemize}
    \item $i_* [\Sigma _g]$ is non-torsion, and 
    \item  $g=0 $. 
\end{itemize}
\item[(ii)]
\begin{itemize}
 \item  $g > 0 $ and  
    \item the adjunction inequality for $(i(\Sigma_g), \s) $ is violated.  
\end{itemize}
\end{itemize}
For any diffeomorphism $f$ on $W$ fixing the boundary pointwise and preserving the Spin$^c$ structure $\s$ such that $i$ is smoothly isotopic (rel $\partial$) to $f\circ i $, we have 
\[
 FKM(X, \s,\xi,  f)= \begin{cases} 0 \in \Z \text{ if the action of $f$ on $\Lambda (Y, \s, \xi ) $ is trivial},  \\
 0 \in \Z_2 \text{ if the action of $f$ on $\Lambda (Y, \s, \xi ) $ is non-trivial}.
 \end{cases}  
 \]
Moreover, in the case (i), we can replace the assumption that $i$ is isotopic to $f \circ i$ with the assumption that the image $i(\Sigma_g)$ is smoothly isotopic to the image of $f^2\circ i(\Sigma_g)$.
%Suppose $W$ includes a smoothly and essentially embedded surface $S$ violating adjunction inequality and with non-zero genus and $[S]^2=0$.
\end{thm}

\begin{proof}
The proof is also based on the neck stretching argument near a normal neighborhood of $i(\Sigma_g)$. The closed case is treated in \cite{Ba20}. Since there is no big difference between the proof of \cref{KM family adjunction} and \cite[Theorem 1.2]{Ba20}, we omit the proof. 
\end{proof}

%\begin{rem}
%The same concolusion holds for the usual family Seiberg--Witten invariant. 
%\end{rem}

\subsection{A fiberwise connected sum formula}
We first review a fiberwise connected sum formula which is first proven in \cite[Theorem 7.1]{Ko21}. Note that in \cite[Theorem 7.1]{Ko21}, the case that the connected sum along $S^3$ is treated. We generalize the vanishing result in \cite[Theorem 7.1]{Ko21} to the result on the connected sums along any strong L-spaces with $b_1=0$.   In the context of Donaldson's theory, the connected sum result is written in \cite[Theorem 3.3]{Ru99}. 
For a 4-manifold $X$, $X^\circ$ denotes a compact punctured 4-dimensional submanifold of $X$. 
\begin{thm}\label{conn sum vanish family}%\cite[Theorem 7.1]{Ko21}
Let $(Y,h)$ be a strong L-space with $b_1(Y)=0$.
Let $(X, \s), (X', \s')$ be two compact Spin$^c$ 4-manifolds with $b^+_2> 1$, $\partial X= -Y$,  $\partial X'=Y$, and suppose there is an isomorphism $\s|_{\partial X} =\s'|_{-\partial X'}$.
Let $f$ be an orientation-preserving diffeomorphism on the closed 4-manifold $X^\# :=X \cup_Y X'$ such that 
\begin{itemize}
    \item[(i)] the diffeomorphism $f$ is smoothly isotopic to a connected sum $f' \#_Y g'$ on $X \cup_Y X'$ of  diffeomorphisms $f'$ and $g'$ on $X$ and $X'$ which are, near the boundary of $X$ and $X'$, the product of an isometry of $(Y,h)$ with the identity on $[0,1]$, and 
    \item[(ii)] the diffeomorphism $f$ preserves the Spin$^c$-structure  $\s \# \s'$ on $X^\#$.
\end{itemize}
Then we have 
\[
 FSW(X^\# , \s \# \s', f)= \begin{cases} 0 \in \Z \text{ if $f$ does not flip the homology orientation of $X$,} \\
  0 \in \Z_2 \text{ if $f$ flips the homology orientation of $X$.} \end{cases}
\]

\end{thm}
\begin{proof}
For completeness, we give a sketch of the proof. 
Let $Y$ be a rational homology 3-sphere with a Riemannian metric $h$ such that there is no irreducible solution to the Seiberg--Witten equation on $Y$ with respect to $(\s|_Y, h)$. 
Suppose $X^\circ$  and $(X')^\circ$ are 4-manifolds with boundary $-Y$ and $Y$ and define 
\[
X \# X' = X^\circ \cup _Y  (X')^\circ. 
\]
Since the family Seiberg--Witten invariant $FSW$ is an isotopy invariant, we can assume that $f$ is described as the connected sum $f' \# g'$ on $X \#_Y X'$ of diffeomorphisms $f'$ and $g'$ on $X^\circ$ and $(X')^\circ$ which are the identity near the boundary of $X^\circ$ and $(X')^\circ$. With respect to the Spin$^c$ structure $\s \# \s'$, the gluing theory enables us to construct a diffeomorphism 
\begin{align}\label{gluing diffeo}
\mathcal{PM}(( X^\circ \cup [0, \infty) \times Y ), f')  \times_{U(1)} \mathcal{PM}(( (-\infty, 0] \times Y )\cup ( X' )^\circ , g') \to  \mathcal{PM}(X\# X' , f' \# g' ),  
\end{align}
where the spaces $\mathcal{PM}(X^\circ \cup [0, \infty) \times Y ), f')$ and 
$\mathcal{PM}(( (-\infty, 0] \times Y )\cup( X' )^\circ , g')$ are parametrixed moduli spaces on the cylindrical-end Riemannian 4-manifolds $X^\circ \cup [0, \infty) \times Y$ and $((-\infty, 0] \times Y ) \cup ( X' )^\circ  $ assymptotically the flat redusible solutions. Here we used the product metric $h + dt^2$ on the cylinder part. 
Precisely, we need to use weighted $L^2_k$ norms to obtain Fredholm properties of linearized Seiberg--Witten equations with slice. From index calculations, we have
\[
\dim \mathcal{PM}(X^\circ \cup [0, \infty) \times Y , f') = \dim \mathcal{PM}(X, f')
\]
and 
\[
\dim \mathcal{PM}((-\infty, 0] \times Y  \cup ( X' )^\circ ), g')=  \dim \mathcal{PM}(X', g'). 
\]
So we have 
\[
\dim \mathcal{PM}(X, f') +  \dim \mathcal{PM}(X', g') + 1 = \dim \mathcal{PM}(X \# X',f) . 
\]
Note that $\dim \mathcal{PM}(X \# X',f)$ is $0$ if we assume the family Seiberg--Witten invariant of $(X\# X', \s\#\s',  f)$ is non-zero. (Otherwise, we define $0$ as the family Seiberg--Witten invariant in this paper.  )
Thus, one of $\dim \mathcal{PM}(X, f')$ and $\dim \mathcal{PM}(X', g')$ is negative. This implies one of $\mathcal{PM}(X, f')$ and $\mathcal{PM}(X', g')$ is empty since they are assumed to be regular. This completes the proof. 
\end{proof}

In \cref{conn sum vanish family}, we assumed that $f$ is isotoped into $f'$ so that $f'|_Y$ is an isometry. However, if $Y$ admits a positive scalar curvature metric, then the following stronger result holds: 

\begin{thm}\label{conn sum vanish family:PSC}
Let $Y$ be a rational homology 3-sphere with a positive scalar curvature metric and $b_1(Y)=0$.
Let $(X, \s), (X', \s')$ be two compact Spin$^c$ 4-manifolds with $b^+_2> 1$, $\partial X= -Y$ and $\partial X'=Y$. Suppose $\s|_{\partial X} = \s'|_{-\partial X'}$.
Let $f$ be an orientation-preserving diffeomorphism on the closed 4-manifold $X^\# :=X \cup_Y X'$ such that 
\begin{itemize}
    \item[(i)] the diffeomorphism $f$ is smoothly isotopic to a connected sum $f' \#_Y g'$ on $X \cup_Y X'$ of  diffeomorphisms $f'$ and $g'$ on $X$ and $X'$ so that $f(Y)=Y$,
    \item[(ii)] the diffeomorphism $f$ preserves the Spin$^c$-structure $\s \# \s'$ obtained as the connected sum of $\s$ and $\s'$ along $Y$.
\end{itemize}
Then we have 
\[
 FSW(X^\# , \s \# \s', f)= \begin{cases} 0 \in \Z \text{ if $f$ does not flip the homology orientation of $X$,} \\
  0 \in \Z_2 \text{ if $f$ flips the homology orientation of $X$.} \end{cases}
\]
\end{thm}

\begin{proof}
The proof is similar to that of \cref{conn sum vanish family}. Instead of assuming the isometric property of diffeomorphisms, we use the contractivity of the space of positive scalar curvature metrics proven in \cite{BaKl19}.  Let us explain how to take a fiberwise Riemann metric to obtain a diffeomorphism corresponding to \eqref{gluing diffeo}. Because the family Seiberg--Witten invariant is an isotopy invariant, we can assume that $f$ is a product in a neighborhood $N$ of $Y$ preserving the level of $N=[0,1]\times Y$.
Now, we shall consider the parametrized moduli space
\[
\mathcal{PM} (X, \s, f)   \to S^1,
\]
regarded as the quotient the moduli space over $[0,1]$: 
\[
\bigcup_{t\in [0,1]} \mathcal{M}(X, \s, g_t )  \to [0,1],  
\]
where $g_t$ is a smooth 1-parameter family of metrics such that 
\begin{itemize}
    \item $ g_1 = f^*g_0$, 
    \item for any $t\in [0,1]$, we have $g_t|_{N} = h_t + dt^2$ for a smooth 1-parameter family of metrics $h_t$ on $Y$, 
    \item $h_0$ is a positive scalar curvature metric on $Y$.
\end{itemize} 
By the connectivity of the space of positive scalar curvature metrics proven in \cite{BaKl19}, we can take $h_t$ so that $h_t$ is a positive scalar curvature metric for every $t\in [0,1]$. Under these settings, we obtain a diffeomorphism between moduli spaces of the form \eqref{gluing diffeo}, and the rest of the proof is the same as that of \cref{conn sum vanish family}.
\end{proof}

We next show a vanishing result of the Kronheimer--Mrowka invariant similar to \cref{conn sum vanish family}.  

\begin{thm}\label{conn sum vanish family1}%\cite[Theorem 7.1]{Ko21}
Let $Y$ be a strong L-space with $b_1(Y)=0$.
Let $(W, \s), (X', \s')$ be Spin$^c$ 4-manifolds with $b^+_2> 1$, $\partial W= -Y \cup Y'$ and $\partial X'=Y$ for an oriented 3-manifold $Y'$.
Suppose $Y'$ is equipped with a contact structure $\xi$ such that $\s_\xi = \s|_{Y'}$ and $\s|_{Y}= \s'|_{Y}$. 
Let $f$ be a diffeomorphism on the closed 4-manifold $W^\# :=W \cup_Y X'$ such that 
\begin{itemize}
    \item[(i)] the diffeomorphism $f$ is smoothly isotopic to $f' \cup g'$ on $W \cup_Y X'$ of  diffeomorphisms $f'$ and $g'$ on $X$ and $X'$ which are isometry (not necessarily identity) near the boundary of  $W$ and $X'$ with respect to metric $h$ on $Y$, and 
    \item[(ii)] the diffeomorphism $f$ preserves the Spin$^c$-structure $\s\# \s'$ on $W^\#$.
\end{itemize}
Then we have 
\[
 FKM(W^\# , \s \# \s', f)= \begin{cases} 0 \in \Z \text{ if $f$ does not flip the elements in $\Lambda(W^\#, \s \# \s', \xi)$,} \\
  0 \in \Z_2 \text{ if $f$ flips the homology orientation of $X$.} \end{cases} 
\]

\end{thm}

\begin{proof}
The proof is completely the same as that of \cref{conn sum vanish family}. 
\end{proof}

\subsection{A fiberwise blow up formula}
We first review a fiberwise blow up formula proven in \cite{BK20} for the family Seiberg--Witten invariant. 
\begin{thm}[\cite{BK20}]
Let $(X, \s), (X', \s')$ be closed Spin$^c$ 4-manifolds with $b^+_2(X)> 1$ and $b^+_2(X')=1$. Let $f$ be an orientation-preserving diffeomorphism of the closed 4-manifold $X \# X'$. Suppose that the formal dimension of the Seiberg--Witten moduli spaces for $(X, \s)$ and $( X' , \s')$ are $0$ and $-2$ respectively.

\begin{itemize}
    \item[(i)] Let $f$ be a homologically trivial diffeomorphism  on $X \# X'$ such that $f$ is smoothly isotopic to a connected sum $f' \# g'$ of $f'= \id_X$ and a homologically trivial diffeomorphism $g'$ of $X'$, which is the identity near boundary.
Then we have 
\[
 FSW(X \# X', \s \# \s', f)= 0 \in \Z . 
\]
\item[(ii)] Let $f$ be a self-diffeomorphism  on $X \# X'$ such that 
\begin{itemize}
\item $f$ preserves the Spin$^c$ structure $\s \# \s'$, 
    \item $f$ is smoothly isotopic to a union $f' \cup g'$ on $X \# X'$ of the identity $f'= \id_X$ of $X$ and a diffeomorphism $g'$ on $X'$ which is the identity near $\partial X'$, and 
    \item $g'$ reverses the homology orientation of $X'$.   
\end{itemize}
Then we have 
\[
 FSW(X \# X' , \s \# \s', f)= SW (X, \s) \in \Z_2, 
\]
where the right-hand side is the mod 2 Seiberg--Witten invariant. 
\end{itemize}
\end{thm}

\begin{rem}
Note that \cite[Theorem 1.1]{BK20} only treats the connected sum along $S^3$, but there is no essential change when we extend their result to the case of the sums along any strong L-space with $b_1=0$. 
\end{rem}

%\section{Fiberwise connected sum formula}
%\label{section Fiberwise connected sum formula}
Now, we state a fiberwise blow up formula for the families Kronheimer--Mrowka invariant. 

\begin{thm}\label{fiberwise blow up KM}
 Let $(W, \s), (X, \s')$ be Spin$^c$ 4-manifolds with $b^+_2(X)=1$, $\partial W= - Y$ and $\partial X=\emptyset$ for an oriented 3-manifold $Y$.
Suppose $Y$ is equipped with a contact structure $\xi$ such that $\s_\xi = \s|_{Y}$, $d(W, \s, \xi)=0$ and $d(X,\s')=-2$, where $d(X,\s')$ is the virtual dimension of the Seiberg--Witten moduli space for $(X,\s')$. 
Let $f$ be a diffeomorphism on the 4-manifold $X \# W$ such that $f$ preserves the Spin$^c$ structure $\s \# \s'$ obtained as the connected sum of $\s$ and $\s'$. 
\begin{itemize}
    \item[(i)] Suppose $f$ is smoothly isotopic (rel $\partial$) to a union $f' \cup g'$ on $X \# W$ of the identity $f'=\id_X$ of $X$ and a homologically trivial diffeomorphism $g'$ of $W$ which is the identity near boundary. 
Then we have 
\[
 FKM(X \# W, \s \# \s',\xi,  f)= 0 \in \Z . 
\]
\item[(ii)] Suppose 
 $f$ is smoothly isotopic to a union $f' \cup g'$  of the identity $f'=  \id_X$ of $X$ and a diffeomorphism $g'$ of $W$ which is the identity near boundary and $g'$ {\it reverses} the homology orientation of $X$.   

Then we have 
\[
 FKM(X \# W, \s \# \s', \xi, f)= \frakm(W, \s, \xi) \in \Z_2. 
\]
\end{itemize}

\end{thm}
\begin{proof}
The proof is essentially the same as that for a similar gluing formula for the families Seiberg--Witten invariant of families of closed 4-manifolds \cite[Theorem 1.1]{BK20}, and we omit the proof.
\end{proof}

\begin{rem}
We also remark that \cref{fiberwise blow up KM} can easily be generalized to the case of the connected sum along a strong L-space. 
\end{rem}

As a special case, we give a certain fiberwise connected sum formula for a certain class of 4-manifolds parametrized by $S^1$.
%We assume the base space $B= S^1$. 

%Let $(W, \s)$ be a connected compact oriented $\Spinc$ 4-manifold with connected contact boundary $(Y, \xi)$ and $W \to E \to B$ a fiber bundle over a compact smooth base space $B$ whose fiber is $(W, \s, \xi)$. Suppose that $\s_{\xi} =  \s|_Y$ on each fiber. As the monodromy of $E$, we have a diffeomorphism 
%\[
%\phi : W \to  W 
%\]
%which fixes the $\spinc$ structure $\s$ and $\partial W$ pointwise.

Let $W$ be an oriented compact smooth 4-manifold with contact boundary $(Y,\xi)$ and with $b^+_2>1$.
Let $\fraks$ be a $\spc$ structure on $W$ of formal dimension $0$.
Let $N, \frakt, f_N$ be as in \cref{section: Construction of diffeomorphisms and non-vanishing results}.
Let us consider the manifold $W \# N$ obtained as the connected sum along $D^4$ in $N$ and the diffeomorphism $\id_W \# f_N$. 

\begin{prop}
\label{gluing}
Under the above notation, one has 
\[
FKM(W \# N, \fraks\#\frakt,\xi, \id \# f_N)  =  \frakm(W,\fraks, \xi)\in \Z_2
\]
\end{prop}
\begin{proof}
This is a corollary of \cref{fiberwise blow up KM}. 
\end{proof}

\section{Construction of diffeomorphisms and non-vanishing results}
\label{section: Construction of diffeomorphisms and non-vanishing results}

In this section, we construct an ingredient of desired exotic diffeomorphisms in our main theorems. 

First, we describe the setting we work on.
Set
\[
N:=\C P^2\# 2(-\C P^2)=\C P^2\# (-\C P^2_1)\# (-\C P^2_2).
\]
Let $\mathfrak{t}$ be a $\Spinc$ structure on $N$ such that each component of 
\[
c_1(\mathfrak{t})\in H^2(N)= H^2(\C P^2)\oplus H^2(-\C P^2_1)\oplus H^2(-\C P^2_2)
\]
is a generator of $H^2(\C P^2),  H^2(-\C P^2_1)$, and $H^2(-\C P^2_2)$.
Let $H, E_1, E_2$ be the generators of $H^2(\C P^2),  H^2(-\C P^2_1), H^2(-\C P^2_2)$, namely
\[
c_1(\mathfrak{t}) = H + E_1 +E_2.
\]
By abuse of notation, let $H, E_1, E_2$ denote also representing spheres of the Poincar{\' e} duals of these classes.
A diffeomorphism $f_N : N \to N$ satisfying the following properties is constructed in \cite[Proof of Theorem~3.2]{K19} (where the diffeomorphism is denoted by $f_0'$): 
\begin{itemize}
    \item $f_N$ fixes the isomorphism class of $\mathfrak{t}$, and 
    \item $f_N$ reverses orientation of $H^+(N)$. 
\end{itemize}
By isotopy, we may suppose also that $f_N$ fixes a 4-dimensional small disk $D^4$ in $N$.
\begin{comment}
A concrete construction of such $f$ is given as follows.
See \cite[Proof of Theorem~3.2]{K19} for the detail.
In general, for an embedded sphere $S$ in a 4-manifold $X$ with self-intersection number $\pm1$, one may define a self-diffeomorphism $\rho_S : X \to X$ called the {\it reflection}.
The embedded surface $S$ and $\rho_S$ are modeled by the sphere $\C P^1$ in $\C P^2$
 and the self-diffeomorphism of $\C P^2$ defined by $[z_0: z_1 : z_2] \mapsto [\bar{z}_0: \bar{z}_1 : \bar{z}_2]$.
 Again we may suppose that $\rho_S$ fixes an embedded $4$-disk by isotopy.
 Define $f := (\rho_H \# \rho_{E_1} \# \rho_{E_2}) \circ \rho_{H\#E_1\#E_2} : N \to N$.
 Then the action of $f$ on $H_2(N)$ is given by
 \[
 \begin{pmatrix}
-3 & 2 & 2\\
-2 & 1 & 2\\
-2 & 2 &1
\end{pmatrix}
 \]
with respect to the ordered basis $(H,E_1,E_2)$. 
Then one can easily check that $f$ satisfies the desired properties.
\end{comment}

We consider simply connected compact oriented 4-manifolds $W$ and $W'$ with common contact boundary $(Y,\xi)$.
Assume also that we have a diffeomorphism
    \[
    \psi : W\#N \to W' \# N
    \]
that satisfies the following:
If we decompose $H^2(W' \# N;\Z)$ into 
\[
H^2(W' \# N;\Z)
= H^2(W';\Z) \oplus H^2(N;\Z),
\]
the induced action $\psi^\ast : H^2(W' \# N;\Z) \to H^2(W' \# N;\Z)$ on $a+b \in H^2(W';\Z) \oplus H^2(N;\Z)$ is of the form 
\begin{align}
\label{eq: pull back spinc cal}
\psi^\ast(a+b) = h'(a)+b,
\end{align}
where
\[
h' : H^2(W';\Z) \to H^2(W';\Z)
\]
is an isomorphism.
%Set $h = \psi^\ast : H^2(W ;\Z) \to H^2(W;\Z)$.
%By Wall's theorem \cite[Theorem~2]{Wa64},
%there exists a self-diffeomorphism
%$\psi'' : W' \# N \to W' \# N$
%such that
%\[
%(\psi'')^\ast = (\psi'^{-1})^\ast \circ h
%: H^2(W' \# N;\Z) \to H^2(W' \# N;\Z).
%\]
%Set
 %   \[
  %  \psi := \psi'' \circ \psi' : W\#N \to W' \# N.
   % \]
%= (\psi''^{-1})^\ast\circ(\psi'^{-1})^\ast
%= h^{-1} \circ \psi'^{-1} \circ (\psi'^{-1})^\ast

Let $\fraks$ be a Spin$^c$ structure on $W$, and let $\fraks'$ be the Spin$^c$ structure on $W'$ determined by  $h'(c_1(\fraks')) = c_1(\fraks)$.
Then it follows from \eqref{eq: pull back spinc cal} that
\begin{align}
\label{eq: pull back spinc cal2}
(\psi^{-1})^\ast(c_1(\fraks)+c_1(\frakt))
= c_1(\fraks')+c_1(\frakt).
\end{align}

Define a self-diffeomorphism $f$ of $W\# N$ by 
\begin{align}\label{main diffeo f}
f :=  (\id_W \# f_N) \circ \psi^{-1} \circ  (\id_{W'} \# f_N^{-1})  \circ  \psi .
\end{align}
Note that $f$ is the identity on the boundary, while $\psi$ might not be.
Note also that $f$ acts trivially on the (co)homology of $W\#N$.
Indeed, we obtain from \eqref{eq: pull back spinc cal} that
\begin{align*}
    f^\ast 
    &= \psi^\ast \circ \diag(\id_{H^2(W')}, (f_N^{-1})^\ast)
    \circ (\psi^{-1})^\ast
    \circ \diag(\id_{H^2(W)}, f_N^\ast)\\
    &=\diag(h' \circ h'^{-1}, (f_N^{-1})^\ast \circ f_N^\ast) = \id.
\end{align*}

\begin{prop}
 \label{comp+gluing}
 Suppose that
 \[
 (\mathfrak{m}(W,\fraks,\xi),\mathfrak{m}(W',\fraks',\psi_*\xi) ) \equiv (1,0) \text{ or } (0,1) \in \Z_2 \times \Z_2.
 \]
Then we have 
 \[
 FKM( W\# N , \fraks\#\frakt, \xi, f) \equiv 1 \in \Z_2
 \]
for the above diffeomorphism $f$. 
 \end{prop}
 
\begin{proof}
It follows from a combination of the gluing formula \cref{gluing}, the additivity formula \cref{prop: additivity}, and \eqref{eq: pull back spinc cal2} that
\begin{align*}
\begin{split}
&FKM(W\# N,\fraks\#\frakt, \xi, f) \\
=& FKM(W\# N,\fraks\#\frakt, \xi, \id_{W}\#f_N) 
+ FKM(W\# N,\fraks\#\frakt, \xi, \psi^{-1} \circ (\id_{W'}\#f_N^{-1}) \circ \psi)\\
=& FKM(W\# N,\fraks\#\frakt, \xi, \id_{W}\#f_N)
+ FKM(W'\# N,\fraks'\#\frakt, \psi_\ast\xi, \id_{W'}\#f_N^{-1})\\
=& \frakm(W,\fraks,\xi) + \frakm(W',\fraks',\psi_\ast\xi)\\
=&0+1 =1 \text{\ in\ } \Z_2.
\end{split}
\end{align*}
 This completes the proof. 
\end{proof}

\begin{cor}
\label{cor: Ab Z summand}
For every non-zero integer $n \in \Z$, we have
\begin{align}
\label{eq: nonvan for any power}
FKM( W\# N , \fraks\#\frakt, \xi, f^n) \neq 
0 \in \Z.
\end{align}
Moreover, the mapping class of $f$ above generates a $\Z$-summand of the abelianization of
\[
\Ker(\pi_0(\Diff(W\#N, \del) \to \Aut(H_2(W\#N;\Z))).
\]
\end{cor}

\begin{proof}
This follows from \cref{comp+gluing} and
\cref{cor: summary of properties of FKM}.
\end{proof}

We can modify the above argument for the generalized families Kronheimer--Mrowka invariant with a loop in $\Xi^{\mathrm{cont}}(\Xi(Y))$.

\begin{lem}
\label{lem: vanishing for the ab diff}
Let $W$ be a compact oriented 4-manifold and with boundary $Y=\partial W$.
Let $\xi$ be a contact structure on $Y$.
Let $\Sigma$ be an embedded 2-sphere in the interior of  $W$ whose self-intersection is non-negative and whose homology class is non-torsion.
Let $N=\C P^2 \# 2(-\C P^2)$, $\s \in \Spinc(W\# N, \xi)$, and set $\frakt := \fraks|_N$. 
Let $f_0$ be a self-diffeomorphism of $N$ which preserves $\frakt$.
Then, for every $\gamma \in \pi_1(\Xi^{\mathrm{cont}}(Y))$, we have
\[
FKM(W \# N, [\s], \xi, \id_{W}\# f_0,\gamma)=0
\]
in $\Z_2$.
\end{lem}

\begin{proof}
Let us first consider the case with $[\Sigma]^2=0$. General case can be reduced to this case by standard argument using the blow-up formula \cref{nonfamily KM blowup} and the connected sum formula \cref{fiberwise blow up KM}. 
Neck stretching argument along the boundary $S^1\times S^2$ of a tubular neighbourhood of $\Sigma$, as in \cref{Froyshov family1}, gives the conclusion.
\end{proof}

\begin{prop}
 \label{comp+gluing with gamma}
Let $W$ be a compact oriented 4-manifold and with boundary $Y=\partial W$.
Let $\xi$ be a contact structure on $Y$.
Let $\Sigma$ be an embedded 2-sphere in the interior of  $W$ whose self-intersection is non-negative  and whose homology class is non-torsion.
Suppose that
\[
\mathfrak{m}(W',\fraks',\psi_*\xi) \equiv 1 \in \Z_2.
\]
Then, for the above diffeomorphism $f$ \cref{main diffeo f}, we have 
\begin{align}
\label{eq: intermediate FKMnonzero}
FKM( W\# N , \fraks\#\frakt, \xi, f, \gamma)
\neq 0 \in \Z_2
\end{align}
for every $\gamma \in \pi_1(\Xi^{\mathrm{cont}}(Y))$,
and
\begin{align}
\label{eq: distinct FKM}
FKM( W\# N , \fraks\#\frakt, \xi, f^{2n}, \gamma^2)
\neq FKM( W\# N , \fraks\#\frakt, \xi, f^{2n'}, \gamma^2) \in \Z
\end{align}
for every $\gamma \in \pi_1(\Xi^{\mathrm{cont}}(Y))$ and every distinct $n, n' \in \Z$.
\end{prop}

\begin{proof}
For every $\gamma$, it follows from the gluing formula \cref{gluing}, the additivity formula \cref{lem: ab diff: additive}, and \eqref{eq: pull back spinc cal2} that
\begin{align*}
\begin{split}
&FKM(W\# N,\fraks\#\frakt, \xi, f, \gamma) \\
=& FKM(W\# N,\fraks\#\frakt, \xi, \id_{W}\#f_N, \gamma) 
+ FKM(W\# N,\fraks\#\frakt, \xi, \psi^{-1} \circ (\id_{W'}\#f_N^{-1}) \circ \psi)\\
=& FKM(W\# N,\fraks\#\frakt, \xi, \id_{W}\#f_N, \gamma)
+ FKM(W'\# N,\fraks'\#\frakt, \psi_\ast\xi, \id_{W'}\#f_N^{-1})\\
=& 0+ \frakm(W',\fraks',\psi_\ast\xi) =1 \in \Z_2.
\end{split}
\end{align*}
Thus we have \eqref{eq: intermediate FKMnonzero}.
%Note that it follows from \eqref{eq: intermediate FKMnonzero} that
%\begin{align}
%\label{eq: consequence from eq: intermediate FKMnonzero}
%FKM(W\# N,\fraks\#\frakt, \xi, f, \gamma) \neq 0 \in \Z.
%\end{align}

Next, applying \cref{lem: ab diff: additive} inductively, for $n>0$ and every $\gamma$,
we obtain that
\begin{align*}
%\label{eq: inducitve FKM gamma}
\begin{split}
&FKM( W\# N , \fraks\#\frakt, \xi, f^{2n}, \gamma^2)\\
=& FKM( W\# N , \fraks\#\frakt, \xi, (f^2)^{n-1}, \gamma^2)
+ FKM( W\# N , \fraks\#\frakt, \xi, f^2)\\
=&\cdots \\
=& FKM( W\# N , \fraks\#\frakt, \xi, \id, \gamma^2)
+ nFKM( W\# N , \fraks\#\frakt, \xi, f^2)
\end{split}
\end{align*}
This combined with \eqref{eq: nonvan for any power} implies \eqref{eq: distinct FKM} for $n, n' \geq 0$ with $n\neq n'$.
A similar argument works for $n,n' \leq 0$ just by considering $f^{-1}$ in place of $f$.
%By \eqref{eq: intermediate FKMnonzero2}, we have that $FKM( W\# N , \fraks\#\frakt, \xi, f) \neq 0$ in $\Z$, and thus there exists $N > 0$ such that
%\[
%|FKM(W'\# N,\fraks'\#\frakt, \psi_\ast\xi, \id_{W'}\#f_N^{-1})| < (N-1)|FKM( W\# N , \fraks\#\frakt, \xi, f)|.
%\]
%By \eqref{eq: inducitve FKM gamma}, we obtain that 
%\[
%FKM( W\# N , \fraks\#\frakt, \xi, f^{Nn}, \gamma) \neq 0
%\]
%for any $\gamma$ and $n \in \Z\setminus\{0\}$.
\end{proof}

\section{Proof of Theorem~\ref{main1}}
Before proving Theorem~\ref{main1}, we review several definitions and theorems which are used in the proof of \cref{main1}. 
\subsection{Contact topology}
Let $\xi$ be a contact structure on an oriented 3-manifold. 
 A knot $K \subset (Y,\xi)$ is called \textit{Legendrian} if 
 $T_p K \subset \xi_p$ for $p\in K$. A Legendrian knot $K$ in a contact manifold $(Y,\xi)$ has a standard neighborhood $N$ and a framing $fr_\xi$ given by the contact planes. If $K$ is null-homologous, then $fr_\xi$ relative to the Seifert framing is the \textit{Thurston--Bennequin} number of $K$, which is denoted by $tb(K)$.  If one does $fr_\xi-1$-surgery on $K$ by removing $N$ and gluing back a solid torus so as to effect the desired surgery, then there is a unique way to extend $\xi|_{Y-N}$ over the surgery torus so that it is tight on the surgery torus. The resulting contact manifold is said to be obtained from $(Y,\xi)$ by \textit{Legendrian surgery} on $K$. Also, for a knot $K$ in $(S^3, \xi_{\text{std}})$, the {\it maximal Thurston--Bennequin number} is defined as the maximal value of all Thurston--Bennequin numbers for all Legendrian representations of $K$.

  A \textit{symplectic  cobordism} from the contact manifold $(Y_-,\xi_-)$ to $(Y_+,\xi_+)$ is a compact symplectic manifold $(W,\omega)$ with boundary $-Y_-\cup Y_+$ where $Y_-$ is a \textit{concave} boundary component and $Y_+$ is \textit{convex}, this means that there is a vector field $v$ near $\partial W$ which points transversally inwards at $Y_-$ and transversally outwards at $Y_+$, $\mathcal{L}_v \omega= \omega$ and $\iota_v \om |_{Y_\pm}$ is a contact form of $\xi_\pm $. If $Y_-$ is empty, $(W,\omega)$ is called a {\it symplectic filling}.

We mainly follow a technique to construct symplectic cobordisms called {\it Weinstein handle attachment} \cite{Weinstein91}. One may attach a 1-, or 2-handle to the convex end of a symplectic cobordism to get a new symplectic cobordism with the new convex end described as follows.  For a 1-handle attachment, the convex boundary undergoes, possibly internal, a connected sum. A 2-handle is attached along a Legendrian knot $L$ with framing one less than the contact framing, and the convex boundary undergoes a Legendrian surgery. 

\begin{thm}\label{cob}
Given a contact 3-manifold $(Y,\xi=\ker \theta)$ let $W$ be a part of its symplectization, that is $(W= [0,1]\times Y, \omega= d(e^t\theta))$. Let $L$ be a Legendrian knot in $(Y,\xi)$ where we think of $Y$ as $Y\times \{ 1 \}$. If $W'$ is obtained from $W$ by attaching a 2-handle along $L$ with framing one less than the contact framing, then the upper boundary $(Y', \xi ')$ is still a convex boundary. Moreover, if the 2-handle is attached to a symplectic filling of $(Y,\xi)$ then the resultant manifold would be a strong symplectic filling of $(Y'\xi ')$.
\end{thm}

The theorem for Stein fillings was proven by Eliashberg \cite{yasha91}, for strong fillings by Weinstein \cite{Weinstein91}, and was first stated for weak fillings by Etnyre and Honda \cite{eh02}.

\subsection{Proof of Theorem~\ref{main1}}

\label{sectionProof of Theorem main1}
In this section, we will show the existence of exotic diffeomorphisms of 4-manifolds with boundary. First, we need the following result to guarantee the existence of topological isotopy between diffeomorphisms of 4-manifolds with boundary.

\begin{thm}[Orson--Powell, {\cite[Corollary~C]{OP}}]
\label{top isotopy}
Let $M$ be a smooth, connected, simply connected, compact 4-manifold with boundary of rational homology of $S^3$ or of $S^1 \times S^2$.  Let $f : M \to M$ be a diffeomorphism such that $f|_{\partial M} = Id_{\partial M}$ and $f_* = Id : H_2(M;\Z) \to H_2(M;\Z)$.
Then f is topologically isotopic rel. $\partial M$ to $Id_M : M \to M$, i.e. there is a topological isotopy $F_t : M \to M$ with $F_0 = f$, $F_1 = Id_M$ such that $F_t|_{\partial M} = Id_{\partial M}$ for all $t \in [0,1]$.

\end{thm}

\begin{lem} \label{2-handle}
Every closed oriented connected 3-manifold $Y$ bounds a simply-connected 4-manifold $W$ that can be decomposed as $ W_0 \cup (\text{a }2\text{-handle})$, where $W_0$ has a Stein structure and the 2-handle is attached along an unknot with framing $-1$. %Moreover, $\frakm(W,\fraks,\xi)=0 \in \Z$ for all choices of $(\fraks,\xi)$ and elements in $\Lambda  (W,\fraks,\xi)$. 
\end{lem}
\begin{proof}
%If $Y$ has a simply-connected Stein filling $W_0$, then we can construct $W$ from a  by attaching a 2-handles along an unknot with framing $+1$. This does not change the boundary but notice that the new 2-handle gives a non-trivial 2-sphere with self-intersection $+1$. From vanishing result proven by a similar argument to adjunction inequality for the Kronheimer--Mrowka invariant (see \cite[Proposition 4.2.4]{MR06}, \cite{IMT21}), we conclude that $\frakm(W,\fraks,\xi)=0$ for all choice of $(\fraks,\xi).$ 
For a general $Y$, the third author constructed such a manifold $W$ in \cite[Proof of Theorem 1.1]{mukherjee2020note}.
\end{proof}

We shall now prove \cref{main1} where we show the existence of exotic diffeomorphisms of 4-manifolds with boundary.

\begin{proof}[Proof of Theorem~\ref{main1}]\label{Proof of Theorem main1}
Let $N, \frakt, f$ be as in \cref{section: Construction of diffeomorphisms and non-vanishing results}.
For a 3-manifold $Y$, we consider the associated 4-manifold $W_1= W_0 \cup h_1$ as constructed in Lemma~\ref{2-handle}. If necessary, we can attach a 2-handle $h_2$ along the unknot with framing $+1$ on $W_1$.
Note that this process does not change the upper boundary $Y$.
Let us call this modification $W_1$ as well.
Thus we can write
$W_1=W_0 \cup h_1 \cup h_2$. As a core of $h_2$, we have an embedded 2-sphere $S$ whose self-intersection is non-negative and whose homology class is non-torsion. 
%By blowing-up $W_1$, we may assume that $S$ has zero self-intersection.
%\todo{describe this sphere $S$} 

Attach an Akbulut--Mazur cork $(A,\tau)$, i.e, a pair of algebraically canceling 1- and 2-handle as shown in the Figure~\ref{attach cork} on $W_1$ along $Y$ such that the 2-handle of $A$ linked the unknotted 2-handle $h_2$ algebraically thrice and with $h_1$ algebraically once. Thus we get a manifold $W= W_1\cup A$ and denote the boundary by $Y'=\partial W$. In particular, by applying a cork-twist one can get a manifold $W'= (W-\operatorname{int}(A)) \cup_\tau A$ with boundary $Y'$. Notice that the cork-twist changed the dotted 1-handle with the 0 framed 2-handle in the Figure~\ref{attach cork}. Since, in $W'$, the two 2-handles $h_1,h_2$ are passing over the 1-handle, this process increases their Thurston--Bennequin numbers without changing the smooth framing with respect to the standard contact structure on $S^3$, see Figure~\ref{tb1}.

By construction and the use of \cref{cob}, $W'$ has a Stein structure and $\frakm(W',\fraks',\xi) = 1$, where $\fraks'$ is the $\Spinc$ structure corresponding to the Stein structure and $\xi$ is the induced contact structure on the boundary. 

\begin{figure}[]
	\begin{center}
  \begin{overpic}[scale=0.6, tics=20]{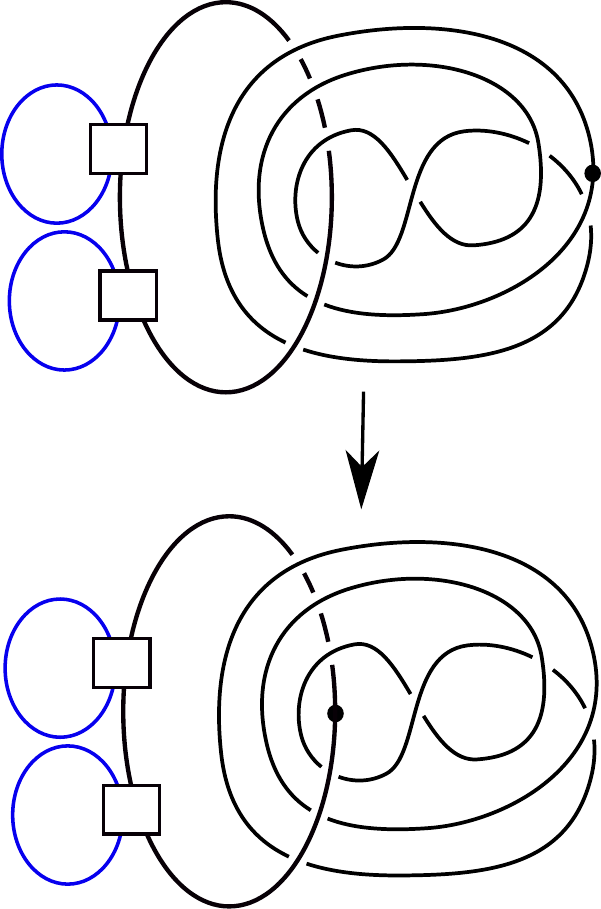}
  \put(-20,220) {$h_1$}
  \put(-20, 170) {$h_2$}
  \put (27,216) {+1}
  \put(30,175) {+3}
  \put(27, 68) {+1}
  \put(30, 25) {+3}
  \put(15,240) {-1}
  \put (15,95){-1}
  \put(20,-5){+1}
  \put(20,145){+1}
  \put(70,140){0}
  \put(160,0){0}
  \put(115,130){Cork twist}
  
\end{overpic}

\caption{Attach the Akbulut--Mazur cork which is linking with with $h_1$ and $h_2$.}

\label{attach cork}
\end{center}
\end{figure}

 \begin{figure}[]
	\begin{center}
  \begin{overpic}[scale=0.6, tics=20]{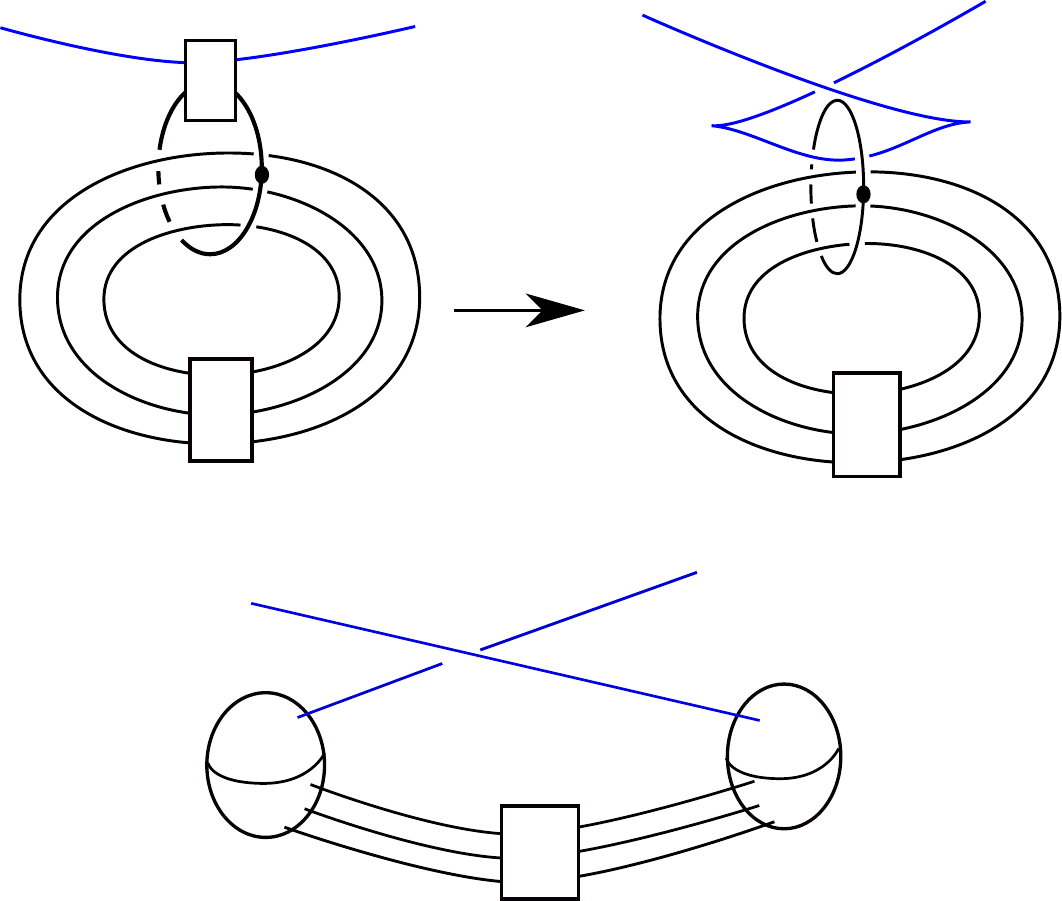}
  \put(55,235){+1}
 \put(135,180){ Isotopy}
 \put(100,-10){Legendrian Representation}
\end{overpic}
\caption{Contact framing of the blue knot is increased by 1 when it passes through the 1-handle. }
\label{tb1}
\end{center}
\end{figure}

Notice that $W$ and $W'$ are related by the cork-twist of $(A,\tau)$ and the cork-twist $\tau$ on $\partial A$ extends over $A \# \CP^2$ \cite{mazur}. This gives a diffeomorphism $W\# \CP^2 \to W'\# \CP^2$ that is the identity map on $W - \operatorname{int}(A)$ and is the extension of the cork-twist to $A \# \CP^2$ on the rest. Thus we get a diffeomorphism $\psi : W\# N \to W'\#N$ 
%where the cork has only one 1-handle and one 2-handle, $W\#N$ and $W'\#N$ are diffeomorphic, in particular one can choose such diffeomorphism so that $\psi': W\#N \to W'\#N$ preserves the homology classes of $H, E_1, E_2$ ().
%Adopt this $\psi'$ 
as the one in \cref{section: Construction of diffeomorphisms and non-vanishing results} (if necessary, we need to precompose $\psi$ with an involution of $\CP^2$ to get the map that acts by identity on homology, a careful proof has been written by Auckly--Kim--Melvin--Ruberman \cite{AKHMR15}), and construct a self-diffeomorphism $f : W\#N \to W\# N$ along the procedure in \cref{section: Construction of diffeomorphisms and non-vanishing results}.
We claim that this diffeomorphism $f$ is the desired diffeomorphism.

Adopt the canonical Spin$^c$ structure on $W$ as $\fraks$ in \cref{section: Construction of diffeomorphisms and non-vanishing results}.
As noted above, $W$ contains an embedded 2-sphere whose self-intersection is non-negative and whose homology class is non-torsion.
Moreover, we have that $\frakm(W',\fraks',\xi)=1$.
Thus it follows from \cref{comp+gluing with gamma} that
\[
FKM( W\# N , \fraks\#\frakt, \xi, f^{2n}, \gamma^2)
\neq FKM( W\# N , \fraks\#\frakt, \xi, f^{2n'}, \gamma^2) \in \Z
\]
for every $\gamma \in \pi_1(\Xi^{\mathrm{cont}}(Y))$ and every distinct $n, n' \in \Z$.
Therefore, by \cref{lem: ab diff: vanishing for some gamma general}, $f^n$ and $f^{n'}$ are not smoothly isotopic to each other through $\Diff(W)$.
On the other hand, it follows from \cref{top isotopy} that all $f^n$ are topologically isotopic to the identity through $\Homeo(W, \del)$.
This completes the proof.
\end{proof}

\begin{proof}[Proof of \cref{thm: emb sphere}]
First, note that all $f^n$ are mutually topologically isotopic as in the proof of Theorem~\ref{main1} above.
Thus all $f^n(S)$ are mutually topologically isotopic.
To show that $f^n(S)$ and $f^{n'}(S)$ are not smoothly isotopic if $n\neq n'$, it suffices to show that $f^{n-n'}(S)$ is not smoothly isotopic to $S$.
This follows from \cref{cor: Ab Z summand} combined with \cref{KM family adjunction}, which we shall prove in \cref{sectionSeveral vanishing results}.
\end{proof}

\begin{rem}
\label{rem: direct sum main}
In the setup of Theorem~\ref{main1},
the mapping class of $f$ in $\Diff(W,\del)$ generates a direct summand isomorphic to $\Z$ in the abelianization of the kernel of 
\[
\pi_0 (\operatorname{Diff} (W,  \partial )  )\to \pi_0(\operatorname{Homeo}(W, \partial)).
\]
This is a direct consequence of \cref{cor: Ab Z summand}.
\end{rem}

\begin{rem}\label{small b_2}
We will construct an explicit example of "small" 4-manifolds with boundary that admits exotic diffeomorphism by following the strategy of the Proof of \cref{main1}. We start with a 4-manifold $W$ which is obtained by attaching a 2-handle $h$ on $B^4$ along an unknot with framing $+1$. Now we will attach a pair of canceling 1- and 2-handle such that the 2-handle of the canceling pair is linked positively with $h$ algebraically thrice, let this be called $W$ and $X= W\# N$ where $b_2(X)= 4$. Now if we apply the cork-twist on $W$, that process will increase the maximum Thurston--Bennequin number of $h$ by 3 and thus the resultant manifold $W'$ will have a Stein structure. Note that $W\#N$ is diffeomorphic to $W'\#N$, and thus by the previous proof, we can construct an exotic diffeomorphism on $X$.  
\end{rem}

%We prove the following gluing formula: 
%\begin{lem}
%For a suitable choice of fiberwise metrics and regular perturbations, there is a diffeomorphism from $
%\mathcal{M}(W \# N, \id \# f )$ to $
%\mathcal{M}(W)$.
%\end{lem}

%\section{Proof of existence of exotic surfaces}
\section{Exotic embeddings of 3-manifolds in 4-manifolds}
\label{section: Exotic embeddings of 3-manifolds in 4-manifolds}

Before introducing the results on exotic embeddings, we first review a relation between generalized Smale conjecture and exotic embeddings.

\subsection{Generalized Smale conjecture and exotic embeddings}

Let $Y$ be a closed 3-manifold with the Riemann metric $g$ whose sectional curvature is $\pm 1$.
The generalized Smale conjecture says that the inclusion $\operatorname{Isom}(Y, g )\to \operatorname{Diff}(Y)$ is a homotopy equivalence, where $\operatorname{Isom}(Y, g )$ denotes the group of isometries on $(Y, g )$.
Some examples of $(Y, g )$ which satisfy this property are known and some examples of $(Y, g )$ which do not satisfy it are known. For example, the hyperbolic 3-manifolds satisfy the generalized Smale conjecture \cite{Ga01}. %and the manifold $S^1\times S^2$ does not satisfy the conjecture (See \cite{Ha81}). 
We use some results related to the generalized Smale conjecture. 
However, we do not need to restrict ourselves to a Riemannian  metric with sectional curvature $\pm 1$ for our purpose. 
Thus, we consider a closed Riemannian 3-manifold more generally in this paper.

\begin{defn}
We say a Riemannian 3-manifold $(Y, g )$ is {\it admissible}, if the cokernel of the induced map  on 
\[
\pi_0 (\operatorname{Isom}^+(Y, g ) )  \to \pi_0 (\operatorname{Diff}^+(Y))
\]
is trivial, where $\operatorname{Isom}^+(Y, g )$ and $\operatorname{Diff}^+(Y)$ are the groups of orientation preserving isometries and diffeomorphisms of $(Y, g)$ and $Y$ respectively. 
\end{defn}
Our techniques to detect exotic embeddings can be applied for a Riemannian 3-manifold $(Y, g )$ with finite cokernel of $\pi_0 (\operatorname{Isom}^+(Y, g ) )  \to \pi_0 (\operatorname{Diff}^+(Y))$. 

We use the following lemma for embeddings of $Y$ into a 4-manifold $X$. 
\begin{lem}\label{isotopy admissible}
Suppose $(Y, g )$ is an admissible 3-manifold.  
Let $i : Y \to X$ be a smooth embedding and $f$ an orientation-preserving self-diffeomorphism of $X$ satisfying that $f(i(Y) ) = i(Y)$.
Then, for every $n \in \Z$, we can deform $f^n$ by a smooth isotopy of diffeomorphisms of $X$ fixing $i(Y)$ setwise so that 
\[f^n|_{i(Y)} \in  \operatorname{Isom}(i(Y), i_*g ) .
\]
\end{lem}

\begin{proof}
 We regard $f$ as a self-diffeomorphism of $i(Y)$. Note that $f|_{i(Y)}$ is orientation preserving. 
Since by the assumption on $(Y, g)$, the map $\pi_0 (\operatorname{Isom}^+(i(Y), i_*g ) )  \to \pi_0 (\operatorname{Diff}^+(i(Y)))$ is sujective, and so the diffeomorphism $f^n$ lies in the image of the map $\pi_0 (\operatorname{Isom}^+(i(Y), i_*g ) )  \to \pi_0 (\operatorname{Diff}^+(i(Y)))$. By isotopy extension lemma, we complete the proof. 
\end{proof}

\begin{lem}\label{admissiblity examle}
Let $(Y, g)$ be one of the following 3-manifolds: 
\begin{itemize}
\item[(i)] elliptic 3-manifolds and 
    \item[(ii)] hyperbolic 3-manifolds. 
\end{itemize}
Then $(Y, g)$ is admissible. 
\end{lem}

\begin{proof}
For elliptic 3-manifolds, the $\pi_0$-Smale conjecture is proven by combining several works, for more details see \cite[Theorem 1.2.1]{SJDH12}. Also, the Smale conjecture for hyperbolic 3-manifolds is solved by Gabai \cite{Ga01}. 
\begin{comment}
For $S^1\times S^2$ with the product metric $g$ of the standard metrics, it is known that the natural inclusion 
\[
 O(2) \times O(3) \to \operatorname{Isom}(S^1\times S^2, g ) 
\]
is a homotopy equivalence. Also, it is proven in \cite{} that the natural inclusion 
\[
O(2) \times O(3) \times \Omega (O(3)) \to \operatorname{Diff}(S^1\times S^2 ) 
\]
is also a homotopy equivalence. So, from these calculations, we conclude 
\[
\operatorname{Cok} (\pi_0 (\operatorname{Isom}^+(Y, g ) )  \to \pi_0 (\operatorname{Diff}^+(Y)) ) \cong \pi_0 (\Omega (O(3))) \cong \Z_2
\]
which gives the conclusion. 
\end{comment}
\end{proof}

In order to prove several vanishing results, we also need the following property: 
\begin{lem}\label{strong L-space}
Let $(Y, g)$ be one of the following geometric 3-manifolds: 
\begin{itemize}
\item[(i)] 3-manifolds having positive scalar curvature metric and 
    \item[(ii)] the hyperbolic three-manifolds lebeled by 
\[
0, 2,3,8 , 12, \ldots, 16, 22, 25, 28 ,\ldots, 33, 39, 40, 42, 44, 46, 49
\]
in the Hodgson-Weeks census which correspond to 3-manifolds in \cite[Table 1]{FM21}. 
\end{itemize}
Then, for any Spin$^c$-structure $\mathfrak{t}$ on $Y$, there is no irreducible Seiberg--Witten solution to $(Y, \mathfrak{t}, g)$, i.e. $Y$ is a strong L-space. 
\end{lem}

\begin{proof}
Since 3-manifolds listed in (i) have positive scalar curvature, by the Weitzenb\"{o}ck formula, we see that there is no irreducible Seiberg--Witten solution to $(Y, \mathfrak{t}, g)$. 
For hyperbolic 3-manifolds in (iv), \cite[Theorem 1.1]{FM21} implies the conclusion. 
\end{proof}

\subsection{Results on exotic embeddings}
In this section, we prove \cref{strongly exotic}.

Let $Y$ be a closed, oriented, connected 3-manifold, and let $X$ be a smooth 4-manifold possibly with boundary. 
We first construct a 4-manifold for a given 3-manifold. 
\begin{lem}\label{emb of 3-manifold}
Given a closed, connected, oriented 3-manifold $Y$, there exists a closed simply-connected 4-manifold $X$ such that $X= X_1\cup_{Y} X_2$ with $b^+_2(X_i)>1$ for $i=1,2$.
Moreover, we can construct $X$ in such a way that there exists a diffeomorphism $f : X \to X$ which is topologically isotopic to the identity but $FSW(X,\mathfrak{s}, f^n)\neq 0$ for every $n\in \Z \setminus \{0\}$ and for some Spin$^c$ structure $\mathfrak{s}$ on $X$, and thus not smoothly isotopic to the identity. 
\end{lem}

\begin{proof}
%Given a 3-manifold $Y$, we proved in Lemma 6.1 that $Y$ bounds a simply-connected 4-manifold $W= W_0 \cup 2-handle$, where $W_0$ has a Stein structure. If necessary, we can add a $+1$ framed 2-handle on $W$ along an unknot in $Y$. Notice that, this process does not change the boundary 3-manifold $Y$.
We will follow the strategy of Proof of the Theorem~\ref{main1}, where we showed that given a 3-manifold $Y$ we can construct a compact simply connected 4-manifold $W\# N$ with boundary $Y'$ (where there is a ribbon homology cobordism from $Y$ to $Y'$) and there exists a self-diffeomorphism $f: W\# N \to W\# N$ which is topologically isotopic to the identity rel to the boundary. By construction, $Y$ is smoothly embedded in $W$ and $Y$ bounds a submanifold $W_1$ with $b_2^+(W_1)>1.$ Now by attaching a simply connected symplectic cap on $(Y',\xi')$ with $b^+_2 >1$ (existence of such caps are shown in \cite{EMM20}) we can get our desired simply-connected 4-manifold $X$.
Let $\mathfrak{s}$ be the Spin$^c$ structure on $X$ obtained as the connected sum of the canonical Spin$^c$ structure and the Spin$^c$ structure $\mathfrak{t}$ on $N$ considered in \cref{section: Construction of diffeomorphisms and non-vanishing results}.
Also, we can extend $f$ on the symplectic cap as the identity and get our desired diffeomorphism $f$.
Now $FSW(X,\mathfrak{s},f^n)\neq 0$ for $n\neq0$ follows from \cite[Corollary~9.6, Proof of Theorem~9.7]{BK20}.
\end{proof}

\begin{proof}[Proof of \cref{strongly exotic}]
Let $Y$ be a hyperbolic 3-manifold listed in Theorem~\ref{strongly exotic} and let $X$ and $f$ be a corresponding 4-manifold and diffeomorphism constructed in \cref{emb of 3-manifold}.
For $n \in \Z$,
define smooth embeddings $i_n : Y \to X$ by $i_n(y)=f^n(y)$ for $y \in Y$.
Since $f:X\to X$ is topologically isotopic to the identity, we see that all the embeddings $i_n$'s are topologically isotopic to each other. It is now enough to prove that the image of $i_0$ is not smoothly isotopic to the image of $i_n$ for every $n\neq0$. If the image of $i_0$ and the image of $i_n$ are smoothly isotopic, then we can further deform $f^n$ by smooth isotopy so that $f^n(Y)= i_n(Y)=Y$. Moreover, using \cref{isotopy admissible}, we may also assume $f^{-1}_n|_{Y}$ is isometry with respect to the metric $g$ considered in \cref{strong L-space}. Moreover, there is no irreducible solution to the Seiberg--Witten equation with respect to $(Y, \mathfrak{t}, g)$ from \cref{strong L-space}. 
But then the vanishing theorem Theorem~\ref{conn sum vanish family} implies that $FSW(X,f^n)=0$ which is a contradiction. For a connected sum of elliptic 3-manifolds, we use \cref{conn sum vanish family:PSC} instead of \cref{conn sum vanish family}. 
\end{proof}
\begin{rem}
The proof of \cref{strongly exotic} is a constructive proof. One thing that we are not sure about is how to control the second Betti number of $X$. So one may ask the following question.
\end{rem}
\begin{ques}
Given an oriented, connected 3-manifold $Y$, what would be the small second Betti number for $X$ such that $Y$ admits an exotic embedding in $X$.
\end{ques}

More generally, we can find exotic embeddings when the family gauge theoretic invariants do not vanish: 
\begin{thm}
\label{exotic 3-manifolds non-vanishing}
Let $(X, \s)$ be a compact simply-connected Spin$^c$ 4-manifold with or without boundary.
If $\del X \neq \emptyset$, we equip $\del X$ a contact structure $\xi$.
Let $f:X\to X$ be a self-diffeomorphism which is the identity on $\del$ if $\del X \neq \emptyset$.
Suppose that $FSW(X,\s,f)\neq 0$ if $\del X=\emptyset$, and that $FKM(X,\fraks,\xi,f)\neq0$ if $\del X \neq \emptyset$.
Let $Y$ be one of the following 3-manifolds: 
\begin{itemize}
\item[(i)] the connected sum of elliptic 3-manifolds, and 
    \item[(ii)] the hyperbolic three-manifolds labelled by 
\[
0, 2,3,8 , 12, \ldots, 16, 22, 25, 28 ,\ldots, 33, 39, 40, 42, 44, 46, 49
\]
in the Hodgson-Weeks census which correspond to 3-manifolds in \cite[Table 1]{FM21}. 
\end{itemize}
\begin{enumerate}
    \item
If $\partial X = \emptyset $ and $X$ has the decomposition $X= X_1 \cup_Y  X_2$ such that $b_2^+(X_i)>1$ for $i=1,2$, where $X_1$ and $X_2$ are compact 4-manifold with boundary $Y$ and $-Y$ respectively.  
Then there exist infinitely many embedded 3-manifolds $\{f^n(Y)\}_{n \in \Z}$ that are mutually not smoothly isotopic. 
\item If $\partial X \neq \emptyset $ and $X$ has the decomposition $X= X_1 \cup_Y  X_2$ such that $b_2^+(X_2)>1$, where $X_1$ and $X_2$ are compact 4-manifold with boundary $(\partial X) \amalg Y$ and $-Y$ respectively.  
Then there exist infinitely many embedded 3-manifolds $\{f^n(Y)\}_{n \in \Z}$ that are mutually not smoothly isotopic. 
\end{enumerate}
\end{thm}
Note that \cref{exotic 3-sphere} follows from \cref{exotic 3-manifolds non-vanishing}. 
\begin{proof}
When $\partial X= \emptyset$, the result follows from the proof of \cref{strongly exotic}. When $\partial X \neq \emptyset$, we just use the vanishing result \cref{conn sum vanish family1} on the families Kronheimer--Mrokwa's invariant instead. 
\end{proof}

\begin{rem}
Notice that, given any closed, connected 3-manifold $Y$, we can always find a closed simply-connected 4-manifold $X$ where $Y$ is smoothly embedded and a self-diffeomorphism $f:X\to X$ which is topologically isotopic but not smoothly. So it is very natural to think that the set $\{f^n(Y)\}_{n\in\Z}$ contains all exotically embedded pairs of $Y$, i.e. topologically isotopic as a pair but not smoothly. However, we cannot conclude that at this point because our vanishing result doesn't hold for all 3-manifolds. So we can ask the following question:
\end{rem}

\begin{ques}
Given a closed, connected 3-manifold $Y$ how does one construct a closed 4-manifold $X$ such that there exists a pair of smooth embeddings $i_1,i_2 : Y\to X$ that are topologically isotopic but not smoothly? 
\end{ques}

%Finally, we conisider exotic embeddings of $S^1\times S^2$. 

\appendix

\section{Excision for determinant line bundles} 
\label{section: Excision for determinant line bundles} 

In this section, we explain the excision principle that is used to give signs to variants of Kronheimer--Mrowka's invariant for 4-manifolds with contact boundary. 
This argument is well-known for experts and essentially done in Appendix B of \cite{Char04}.

For $i=1, 2$
let $X_i$ be a Riemannian 4-manifold and $A_i, B_i$ be  codimension 0 submanifold of $X_i$.
Here we assume $X_1$ and  $X_2$ are closed for simplicity, but this assumption is not essential.
For example, we can apply similar arguments to manifolds with conical ends under suitable Sobolev completion. Assume 
$A_i \cap B_i \subset X_i$ is a compact codimension-0 submanifold and also 
an isometry between 
$A_1 \cap B_1$ and $A_2 \cap B_2$ is fixed.
We will identify them by this isometry.
For $i=1, 2$, 
suppose we are given vector bundles
$E_i, F_i$ on $X_i$ and 
 elliptic differential operators of order $l \in \Z_{\geq 1 }$
\[
D_i: \Gamma(X_i; E_i)\to \Gamma(X_i; F_i)
\]
which are identical on $A_1\cap B_1 =A_2\cap B_2$.

Using the identification given above, we form Riemannian 4-manifolds 
$\wt{X}_1=A_1 \cup B_2$,  $\wt{X}_2=A_2 \cup B_1$
and
vector bundles 
 $\wt{E}_1=E_1|_{A_1}\cup E_2|_{B_2}$, 
 $\wt{E}_2=E_2|_{A_1}\cup E_1|_{B_1}$, 
 $\wt{F}_1=F_1|_{A_1}\cup F_2|_{B_2}$
 $\wt{F}_2=F_2|_{A_1}\cup F_1|_{B_1}$.

Define elliptic operator 
\[
\wt{D}_1=
\begin{cases}
D_1 & \text{on } A_1 \\
D_2 & \text{on } B_2
\end{cases}
\]
\[
\wt{D}_2=
\begin{cases}
D_2 & \text{on } A_2 \\
D_1 & \text{on } B_1
\end{cases}
\]
$D_1, D_2, \wt{D}_1, \wt{D}_2$ define Fredholm operators under Sobolev completions
\[
L^2_{k+l}\to L^2_k
\]
for $k \in \R$.
In general, 
for a Fredholm operator $D$, 
we will define the 1-dimensional real vector space $\det D$ by
\[
\det D=\Lambda^{\max}\ker D \otimes \Lambda^{\max}\cok D^*.
\]
\begin{thm}\label{excision}
We can associate a linear isomorphism of 1-dimensional real vector space
\begin{align}\label{excision isom}
\det D_1 \otimes \det D_2\to 
\det \wt{D}_1 \otimes \det \wt{D}_2, 
\end{align}
which is independent of data used in the construction up to homotopy.
\end{thm}
As remarked after the proof, this can be easily adapted to the case with conical ends as considered in this paper.
Note that in order to fix a sign of the unparametrized Kronheimer--Mrowka invariant, considering families of operators, since it is enough to give an orientation of one fiber of the determinant line bundle.
\par
\begin{proof}
Choose square roots of partition of unity
\[
\phi^2_1+\psi^2_1=1
\]
\[
\phi^2_2+\psi^2_2=1
\]
subordinate to
$(A_1, B_1)$ and $(A_2, B_2)$ such that 
\[
\phi_1=\phi_2 \text{ and } \psi_1=\psi_2
\]
on $A_1 \cap B_1 =A_2 \cap B_2$.
Define
\[
\Phi: \Gamma(X_1; E_1)\oplus \Gamma(X_2; E_2)
\to \Gamma(\wt{X}_1; \wt{E}_1)\oplus \Gamma(\wt{X}_2; 
\wt{E}_2)
\]
and
\[
\Psi:\Gamma(\wt{X}_1; \wt{F}_1)\oplus \Gamma(\wt{X}_2; 
\wt{F}_2) 
\to 
\Gamma(X_1; F_1)\oplus \Gamma(X_2; F_2)
\]
by 
\[
\Phi=
\begin{bmatrix}
\phi_1 & \psi_2 \\
-\psi_1 & \phi_2
\end{bmatrix}
\]
and 
\[
\Psi=
\begin{bmatrix}
\phi_1 & -\psi_1 \\
\psi_2 & \phi_2.
\end{bmatrix}
\]
Using the fact that  
$\psi_1 \phi_1=\psi_2 \phi_2$ 
holds on the whole manifold, we can see $\Phi$ and $\Psi$ are inverse of each other (See \cite{Char04}).
Set $D=D_1\oplus D_2$ and $\wt{D}=\wt{D}_1\oplus \wt{D}_2$.
Then we have
\[
\det (D)=\det(D_1)\otimes \det(D_2)
\]
and 
\[
\det (\wt{D})=\det(\wt{D}_1)\otimes \det(\wt{D}_2).
\]
We have
\begin{align*}
&\Psi \wt{D}\Phi \\
&=\begin{bmatrix}
\phi_1 & -\psi_1 \\
\psi_2 & \phi_2
\end{bmatrix}
\begin{bmatrix}
\wt{D}_1 &  \\
 & \wt{D}_2
\end{bmatrix}
\begin{bmatrix}
\phi_1 & \psi_2 \\
-\psi_1 & \phi_2
\end{bmatrix}
\\
%&=\begin{bmatrix}
%\phi_1 & -\psi_1 \\
%\psi_2 & \phi_2
%\end{bmatrix}
%\begin{bmatrix}
%\wt{D}_1\circ \phi_1 & \wt{D}_1\circ\psi_2 \\
%-\wt{D}_2\circ\psi_1 & \wt{D}_2\circ\phi_2
%\end{bmatrix}
%\\
%&=\begin{bmatrix}
%\phi_1\circ\wt{D}_1\circ\phi_1+\psi_1\circ\wt{D}_2\circ\psi_1 &\phi_1\circ\wt{D}_1\circ\psi_2 -\psi_1\circ\wt{D}_2\circ \phi_2\\
%\psi_2 \circ\wt{D}_1\circ \phi_1-\phi_2\circ\wt{D}_2\circ\psi_1 & \psi_2\circ \wt{D}_1\circ \psi_2+\phi_2\circ\wt{D}_2\circ\phi_2
%\end{bmatrix}
%\\
%&=\begin{bmatrix}
%\phi_1\circ{D}_1\circ\phi_1+\psi_1\circ{D}_1\circ\psi_1 &\phi_1\circ{D}_2\circ\psi_2 -\psi_1\circ{D}_2\circ \phi_2\\
%\psi_2 \circ{D}_1 \circ\phi_1-\phi_2\circ {D}_1\circ\psi_1 & \psi_2\circ{D}_2\circ \psi_2+\phi_2\circ{D}_2\circ\phi_2.
%\end{bmatrix}
%\\ 
%&=\begin{bmatrix}
%\phi_1\circ{D}_1\circ\phi_1+\psi_1\circ{D}_1\circ\psi_1 &\phi_1\circ([{D}_2, \psi_2]+\psi_2 D_2) -\psi_1\circ([{D}_2, \phi_2]+\phi_2D_2)\\
%\psi_2 \circ([{D}_1 , \phi_1]+\phi_1 D_1)-\phi_2\circ ([{D}_1, \psi_1 ]+\psi_1 D_1)& \psi_2\circ{D}_2\circ \psi_2+\phi_2\circ{D}_2\circ\phi_2
%\end{bmatrix}
%\\ 
&=\begin{bmatrix}
\phi_1\circ{D}_1\circ\phi_1+\psi_1\circ{D}_1\circ\psi_1 &\phi_1[{D}_2, \psi_2] -\psi_1 [{D}_2, \phi_2]\\
\psi_2[{D}_1 , \phi_1]-\phi_2[{D}_1, \psi_1 ]& \psi_2\circ{D}_2\circ \psi_2+\phi_2\circ{D}_2\circ\phi_2
\end{bmatrix}. 
\end{align*}
Here, we used obvious relations
\[
\wt{D}_1\circ \phi_1=D_1\circ \phi_1, \quad
\wt{D}_2\circ \phi_2= D_2\circ \phi_2
\]
\[
\wt{D}_2\circ \psi_1=D_1\circ \psi_1,\quad 
 \wt{D}_1\circ \psi_2= D_2\circ \psi_2.
\]
and
\[
\psi_1\phi_1=\phi_2 \psi_1\quad \phi_1\psi_2= \psi_1\phi_2.
\]
On the other hand, we have
\[
D_i=D_i\circ (\phi^2_i+\psi^2_i)=([D_i, \phi_i]+\phi_i \circ D_i)\circ \phi_i+
([D_i, \psi_i]+\psi_i \circ D_i)\circ \psi_i
\]
for $i=1, 2$.
Thus 
\[
K:=\Psi \wt{D}\Phi-D: L^2_{k+l}(X)\to L^2_k(X)
\]
is calculated as
\begin{align*}
K%&=
%\begin{bmatrix}
%\phi_1{D}_1\circ\phi_1+\psi_1{D}_1\circ\psi_1 &\phi_1[{D}_2, \psi_2] -\psi_1 [{D}_2, \phi_2]\\
%\psi_2[{D}_1 , \phi_1]-\phi_2[{D}_1, \psi_1 ]& %\psi_2{D}_2\circ \psi_2+\phi_2{D}_2\circ\phi_2
%\end{bmatrix}-
%\\
%&\begin{bmatrix}
%\phi_1{D}_1\circ\phi_1+\psi_1{D}_1\circ\psi_1 +[D_1, \phi_1]\circ \phi_1+[D_1, \psi_1]\circ \psi_1&\\
%& \psi_2{D}_2\circ \psi_2+\phi_2{D}_2\circ\phi_2+[D_2, \phi_2]\circ \phi_2+[D_2, \psi_2]\circ \psi_2
%\end{bmatrix}
%\\
%&
=\begin{bmatrix}
-[D_1, \phi_1]\circ \phi_1-[D_1, \psi_1]\circ \psi_1 &\phi_1[{D}_2, \psi_2] -\psi_1 [{D}_2, \phi_2]\\
\psi_2[{D}_1 , \phi_1]-\phi_2[{D}_1, \psi_1 ]& 
-[D_2, \phi_2]\circ \phi_2-[D_2, \psi_2]\circ \psi_2.
\end{bmatrix}
\end{align*}
The order of this operator is strictly smaller than $l$, so
\[
K: L^2_{k+l}(X)\to L^2_k(X)
 \]
 is a compact operator.
 Thus the family of operators
\[
\{D_t=D+t K\}_{t \in [0, 1]}
\]
gives a desired isomorphism between 
$\det (D)$ and $\det(\wt{D}) $.
\end{proof}
Note that each entry of $K$ is supported on $(A_1\cap B_1) \times (A_2\cap B_2)$, so 
the same conclusion holds even if $X_1$ and $X_2$ have conical ends as considered in this paper, as long as suitable Sobolev completions are used and $A_1\cap B_1$, $A_2\cap B_2$ are relatively compact.

\section{Blow up formula for Kronheimer--Mrowka's invariant}

\begin{defn} Fix an element of $\Lambda(W, \s, \xi)$. 
For a pair $(W, \xi )$ and a fixed reference Spin$^c$ structure $\s_0 \in \Spinc(W, \xi)$, we define two functions
\[
KM(W, \s_0, \xi ) := \sum_{e \in H^2 (W, \partial W; \Z) } \mathfrak{m}(W,\s_0+e, \xi)  \exp(\langle 2e, -\rangle )  :H_2(W, \partial W; \R)\to \R. 
\]
and
\[
\wt{KM}(W,\xi)(\nu):=e^{\frac{i}{2\pi}\int_{\nu}F_{A^t_0}}KM(W, \s_0, \xi)(\nu):H_2(W, \partial W; \R)\to \R. 
\]
Here, 
$A_0$ is a $\spinc$ connection for $\s_0$ extending the canonical $\spinc$ connection on the conical end.
\end{defn}
Note that $KM(W, \s_0, \xi )$ depends on the fixed $\Spinc$ structure $\s_0$. On the other hand, we can see the following: 
\begin{lem} 
The function $\wt{KM}(W,\xi)$ does not depends on a fixed $\Spinc$ structure $\s_0$.
\end{lem} 
\begin{proof}
Indeed, for $l\in H^2(X; \Z)$, $\s'_0:=\s_0+l$, we have
\[
2l=c_1(\s'_0)-c_1(\s_0)=\frac{i}{2\pi}\int_\nu (F_{A'^t_0}-F_{A^t_0}),
\]
 so
\begin{align*}
&e^{\frac{i}{2\pi}\int_{\nu}F_{A'^t_0}}KM(W, \s'_0, \xi)(\nu)
\\
&=e^{\frac{i}{2\pi}\int_{\nu}F_{A'^t_0}}\sum_{e' \in H^2 (W, \partial W; \Z) } \mathfrak{m}(W,\s'_0+e', \xi)  \exp(\langle 2e', \nu\rangle ) 
\\
&=e^{\frac{i}{2\pi}\int_{\nu}F_{A'^t_0}}\sum_{e' \in H^2 (W, \partial W; \Z) } \mathfrak{m}(W,\s_0+l+e', \xi, o)  \exp(\langle 2e', \nu\rangle )
\\
&=e^{\frac{i}{2\pi}\int_{\nu}F_{A'^t_0}}\sum_{e \in H^2 (W, \partial W; \Z) } \mathfrak{m}(W,\s_0+e, \xi, o)  \exp(\langle 2(e-l), \nu\rangle )
\\
&=e^{\frac{i}{2\pi}\int_{\nu}F_{A'^t_0}}\sum_{e \in H^2 (W, \partial W; \Z) } \mathfrak{m}(W,\s_0+e, \xi, o)  \exp(\langle 2e, \nu\rangle )\exp(-\frac{i}{2\pi}\int_\nu F_{A'^t_0}-F_{A^t_0})
\\
&=e^{\frac{i}{2\pi}\int_{\nu}F_{A^t_0}}\sum_{e \in H^2 (W, \partial W; \Z) } \mathfrak{m}(W,\s_0+e, \xi, o)  \exp(\langle 2e, \nu\rangle )
\\
&=e^{\frac{i}{2\pi}\int_{\nu}F_{A^t_0}}\sum_{e \in H^2 (W, \partial W; \Z)} KM(W, \s_0, \xi)(\nu)
\end{align*}
Here we changed the variables by
\[
e=l+e'.
\]
\end{proof}

We establish the blow-up formula for Kronheimer--Mrowka's invariant for 4-manifold with boundary, using the pairing formula and the formal $(3+1)$-TQFT property of the monopole Floer homology. As a partial result, \cite[Theorem 1.3]{Iida21} combined with a computation of the Bauer--Furuta invariant for $-\C P^2$ gives the blow up formula for Kronheimer--Mrowka's invariant for 4-manifold with boundary under the condition $b_3=0$. In this section, we remove the condition $b_3=0$ by following the discussion given in Section 39.3 of \cite{KM07}. 
\begin{thm}\label{nonfamily KM blowup}
Let $X$ be a compact oriented 4-manifold equipped with a contact structure $\xi$ on the boundary $Y=\partial X$.
Denote by the blow up of $X$ at an interior point by
\[
\hat{X}=X \# (-\C P^2)
\]
and the exceptional sphere by $E$.
Then for $\nu \in H_2(X, \partial X; \R)$ and $\lambda \in \R$,  
\[
\wt{KM}(\hat{X}, \xi)(\nu+\lambda E)=2\cosh (\lambda)\wt{KM}(X, \xi)(\nu)
\]
holds.
\end{thm}
\begin{proof}
Let us denote by
\[
\overset{\circ}{X}: S^3\to Y
\]
and
\[
\overset{\circ}{\hat{X}}: S^3\to Y
\]
the cobordism obtained by removing small open disk from  $X$ and $\hat{X}$ respectively and let
\[
N: S^3\to S^3
\]
be the cobordism obtained by removing two small open disks from $-\C P^2$.
By the pairing formula  which is proves in \cite{IKMT} and the composition law Proposition 26.1.2 of \cite{KM07}, we have
\begin{align*}
\wt{KM}(\hat{X}, \xi)(h+\lambda E)&=
\langle \HMf (\overset{\circ}{\hat{X}}, \nu+\lambda E)(1),  \Psi_{\partial \nu }(\xi)\rangle \\
& =\langle  \HMf (N, \lambda E)\circ \HMf (\overset{\circ}{X}, \nu)(1),  \Psi_{\partial \nu }(\xi)\rangle, 
\end{align*}
where we denote the local system $\Gamma_\nu$ just by $\nu$.
Now, as explained in the proof of Theorem 39.3.2 of \cite{KM07}, we have
\[
\HMf (N, \lambda E)=\sum_{m \in \Z} U^{m(m+1)/2} e^{-(2m+1)\lambda},
\]
which can in fact be expressed using the Jacobi eta function.
Kronheimer--Mrowka's invariant is defined to be zero when formal dimension is non-zero, so all terms including higher powers of $U$ disappear.
Thus we have
\begin{align*}
\wt{KM}(\hat{X}, \xi)(\nu+\lambda E)&=(e^\lambda+e^{-\lambda})
\langle  \HMf (\overset{\circ}{X}, \nu)(1),  \Psi_{\partial \nu }(\xi)\rangle \\ 
&=2\cosh(\lambda) \wt{KM}(X, \xi)(\nu).
\end{align*}
\end{proof}
In particular, if Kronheimer--Mrowka's invariant is non-trivial for some element of $\Spinc(X, \xi)$, then 
Kronheimer--Mrowka's invariant is non-trivial for some element of $\Spinc(\hat{X}, \xi)$.

\section{Conflict of interest}
On behalf of all authors, the corresponding author states that there is no conflict of interest.

\bibliographystyle{plain}
\bibliography{tex}

\end{document}